%% file: main.tex
\documentclass{article}

\input{inputs/packagesAndDocumentConfiguration}

\input{inputs/macros}

\usepackage{authblk}

\title{Degrees in the $\beta$- and $\beta^\prime$-Delaunay graphs}

\author[1,2]{Gilles Bonnet}
\author[1]{Joseph Gordon}

\affil[1]{Bernoulli Institute, University of Groningen}
\affil[2]{CogniGron, University of Groningen}

\newcommand{\email}[1]{\texttt{\href{mailto:#1}{#1}}}
\affil[ ]{\email{g.f.y.bonnet@rug.nl}, \email{joseph.gordon.math@gmail.com}}

\date{\today}  

\begin{document}

\maketitle

\begin{abstract}
We investigate the typical cells $\extypcell$ and $\extypcellprime$ of $\beta$- and $\beta'$-Voronoi tessellations in $\R^d$, establishing a Complementary Theorem which entails: 1) a gamma distribution of the $\Phi$-content (a suitable homogeneous functional) of the typical cell with $n$-facets; 2) the independence of this $\Phi$-content with the shape of the cell; 3) a practical integral representation of the distribution of $\typcellany$. We exploit the latter to derive bounds on the distribution of the facet numbers. 
Using duality, we get bounds on the typical degree distributions of $\beta$- and $\beta'$-Delaunay triangulations. For $\beta'$-Delaunay, the resulting exponential lower bound seems to be the first of its kind for random spatial graphs arising as the skeletons of random tessellations. For $\beta$-Delaunay, matching super-exponential bounds allow us to show concentration of the maximal degree in a growing window to only a finite number of deterministic values (in particular, only two values for $d=2$).
\end{abstract}


\noindent \textbf{MSC Classes:} 
Primary: 
52A22,
60D05,
60F05;
Secondary:
52B11,
60G55.

\input{inputs/content}

\end{document}

%% file: inputs/packagesAndDocumentConfiguration.tex
\usepackage[utf8]{inputenc}
\usepackage{amssymb, amsmath, amsthm, amsxtra, parskip, hyperref, mathrsfs}
\usepackage[dvipsnames]{xcolor}
\usepackage[capitalise]{cleveref}
\usepackage[bb=dsserif]{mathalpha}
\usepackage[a4paper, left=2cm, right=2cm, top=2cm, bottom=2cm]{geometry}
\usepackage[many]{tcolorbox}
\usepackage{xpatch}
\usepackage{mathtools}
\usepackage{doi}
\usepackage{float}
\usepackage{adjustbox}
\usepackage{caption}
\usepackage{subcaption}

\usetikzlibrary{arrows.meta}

\hypersetup{
	colorlinks=true,
	linkcolor=Sepia,
	filecolor=magenta,      
	urlcolor=Blue,
	citecolor=BlueViolet,
}

\usepackage[backend=biber, style=numeric-comp, url=false, doi=true, isbn=false, eprint=false, sorting = nyt, giveninits=true, backref=true]{biblatex}

%% file: inputs/macros.tex
\def\N{\mathbb{N}}
\def\R{\mathbb{R}}
\def\B{\mathbb{B}}
\def\Sph{\mathbb{S}}
\def\K{\mathbb{K}}
\def\Z{\mathbb{Z}}

\def\cD{\mathcal{D}}
\def\cWD{\mathcal{WD}}
\def\cL{\mathcal{L}}
\def\cV{\mathcal{V}}
\def\cT{\mathcal{T}}
\def\cF{\mathcal{F}}
\def\cP{\mathcal{P}}

\renewcommand{\emptyset}{\varnothing}

\newcommand\E[1]{\mathbb{E}\left[#1\right]}
\renewcommand\P[1]{\mathbb{P}\left(#1\right)}
\newcommand\1[1]{\mathbb{1}\left\{#1\right\}}

\newcommand\norm[1]{\lVert #1 \rVert}
\newcommand\scale[1]{\phi_{#1} \,}

\DeclarePairedDelimiter{\ceil}{\lceil}{\rceil}
\DeclarePairedDelimiter{\floor}{\lfloor}{\rfloor}

\newcommand{\origin}{\ensuremath{\mathbf{o}}}
\newcommand{\support}{\operatorname{h}}
\newcommand{\Poisson}{\operatorname{Poi}}

\DeclareMathOperator*{\argmax}{arg\,max}

\def\dint{\operatorname{d}\!}

\def\Int{\operatorname{Int}}

\def\conv{\operatorname{conv}}

\def\Ext{\operatorname{Ext}}

\newcommand{\halfspaces}{\mathcal{H}^-}
\newcommand{\hyperplanes}{\mathcal{H}}
\newcommand{\halfspace}{H^-}
\newcommand{\hyperplane}{H}
\newcommand{\height}{H}
\newcommand{\cell}{\mathscr{C}}
\newcommand{\halfspacesset}{\mathbf{H}}
\newcommand{\distH}{d_{\mathrm{H}}}
\newcommand{\flower}{\mathscr{F}}
\newcommand{\convexbody}{K}
\newcommand{\biggerconvexbody}{L}
\newcommand{\bop}{\operatorname{b}}
\newcommand{\arc}{\operatorname{arc}}
\newcommand{\shape}{\mathfrak{s}}
\newcommand{\maybeprime}{{(\prime)}}
\newcommand{\typcell}{Z}
\newcommand{\typcellany}{{\typcell^\maybeprime}}
\newcommand{\typcellprime}{{\typcell^\prime}}
\newcommand{\extypcell}{\widehat{\typcell}}
\newcommand{\extypcellany}{\widehat{\typcell}^\maybeprime}
\newcommand{\extypcellprime}{\widehat{\typcell}^\prime}

\renewcommand{\mid}{\,:\,}

\xpatchcmd{\proof}{\itshape}{\bfseries\itshape}{}{}

\newtheorem{theorem}{Theorem}[section]
\newtheorem{prop}[theorem]{Proposition}
\newtheorem{lemma}[theorem]{Lemma}
\newtheorem{corollary}[theorem]{Corollary}
\newtheorem{rem}[theorem]{Remark}
\newtheorem{conj}[theorem]{Conjecture}

%% file: inputs/content.tex
\input{inputs/intro}
\input{inputs/notation}
\input{inputs/setup}

\input{inputs/complementary}
\input{inputs/lower-bounds}
\input{inputs/upper-bound}

\input{inputs/concentration}
\input{inputs/epilogue}
\input{inputs/bibliography}

%% file: inputs/intro.tex
\section{Introduction}
\label{sec:intro}

Consider a sequence of random graphs $(G_n = ([n],E_n))_{n\in\N}$ and the maximal degrees $M_n = \max \{ \deg v \mid v \in V_n  \}$ of these graphs.
The asymptotic behavior of these random variables has been studied for a variety of random graphs
including Erd\"os-R\`enyi graph in some regimes \cite{bollobas_random_2001},
uniformly chosen random trees \cite{carr_tree_1994},
triangulations of a polygon \cite{gao_planar_2000},
planar graphs \cite{mcdiarmid_planar_2006,mcdiarmid_planar_2008,drmota_planar_2014},
and minor-closed graphs \cite{gimenez_minor_2016}.
In all these instances it has been observed that $M_n$ is of order $\log n/f(n)$ where $f(n)$ is a correcting factor of slower growth that depends on the specific model.
In many cases it was shown that $M_n$ concentrates on two values, meaning that there is a deterministic sequence $k_n$ such that $\P{M_n\in\{k_n,k_n+1\}}\rightarrow 1$.

While all aforementioned models are purely combinatorial in nature, Penrose has shown the same in \cite{penrose_geometric_2003} for the random geometric graph (also known as the Gilbert graph).
In \cite{bonnet_maximal_2020} authors proved similar concentration for the graph of the Poisson-Delaunay triangulation in $\R^d$, where the maximum in question is taken as $M_\rho = \max \{ \deg v \mid v \in V(G) \cap [0,\rho^{\frac{1}{d}}]^d  \}$, i.e.\ the maximal degree in a growing cubic window. The concentration is shown to be on at most two values for $d=2$ and a finite number of values depending only on $d$ for higher dimensions. The proof involves looking into intricate geometric properties of the model established largely in \cite{bonnet_cells_2018}.
More recently a different behavior for maximal degree in a window was observed in the work of Bhattacharjee and Schulte \cite{bhattacharjee_degrees_2022} for scale-free inhomogeneous random graphs.

The $\beta$- and $\beta'$-Delaunay tessellations are new random triangulation models introduced  by Gusakova, Kabluchko and Th\"ale in \cite{gusakova_beta-delaunay_2022-1,gusakova_beta-delaunay_2022-2,gusakova_beta-delaunay_2022-3,gusakova_beta-delaunay_2023,gusakova_sectional_2024}, see \cref{fig:tessellations}. Their constructions generalize the one of the classical Delaunay triangulation by taking into account additional weight parameters. Dual tessellations, $\beta$- and $\beta'$-Voronoi, respectively, are a particular case of Poisson-Laguerre tessellations discussed in \cite{gusakova_poisson-laguerre_2024,lautensack_laguerre_2008}. 
This paper is motivated by the study of the maximal degree in a large window for the skeletons of these tessellations.
En route, we establish some fundamental properties of the models, which are of independent interest and may find applications in future work.

\begin{figure}[ht]
	\centering
	\includegraphics[width=0.9\textwidth]{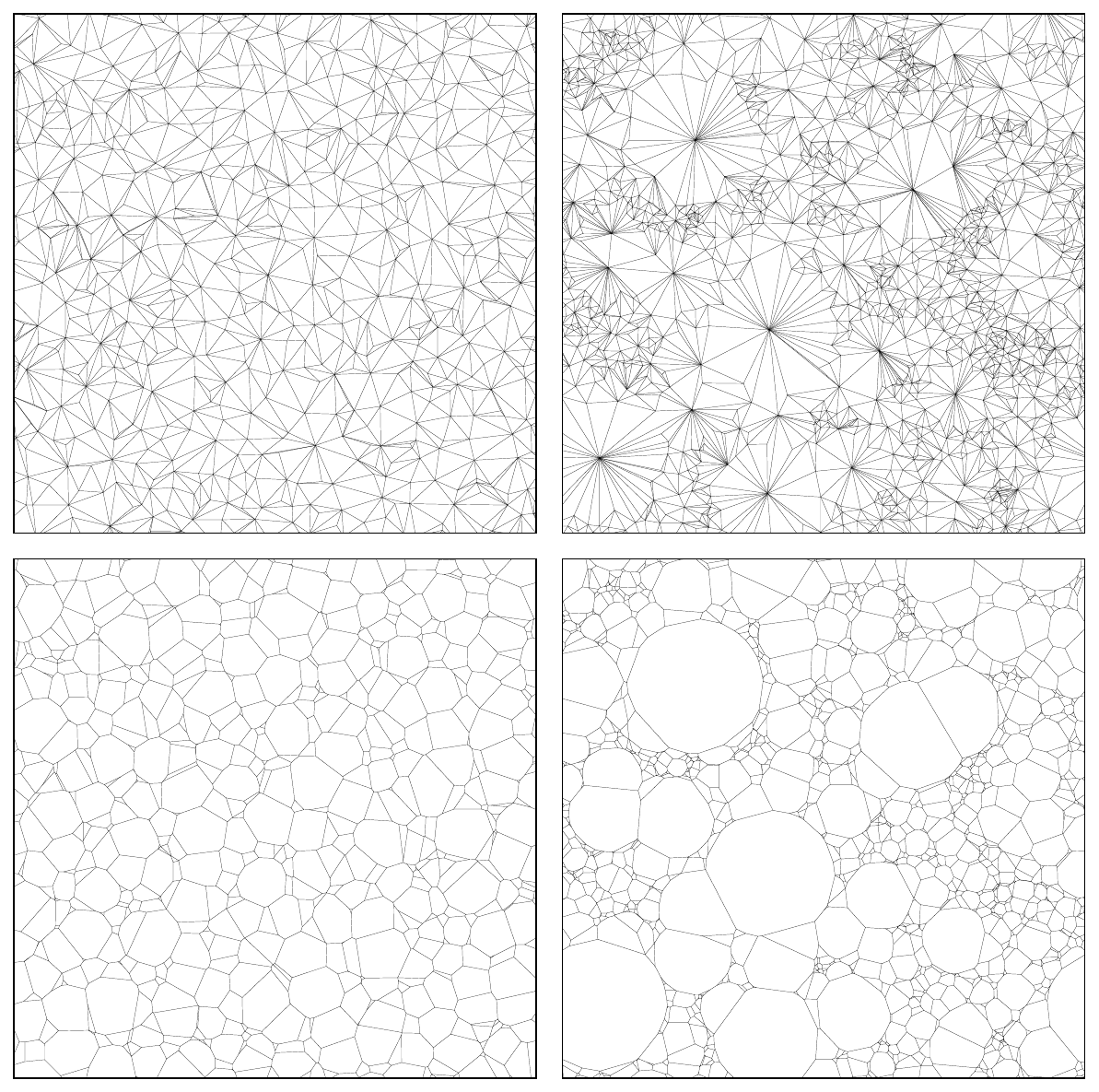}
	\caption{First row: $\beta$-Delaunay and $\beta'$-Delaunay triangulations with $\beta=3$. Second row: corresponding dual Voronoi tessellations.}
	\label{fig:tessellations}
\end{figure}

Before listing our results, we briefly present key concepts that will be defined in detail in the following sections.
Both $\beta$- and $\beta'$-Voronoi tessellations are constructed by assigning a (possibly empty) polyhedral cell in $\R^d$ to each point of a Poisson point process $\eta$ in $\R^d \times \R$, intensity measure of which has density proportional to a power of the height coordinate. The construction is simply the following:
$$\cell_{\eta}(v,h) \coloneqq \{w\in \R^d \mid \norm{w-v}^2 + h \le \norm{w-v'}^2 + h' \text{ for all } (v',h')\in \eta\}, \qquad (v,h)\in\eta.$$
The \textit{typical cells} of these tessellations are denoted by $\extypcellany=(x_\typcellany,\typcellany)$ consisting of a polytope $\typcellany\subset\R^d$ and a nucleus $x_\typcellany=(\origin,h_\typcellany) \in \R^d\times \R$.
We will define the \textit{$\Phi$-content} which is a homogeneous functional on such pairs related to the probability that a given $(d+1)$-tuple of points gives rise to a simplex of the dual triangulation. The \textit{shape} of a cell can be thought of as its homothetic copy of $\Phi$-content $1$.
We are particularly interested in the number of facets of $\typcellany$, which corresponds to the typical degree of a vertex in the dual Delaunay models.

In this paper we:
\begin{itemize}
	\item \textbf{Establish a Complementary Theorem} (\cref{complementary}) for both models in the style of \cite{bonnet_cells_2018} which investigated Poisson hyperplane tessellations. Crucially, it shows that conditional on the number of facets of $\typcellany$, the
	$\Phi$-content and the shape of $\extypcellany$ are distributed independently and the former has the gamma distribution with explicit parameters.
	A general description of complementary theorems in stochastic geometry can be found at the beginning of \cref{sec:complementary}.
	\item \textbf{Give bounds for the distribution of a typical degree} (\cref{lower-bound-beta,lower-bound-beta-prime,upper-bound-beta}). For $\beta$-Delaunay, the lower and upper bound match in order and show a super-exponential decay (more specifically, of order $(ck)^{-\frac{2}{d-1} k}$), similar to what has been shown for the classical Poisson-Delaunay triangulation \cite{bonnet_maximal_2020}, and for the number of facets of the typical cell in a Poisson hyperplane tessellation \cite{bonnet_cells_2018}. For $\beta'$-Delaunay, we establish an exponential (i.e.\ $c^k$) lower bound, indicating a completely different behavior.
	There exists in the literature a number of random graph models with a similar or slower decay of the degree distribution. This includes models of random spatial graphs, e.g. \cite{bhattacharjee_degrees_2022}, but to the best of authors' knowledge, this is the first such result for the graph of a random tessellation.
	\item \textbf{Establish a concentration result} for the maximal degree in a window for the graph of the $\beta$-Delaunay triangulation (\cref{concentration}).
	Our result applies also to the classical Poisson-Delaunay model for which we recover the concentration on two values when $d=2$ obtained in \cite[Theorem 1]{bonnet_maximal_2020} and improve the concentration on finitely many values for $d\geq 3$ established in \cite[Theorem 3]{bonnet_maximal_2020} by showing concentration on a strictly smaller number of values.
	Furthermore, in the case $d=2$, the new proof is significantly simpler. See \cref{relation-poisson} for details.
\end{itemize}

Our contributions lead to the natural formulation of several conjectures, which we present in \cref{sec:open}. There, we provide heuristics supporting these conjectures and discuss the challenges involved in proving them.

A brief guide to this text is as follows.
In \cref{sec:notation}, we recall basic notions essential to this paper.
In \cref{sec:models}, we introduce the tessellation models under investigation.
\cref{sec:complementary}, which is the longest, analyzes the typical cells of the $\beta$- and $\beta'$-Voronoi tessellations (duals of the corresponding Delaunay models). In particular, we estimate the distribution of the typical height of a vertex, establish a Complementary Theorem, and lay the groundwork for the results on typical degrees.
In \cref{sec:lower,sec:upper}, we bound the distribution of a typical degree from below and from above, respectively.
In \cref{sec:concentration}, we address the motivating question of the paper by proving a concentration result for the maximal degree.
Finally, in \cref{sec:open}, we list open questions and directions for future work.

%% file: inputs/notation.tex
\section{Preliminaries and notation}
\label{sec:notation}

In this section we set the notation used throughout the paper, starting with some common conventions.

\subsection{Basics}

We will mostly work in $\R^d \times \R$ with its elements in the form $x = (v,h)$, where $h\in \R$ represents the weight (also called height) associated to the point in $v\in \R^d$.
The origin is denoted by $\origin = (0,\ldots,0)\in \R^d$.
Brackets $\langle \cdot, \cdot \rangle$ denote the standard scalar product and $\norm{\cdot}$ the Euclidean norm.
For an indexed collection of points from any space $E$ we use the short notation
\[ x_{a:b} \coloneqq (x_a,x_{a+1},\ldots,x_b) \in E^{b-a+1}. \]
The unit ball in the Euclidean space is denoted as $\B^d \coloneqq \{v\in\R^d \mid \norm{v} < 1\}$ and the unit sphere $\Sph^{d-1} \coloneqq \partial \B^d$.
Their respective volume and surface area are denoted by $\kappa_d$ and $\omega_d$ and their values are recalled in \Cref{sec:constants}.
A ball with center $v\in \R^d$ and radius $r>0$ is denoted by $\B^d_v(r) \coloneqq r\B^d + v$.
For $v_0, \ldots, v_d \in \R^d$ in general position there exist unique $v\in \R^d$, $r>0$ such that $v_0, \ldots, v_d \in \partial\B^d_v(r)$ and we write $\B^d_v(r) \eqcolon B(v_0,\ldots,v_d)$.

We will use $\lambda_d$ for the Lebesgue measure on $\R^d$ and $\sigma_{d-1}$ for the spherical Lebesgue measure on $\Sph^{d-1}$.
We use the short notation $[n] \coloneqq \{1,\ldots,n\}$ for any $n\in\N$.
For a set $X$ and $n\in\N$, we denote by $X^n_{\neq}$ the set of all $n$-tuples of distinct elements of $X$.

\subsection{Half-spaces, polytopes, convex bodies}

Let $\overline{\R}=\R\cup\{\pm\infty\}$ be the extended real line.
For $u\in \R^d$ and $t\in \overline{\R}$, denote $\halfspace(u,t) \coloneqq \{v\in\R^d \mid \langle v,u \rangle \le t\}$. If $\halfspace(u,t)$ is distinct from $\R^d$ or $\emptyset$, we call it a \textbf{half-space} of $\R^d$.
We will use the space of all half-spaces and its extension:
$$\halfspaces \coloneq \{\halfspace(u,t) \mid u\in \Sph^{d-1},\ t\in \R\},
\qquad
\overline{\halfspaces} \coloneq \halfspaces \cup \{\emptyset,\R^d\}.$$
The topology on $\halfspaces$ is induced by its bijection with $\Sph^{d-1} \times \R$, making them homeomorphic. 
The extended space $\overline{\halfspaces}$ can be constructed by taking $\Sph^{d-1} \times \overline{\R}$ and contracting the spheres $\Sph^{d-1}\times\{-\infty\}$ and $\Sph^{d-1}\times\{\infty\}$ into points. It inherits its topology from this construction, and is thus homeomorphic to the sphere $\Sph^d$.

A \textbf{hyperplane} is the boundary of a half-space:
$$\hyperplane(u,t) = \hyperplane(-u,-t) \coloneqq \halfspace(u,t) \cap \halfspace(-u,-t) = \{v\in\R^d \mid \langle v,u \rangle = t\}.$$
The space of all hyperplanes 
$$ \hyperplanes \coloneqq \{ \hyperplane(u,t) \mid u\in \Sph^{d-1}, t\in \R\}$$
is equipped with the topology inherited from $\Sph^{d-1} \times \R$ through the canonical surjection.

We will use the term \textbf{polyhedron} for a finite intersection of half-spaces $P = \bigcap_{i=1}^n \halfspace_i$, where $\halfspace_1,\ldots,\halfspace_n\in\halfspaces$.
And we will call a polyhedron that is bounded as a set in $\R^d$ a \textbf{polytope}.

A hyperplane $\hyperplane$ is a \textbf{supporting hyperplane} of a polyhedron $P$ if $F = P \cap \hyperplane \ne \emptyset$ and $\Int(P) \cap \hyperplane = \emptyset$, in which case we call $F$ a \textbf{face} of $P$. Note that a face of a polyhedron is a polyhedron itself.
For a polyhedron $P$ we denote the set of its faces of affine dimension $k$ by $\cF_k(P)$. Let $\cF = \bigsqcup_{k=0}^{d-1} \cF_k$ and $f_k(P) = |\cF_k(P)|$.
Faces of dimension $d-1$, $1$, and $0$ are called \textbf{facets}, \textbf{edges}, and \textbf{vertices}, respectively.

We denote the spaces of $n$-facet polytopes and of all polytopes by
$$\cP_n \coloneqq \{P = \bigcap_{i=1}^n \halfspace_i \mid P \text{ is bounded},\  f_{d-1}(P) = n\}
\quad \text{and} \quad 
\cP \coloneqq \bigsqcup_{n=d+1}^\infty \cP_n,$$
respectively.
It is a subspace of $\mathbb{K}$ the space of all \textbf{convex bodies} (i.e.\ compact nonempty convex sets) in $\R^d$, equipped with Hausdorff metric $\distH$.
On $\mathbb{K} \times \Sph^{d-1}$ we can define
\begin{equation}
	\label{e:support}
	\support(\convexbody,u) \coloneqq \max_{x\in \convexbody} \{\langle x, u \rangle\} = \min \{t\in \R \mid \convexbody\subset \halfspace(u,t)\},
\end{equation}
the value of the support function of $\convexbody$ in direction $u$.

\subsection{Tessellations}

A \textbf{tessellation} of $\R^d$ is a locally finite covering by polyhedra that do not have common interior points. More precisely, a set of polyhedra $\cT$ is a tessellation of $\R^d$ if:
\begin{itemize}
	\item $\bigcup_{P\in\cT} P = \R^d$.
	\item $\left|\{P\in\cT \mid P\cap B \ne \emptyset\}\right|<\infty$ for any bounded $B\subset \R^d$.
	\item $\Int(P_1)\cap \Int(P_2) = \emptyset$ for any $P_1,P_2 \in \cT$.
\end{itemize}
We call the elements of a tessellation its \textbf{cells}.
We say that $\cT$ is \textbf{face-to-face} when, additionally, for any 
two cells their intersection is either empty or a face of both.
Denote $\cF_k(\cT) \coloneqq \bigcup_{P\in\cT} \cF_k(P)$, $k=0,\ldots,d-1$, and $\cF_d(\cT) \coloneqq \cT$ and call these faces of $\cT$. 
A face-to-face tessellation $\cT$ is called \textbf{normal} when for each $k\in\{0,\ldots,d-1\}$ and $F\in \cF_k(\cT)$,
there are exactly $d-k+1$ cells in $\cT$ having $F$ as a face.

Lastly, vertices and edges of the cells in a tessellation $\cT$ form a graph $G(\cT) \coloneqq (V,E)$, where
$V = \cF_0(\cT)$ and ${E = \{\cF_0(e) \mid e \in \cF_1(\cT)\}}$. 

\subsection{Paraboloids}

\begin{figure}[H]
	\centering
	\input{inputs/parabolas}
	\caption{Paraboloids in $\R^1\times\R$.} 
	\label{fig:paraboloids}
\end{figure}
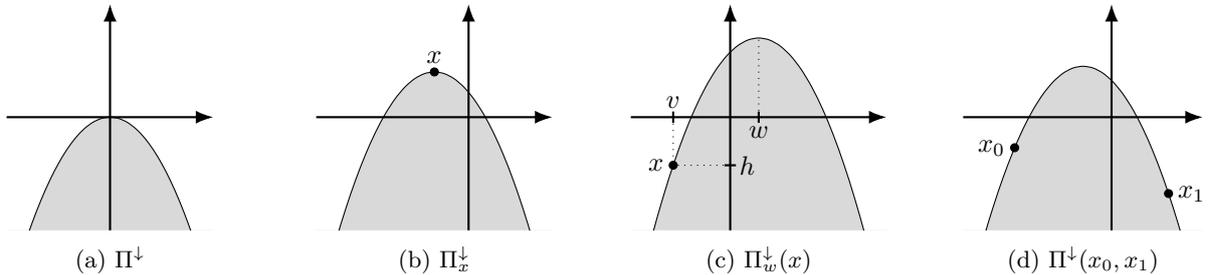

We will consider numerous times (downward) paraboloids in $\R^d \times \R$ of the forms
\[
	\Pi^\downarrow \coloneqq \{(v,h)\in \R^d \times \R \mid h\le - \norm{v}^2\} ,
	\quad
	\Pi^\downarrow_x \coloneqq \Pi^\downarrow + x ,
	\quad
	\Pi^\downarrow_w(x) \coloneqq \Pi^\downarrow_{(w,h+\norm{w-v}^2)} ,
\]
for arbitrary $x=(v,h)\in \R^d \times \R$ and $w\in \R^d\setminus\{v\}$, see \cref{fig:paraboloids}.
The first one is the paraboloid with an \textit{apex at the origin}, the second one is a translation of the first one which has an \textit{apex at $x$}, and the third one is a paraboloid that has \textit{$x$ on its boundary and its apex on the line $\{w\}\times\R$}.
We complete this list by the \textbf{circumscribed downward paraboloid} $\Pi^\downarrow (x_0,\ldots,x_d)$ of $d+1$ boundary points $x_i=(v_i,h_i)$, $i=0,\ldots,d$, with $v_0,\ldots,v_d\in \R^d$ affinely independent. It is defined as the unique paraboloid such that
${x_0,\ldots,x_d \in \partial\Pi^\downarrow (x_0,\ldots,x_d)}$.

\subsection{Functions and constants}
\label{sec:constants}

Recall the definitions of the \textbf{gamma function}, the \textbf{incomplete gamma function} and the \textbf{beta function}:
\[ \Gamma(z) = \int_0^\infty t^{z-1} e^{-t} \dint t,
\qquad
\Gamma(z,x) = \int_x^\infty t^{z-1} e^{-t} \dint t ,
\qquad 
B(z_1,z_2) = \int_0^1 t^{z_1-1} (1-t)^{z_2-1} \dint t = \frac{\Gamma(z_1)\Gamma(z_2)}{\Gamma(z_1+z_2)} , \]
well defined in particular for $x,z,z_1,z_2>0$. Recall that $\Gamma(z+1) = z \Gamma(z)$.
The following bound will be useful for our purposes.

\begin{prop}[{\cite[Satz~4.4.3]{gabcke_1979} and \cite[Eq. (2) and (9)]{qi_inequalities_1999}}]
	\label{incomplete-gamma}
	Let $x > 0$ and $z\in(0,\max(x,1))$.
	Then
	$$\Gamma(z,x) \le \max(z,1) e^{-x} x^{z-1}.$$
\end{prop}

We denote the volume of the $d$-dimensional unit ball and the surface area of its boundary by
$$
\kappa_d \coloneqq \lambda_d(\B^d) = \frac{\pi^\frac{d}{2}}{\Gamma(1+\frac{d}{2})}
\qquad \text{and} \qquad
\omega_d \coloneqq \sigma_{d-1}(\Sph^{d-1}) = \frac{2\pi^\frac{d}{2}}{\Gamma(\frac{d}{2})}.
$$
We also introduce, for any $c>0$ and $v'\in \R^d$, the scaling operator $\scale{c}$ and the translation by $v'$ on $\R^d \times \R$ defined by
\begin{align}
	\label{e:scale}
	\scale{c} (v,h) \coloneqq (cv,c^2 h) 
	\quad \text{and} \quad
	(v,h) + v' = v' + (v,h) \coloneqq (v+v', h'),
\end{align}
for any $(v,h)\in \R^d \times \R$.
These operators also occasionally extend to spaces of the form
$(\R^d \times \R)^l$,
$(\R^d \times \R)^l \times (\R^d)^m \times \R^n$
or
$(\R^d \times \R) \times \K$ by setting
\begin{align*}
	\scale{c} (x_{1:l})
	&= (\scale{c} x_1, \ldots, \scale{c} x_l),
	&
	x_{1:l}+v'
	&= (x_1+v', \ldots, x_l+v'),
	\\
	\scale{c} (x_{1:l}, v_{1:m}, h_{1:n}) 
	&= (\scale{c} x_{1:l}, c\,v_{1:m}, c^2\,h_{1:n}),
	&	
	(x_{1:l}, v_{1:m}, h_{1:n}) + v'
	&= (x_{1:l}+v', v_{1:m}+v', h_{1:n}),
	\\
	\scale{c}(x,K)
	&= (\scale{c} x, c K),
	&
	(x,K) + v'
	&= (x+v', K+v'),
\end{align*}
for any $x, x_1,\ldots,x_l\in \R^d \times \R$, $v_1,\ldots,v_m\in \R^d$, $h_1,\ldots,h_n\in \R$ and $K\in \K$.

\begin{rem}[Scaling property and translation equivariance of paraboloids]
	\label{scprop}
	An essential property of these operators is that for any $c>0$, $v\in\R^d$ and $x,x_0,\ldots,x_d\in\R^d\times\R$,
	\begin{enumerate}
		\item $\Pi^\downarrow_{v + \scale{c} x} = v + \scale{c} \Pi^\downarrow_x $;
		\item $\Pi^\downarrow (v+\scale{c} x_{0:d}) = v + \scale{c} \Pi^\downarrow (x_{0:d})$.
	\end{enumerate}
\end{rem}

%% file: inputs/parabolas.tex
\def\Ymin{-2}  
\def\Ymax{2}
\def\Xmargin{0.4}
\def\Myscale{0.75}

\begin{subfigure}{0.24\textwidth}
\centering
\begin{tikzpicture}[scale=\Myscale]
    \def\shiftx{0}
    \def\shifty{0}
    \pgfmathsetmacro{\Xmin}{\shiftx - sqrt(\shifty - \Ymin)-\Xmargin}
    \pgfmathsetmacro{\Xmax}{\shiftx + sqrt(\shifty - \Ymin)+\Xmargin}
    \clip (\Xmin,\Ymin) rectangle (\Xmax,\Ymax);
     \fill[gray!30, domain=\Xmin:\Xmax, variable=\x]
     plot ({\x},{-(\x-\shiftx)*(\x-\shiftx)+\shifty})
     -- (\Xmax,\Ymin) -- (\Xmin,\Ymin) -- cycle;

    \draw[->, thick, >=Latex] (\Xmin,0) -- (\Xmax,0); 
    \draw[->, thick, >=Latex] (0,\Ymin) -- (0,\Ymax); 
    \draw[domain=\Xmin:\Xmax,smooth,variable=\x] plot ({\x},{-(\x-\shiftx)*(\x-\shiftx)+\shifty}); 
\end{tikzpicture}
\caption{$\Pi^\downarrow$}
\end{subfigure}
\hfill
%
\begin{subfigure}{0.24\textwidth}
\centering
\begin{tikzpicture}[scale=\Myscale]
    \def\shiftx{-0.6}
    \def\shifty{0.8}
    \def\xx{\shiftx}
    \def\xy{\shifty}
    \pgfmathsetmacro{\Xmin}{\shiftx - sqrt(\shifty - \Ymin)-\Xmargin}
    \pgfmathsetmacro{\Xmax}{\shiftx + sqrt(\shifty - \Ymin)+\Xmargin}
    \clip (\Xmin,\Ymin) rectangle (\Xmax,\Ymax);
    \fill[gray!30, domain=\Xmin:\Xmax, variable=\x]
    plot ({\x},{-(\x-\shiftx)*(\x-\shiftx)+\shifty})
    -- (\Xmax,\Ymin) -- (\Xmin,\Ymin) -- cycle;
    \draw[->, thick, >=Latex] (\Xmin,0) -- (\Xmax,0); 
    \draw[->, thick, >=Latex] (0,\Ymin) -- (0,\Ymax); 
    \draw[domain=\Xmin:\Xmax,smooth,variable=\x] plot ({\x},{-(\x-\shiftx)*(\x-\shiftx)+\shifty}); 
    \filldraw[black] (\xx,\xy) circle (2pt); 
    \node[above] at (\xx,\xy) {$x$};
\end{tikzpicture}
\caption{$\Pi^\downarrow_x$}
\end{subfigure}
\hfill
%
\begin{subfigure}{0.24\textwidth}
\centering
\begin{tikzpicture}[scale=\Myscale]
    \def\shiftx{0.5}
    \def\shifty{1.4}
    \def\xx{-1}
    \def\xy{{\shifty-(\shiftx-\xx)*(\shiftx-\xx)}}
    \def\wx{\shiftx}
    \def\wy{0}
    \pgfmathsetmacro{\Xmin}{\shiftx - sqrt(\shifty - \Ymin)-\Xmargin}
    \pgfmathsetmacro{\Xmax}{\shiftx + sqrt(\shifty - \Ymin)+\Xmargin}
    \clip (\Xmin,\Ymin) rectangle (\Xmax,\Ymax);
     \fill[gray!30, domain=\Xmin:\Xmax, variable=\x]
     plot ({\x},{-(\x-\shiftx)*(\x-\shiftx)+\shifty})
     -- (\Xmax,\Ymin) -- (\Xmin,\Ymin) -- cycle;

    \draw[->, thick, >=Latex] (\Xmin,0) -- (\Xmax,0); 
    \draw[->, thick, >=Latex] (0,\Ymin) -- (0,\Ymax); 
    \draw[domain=\Xmin:\Xmax,smooth,variable=\x] plot ({\x},{-(\x-\shiftx)*(\x-\shiftx)+\shifty}); 
    \filldraw[black] (\xx,\xy) circle (2pt); 
    \draw[thick] (\wx,0.1) -- (\wx,-0.1); 
    \draw[thick] (\xx,-0.1) -- (\xx,0.1); 
    \draw[thick] (-0.1,\xy) -- (0.1,\xy); 
    \draw[dotted] (\wx,\wy) -- (\shiftx,\shifty); 
    \draw[dotted] (\xx,0) -- (\xx,\xy); 
    \draw[dotted] (0,\xy) -- (\xx,\xy); 
    \node[left] at (\xx,\xy) {$x$};
    \node[above] at (\xx,0) {$v$};
    \node[right] at (0,\xy) {$h$};
    \node[below] at (\wx,\wy) {$w$};
\end{tikzpicture}
\caption{$\Pi^\downarrow_w(x)$}
\end{subfigure}
\hfill
%
\begin{subfigure}{0.24\textwidth}
\centering
\begin{tikzpicture}[scale=\Myscale]
    \def\shiftx{-0.5}
    \def\shifty{0.9}
    \def\xonex{-1.7}
    \def\xoney{{\shifty-(\shiftx-\xonex)*(\shiftx-\xonex)}}
    \def\xtwox{1}
    \def\xtwoy{{\shifty-(\shiftx-\xtwox)*(\shiftx-\xtwox)}}
    \pgfmathsetmacro{\Xmin}{\shiftx - sqrt(\shifty - \Ymin)-\Xmargin}
    \pgfmathsetmacro{\Xmax}{\shiftx + sqrt(\shifty - \Ymin)+\Xmargin}
    \clip (\Xmin,\Ymin) rectangle (\Xmax,\Ymax);
     \fill[gray!30, domain=\Xmin:\Xmax, variable=\x]
     plot ({\x},{-(\x-\shiftx)*(\x-\shiftx)+\shifty})
     -- (\Xmax,\Ymin) -- (\Xmin,\Ymin) -- cycle;
    \draw[->, thick, >=Latex] (\Xmin,0) -- (\Xmax,0); 
    \draw[->, thick, >=Latex] (0,\Ymin) -- (0,\Ymax); 
    \draw[domain=\Xmin:\Xmax,smooth,variable=\x] plot ({\x},{-(\x-\shiftx)*(\x-\shiftx)+\shifty}); 
    \filldraw[black] (\xonex,\xoney) circle (2pt); 
    \filldraw[black] (\xtwox,\xtwoy) circle (2pt);  
    \node[left] at (\xonex,\xoney) {$x_0$};
    \node[right] at (\xtwox,\xtwoy) {$x_1$};
\end{tikzpicture}
\caption{$\Pi^\downarrow(x_0, x_1)$}
\end{subfigure}

%% file: inputs/setup.tex
\section{The models}
\label{sec:models}

\subsection{Deterministic constructions}

Let a set of points $X\subset \R^d \times \R$ be well-spread and lying in general position, which in our context means that
\begin{subequations}
	\begin{equation}
		\tag{WS} \label{e:well-spread}
		\conv \left( \cup_{(v,h)\in X} \{v\} \right)
		=\R^d ,
	\end{equation}
	\begin{equation}
		\tag{GP} \label{e:general-position}
		\begin{split} 
			&\text{$v_0,\ldots,v_d$ are affinely independent for any $d+1$ points $(v_0,h_0),\ldots,(v_d,h_d)$} \\
			&\text{and no $d+2$ points lie on the boundary of a same downward paraboloid.}
		\end{split}
	\end{equation}
\end{subequations}
Then the \textbf{Weighted Delaunay triangulation} (also known as regular triangulation) $\cWD(X)$ is the set of simplices in $\R^d$ chosen in the following way (see \cref{fig:parpro}). For any $d+1$ points $x_0 = (v_0,h_0),\ldots,x_d = (v_d,h_d)\in X$, their vector coordinates $v_0,\ldots,v_d$ form a simplex of the triangulation whenever their circumscribed downward paraboloid does not contain any other points of $X$:
$$
\conv\{v_0,\ldots,v_d\}\in \cWD(X)
\Leftrightarrow
X\cap \Int(\Pi^\downarrow(x_{0:d})) = \emptyset.
$$
\begin{figure}[ht]
	\centering
	\includegraphics[width=\textwidth]{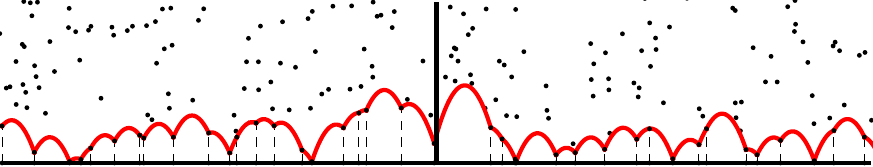}
	\caption{A Weighted Delaunay triangulation on $\R$. Illustration courtesy of Anna Gusakova.}
	\label{fig:parpro}
\end{figure}%
Dual to it is the \textbf{Laguerre diagram} $\cL(X)$, which is also known as Laguerre-Voronoi diagram, power diagram, and generalized Dirichlet tessellation.
It is defined as the collection of polyhedral cells
\begin{equation}
	\label{e:Laguerre}
	\cell_X(v,h) = \{w\in \R^d \mid \norm{w-v}^2 + h \le \norm{w-v'}^2 + h' \text{ for all } (v',h')\in X\} \subset \R^d,
\end{equation}
for all $(v,h)\in X$.
Note that such a cell might be empty, and that even if non-empty it does not necessarily contain  $v$ (sometimes called nucleus of the cell).

Let us assume that $\cL(X)$ is a face-to-face tessellation, with $X$ satisfying \eqref{e:general-position} and \eqref{e:well-spread}.
It was repeatedly observed in literature (see e.g.\ \cite{aurenhammer_power_1987,schlottmann_periodic_1993,aurenhammer_voronoi_2013}) that $\cL(X)$ is also normal, $\cWD(X)$ is a face-to-face tessellation as well, and they are dual to each other. More precisely, let $k\in\{0,\ldots,d\}$.
Any  $F\in \cF_k(\cL(X))$ takes the form $F = \cap_{i=0}^{d-k} \cell_X(v_i,h_i)$ for some distinct points $x_0, \ldots, x_{d-k}\in X$ having the form $x_i = (v_i,h_i)$.
It can be assigned a simplex $\theta_k(F)\coloneqq \conv\{v_1,\ldots,v_{d-k}\}$. Then each map $\theta_k$ is a bijection between $\cF_k(\cL(X))$ and $\cF_{d-k}(\cWD(X))$.
The most important part of this duality for us is the following lemma.
\begin{lemma}
	\label{duality}
	Let $X$ satisfy \eqref{e:general-position} and \eqref{e:well-spread} and assume $\cL(X)$ is a face-to-face tessellation.
	Let $(v_1,h_1)$, $(v_2,h_2)\in X$.
	Then the graph $G(\cWD(X))$ has an edge between $v_1$ and $v_2$ if and only if $\cell_X(v_1,h_1)$ and $\cell_X(v_2,h_2)$ share a facet.
\end{lemma}

We will consider both constructions on the same set of points simultaneously, since they complement each other.

\begin{rem}
The local geometry around each $v_0\in \cF_0(\cWD(X))$ is captured by the \textbf{Voronoi flower} (see \cref{fig:flower}) around $x_0 = (v_0,h_0)\in X$ which is defined as
\begin{equation}
	\label{e:flower}
	\flower_X(x_0)\coloneqq \bigcup_{\substack{x_{1:d}\in X^d_{\ne}\\ \Int(\Pi^\downarrow(x_{0:d})) \cap X = \emptyset}} \Pi^\downarrow(x_{0:d}).
\end{equation}
Precisely, assume $x_1 = (v_1,h_1), x_2 = (v_2,h_2) \in X$. Then
$\{v_1,v_2\}\in G(\cWD(X)) \Leftrightarrow x_2  \in \flower_X(x_1) \Leftrightarrow x_2  \in \partial \flower_X(x_1)$.
This object will be again more heavily discussed in \cref{s:flower}.
\end{rem}

\begin{figure}[ht]
	\centering
	\begin{subfigure}{0.45\textwidth}
		\adjincludegraphics[height=7cm,trim={0 0 0 {0.3\height}},clip]{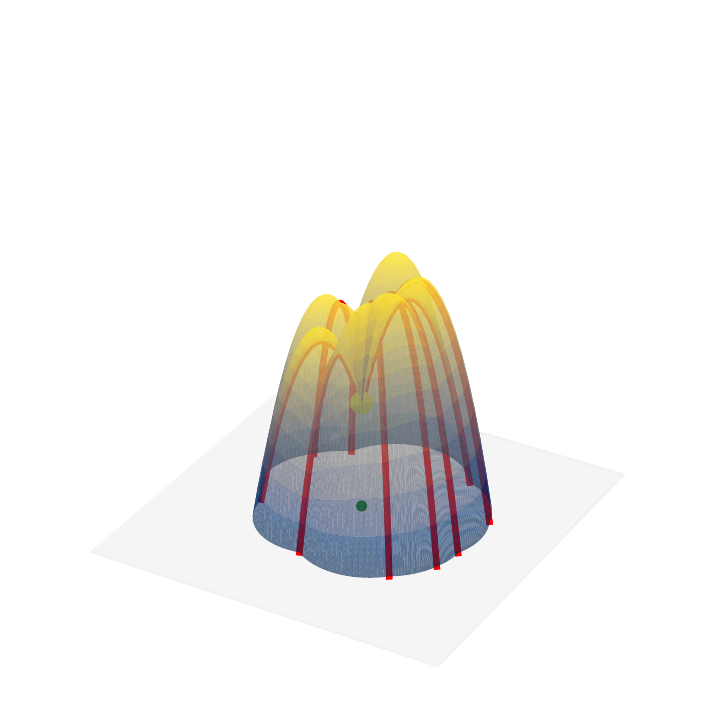}
	\end{subfigure}
	\hfill
	\begin{subfigure}{0.45\textwidth}
		\adjincludegraphics[height=7cm,trim={90 90 80 80},clip]{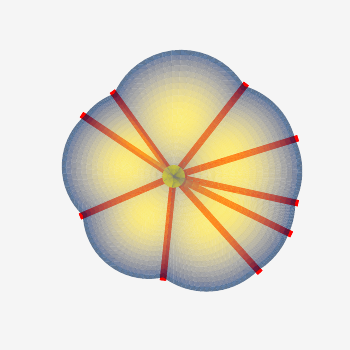}
	\end{subfigure}
	\caption{A Voronoi flower of a point in $\R^2$, view from the side and from above}
	\label{fig:flower}
\end{figure}

\begin{rem}
	\label{flower-above-vertices}
	Paraboloids in the union in \eqref{e:flower} have apices directly above the vertices of the Laguerre cell. In particular, for $x\in X$,
	$$\flower_X(x)= \bigcup_{w\in \cF_0(\cell_X(x))} \Pi^\downarrow_w(x).$$
\end{rem}

Let $V \subset \R^d$ be locally finite. As a special case, define its \textbf{Delaunay triangulation} and \textbf{Voronoi tessellation} as
$$
\cD(V) \coloneqq \cWD(V\times\{0\})
\quad \text{and} \quad
\cV(V) \coloneqq \cL(V\times\{0\}).
$$

\subsection{Random tessellations}

Let $\eta$ be a Poisson point process with intensity $1$ on $\R^d$. Classical models of \textbf{Poisson-Delaunay} and \textbf{Poisson-Voronoi tessellations} are defined as
$\cD \coloneqq \cD(\eta)$ and $\cV \coloneqq \cV(\eta)$ respectively.
Both $\cWD$ and $\cL$ also give rise to random tessellation models which were introduced in \cite{gusakova_beta-delaunay_2022-1}. We will give a brief introduction of them as well.

For any $\beta>-1$ take a Poisson point process $\eta_\beta$ in the space $\R^d \times [0,\infty)$ with the intensity measure $\mu_\beta$ of density proportional to $h^\beta$. More specifically,
\begin{align}
	\label{def:mu_beta}
	\mu_\beta (\dint v, \dint h) \coloneqq \gamma c_{d+1,\beta} h^\beta \1{h\geq 0} \dint v \dint h 
	\quad \text{ where } \quad
	c_{d+1,\beta} \coloneqq \frac{\Gamma(\frac{d+1}{2}+\beta+1)}{\pi^\frac{d+1}{2}\Gamma(\beta+1)}
\end{align}
is a normalization constant
and $\gamma > 0$ the intensity parameter.
We call $\cD_\beta \coloneqq \cWD(\eta_\beta)$ the {\bf $\beta$-Delaunay triangulation}, and its dual counterpart $\cV_\beta \coloneqq \cL(\eta_\beta)$ the {\bf $\beta$-Voronoi tessellation}.

Similarly, for any $\beta > \frac{d}{2}+1$ take a Poisson point process $\eta^\prime_\beta$ in the space $\R^d \times (-\infty,0)$ with the intensity measure $\mu_\beta^\prime$ of density proportional to $(-h)^{-\beta}$. More specifically, 
\begin{align}
	\label{def:mu_beta_prime}
	\mu_\beta^\prime (\dint v, \dint h) \coloneqq \gamma c^\prime_{d+1,\beta} (-h)^{-\beta} \1{h<0} \dint v \dint h
	\quad \text{ where } \quad
	c^\prime_{d+1,\beta} \coloneqq \frac{\Gamma(\beta)}{\pi^\frac{d+1}{2}\Gamma(\beta - \frac{d+1}{2})}
\end{align}
is again a normalization constant
and $\gamma > 0$ the intensity parameter.
Now we call $\cD_\beta^\prime \coloneqq \cWD(\eta^\prime_\beta)$ the {\bf $\beta^\prime$-Delaunay triangulation} and $\cV_\beta^\prime \coloneqq \cL(\eta^\prime_\beta)$ the {\bf $\beta^\prime$-Voronoi tessellation}.
See \cref{fig:tessellations} for an illustration of these four random tessellation models.

Note that the value of the normalizing constants $c_{d+1,\beta}$ and $c^\prime_{d+1,\beta}$ is irrelevant for the main results of the paper.
Here we choose them such that they align with the normalizing constants used in \cite{gusakova_beta-delaunay_2022-1,gusakova_beta-delaunay_2023,gusakova_poisson-laguerre_2024,gusakova_beta-delaunay_2022-2,gusakova_beta-delaunay_2022-3,
gusakova_sectional_2024}.
There is however a distinction with these references.
We choose to define the model in $\R^d \times \R$ instead of $\R^{d-1}\times \R$ as in the aforementioned papers.

Following \cite{gusakova_beta-delaunay_2022-1}, in order to treat $\beta$- and $\beta'$-tessellations in a unified way we will use an additional variable. We put $\kappa = 1$ if the $\beta$-model is considered and $\kappa = -1$ in case of a $\beta'$-model. In this notation we can write that 
\begin{align*}
	\mu_\beta^\maybeprime 
	= \frac{\gamma c^\maybeprime_{d+1,\beta}}{\beta + \kappa} \, \lambda_d \otimes \lambda_\kappa^{(\kappa \beta + 1)}
\end{align*}
on the space $\R^d \times \R$, where we recall that $\lambda_d$ is the Lebesgue measure on $\R^d$ and where $\lambda_\kappa^{(r)}$ denotes a $r$-homogeneous measure on $\R$ defined by the formula: 
\begin{equation}
	\label{e:lambdakr}
	\lambda_\kappa^{(r)} (B) 
	\coloneq  \int_B \left|r\right| \, (\kappa t)^{r-1} \1{\kappa t > 0} \dint t ,
	\qquad \kappa \in \{-1,1\} ,\ r\in\R .	
\end{equation} 

The following lemma was stated first in \cite{gusakova_beta-delaunay_2022-1} and then reproved more rigorously in \cite{gusakova_poisson-laguerre_2024}.

\begin{lemma}[{\cite[Example 3.9]{gusakova_poisson-laguerre_2024}} Correctness of the model]
	Almost surely $\cD_\beta^\maybeprime$ and $\cV_\beta^\maybeprime$ are face-to-face tessellations of $\R^d$ for any $\beta$ in the allowed ranges, with $\cV_\beta^\maybeprime$ additionally being normal.
\end{lemma}

Let us also examine closer the intensity measures of the underlying Poisson point processes.

\begin{lemma}[Properties of intensity measures]
	\label{improp}
	\begin{enumerate}
		\item[]
		\item \label{improp-homo} The measures $\mu_\beta^\maybeprime$ are homogeneous of degree $\kappa 2 \beta + d + 2$ under the scaling operator and translation invariant: For $c>0$, $v\in\R^d$ and a Borel set $A\in \R^d \times \R$,
		$$\mu_\beta^\maybeprime (v+\scale{c} A) = c^{\kappa 2 \beta + d + 2} \mu_\beta^\maybeprime(A).$$
		\item \label{improp-para} For any $v\in \R^d$ and $\kappa h > 0$,
		$$
		\mu_\beta^\maybeprime (\Pi^\downarrow_{(v,h)}) = \gamma \frac{c^\maybeprime_{d+1,\beta}}{c^\maybeprime_{d,\beta}}
		\frac{(\kappa h)^{\kappa \beta + \frac{d}{2} + 1}}{|\kappa \beta + \frac{d}{2} + 1|}
		< \infty.$$
	\end{enumerate}
\end{lemma}
The quantity in \cref{improp-para} was shown to be finite in the proof of \cite[Lemma 3.4]{gusakova_beta-delaunay_2022-1}, and in particular computed for $\kappa=-1$. Nevertheless, we derive both quantities here as well for the sake of completeness.
\begin{proof}
	\cref{improp-homo} is easy to observe by inserting a change of parameters $x = v+\scale{c} y$ into the integrals that define measures in question.
	For \cref{improp-para}, note that due to \cref{improp-homo} and \cref{scprop}, it is sufficient to check that
	$\mu_\beta^\maybeprime (\Pi^\downarrow_{(\origin,\kappa)}) = \gamma \frac{c^\maybeprime_{d+1,\beta}}{|\kappa \beta + \frac{d}{2} + 1| c^\maybeprime_{d,\beta}}.$
	We consider the cases $\kappa = 1$ and $\kappa = -1$ separately.
	In the first case,
	\begin{align*}
		\mu_\beta(\Pi^\downarrow_{(\origin,1)})
		&= \gamma c_{d+1,\beta} \int_{\R_+} \int_{\R^d} \1{h\le 1 - \norm{v}^2} h^\beta \dint v \dint h
		\\
		&= \gamma c_{d+1,\beta}  \omega_d \int_0^1 r^{d-1} \int_0^{1-r^2} h^\beta \dint h \dint r
		\\
		&= \frac{\gamma c_{d+1,\beta}  \omega_d}{\beta+1} \int_0^1 r^{d-1} (1-r^2)^{\beta+1} \dint r.
	\end{align*}
	With the substitution $t=r^2$, the integral evaluates to $\frac{1}{2} B\left(\frac{d}{2},\beta+2\right)$. 
	After simplifying the expression it remains that indeed
	\(
		\mu_\beta(\Pi^\downarrow_{(\origin,1)})
		= \gamma \frac{c_{d+1,\beta}}{(\beta + \frac{d}{2} + 1) c_{d,\beta}}.
	\)
	The second case is handled similarly.
	\begin{align*}
		\mu_\beta^\prime(\Pi^\downarrow_{(\origin,-1)})
		&= \gamma c^\prime_{d+1,\beta} \int_{\R^d} \int_{\R_-} \1{h\le -1 - \norm{v}^2} (-h)^{-\beta}  \dint h \dint v
		\\
		&= \gamma c^\prime_{d+1,\beta} \omega_d \int_0^{\infty} r^{d-1} \int_{-\infty}^{-1-r^2} (-h)^{-\beta} \dint h \dint r
		\\
		&= \frac{\gamma c^\prime_{d+1,\beta} \omega_d}{\beta-1} \int_0^{\infty} r^{d-1} (1+r^2)^{-\beta+1} \dint r.
	\end{align*}	
	This time we use the substitution $t= \frac{1}{1+r^2}$ to find that the integral evaluates to $\frac{1}{2} B\left(\beta - \frac{d}{2} - 1, \frac{d}{2}\right)$.
	After simplifying, it remains that
	\(
		\mu_\beta^\prime(\Pi^\downarrow_{(\origin,-1)})
		= \gamma \frac{c^\prime_{d+1,\beta}}{(\beta - \frac{d}{2} - 1) c^\prime_{d,\beta}} 
	\).
\end{proof}

\begin{rem}
	Notice that changing the intensity parameter $\gamma$ is equivalent to rescaling $\eta^\maybeprime_\beta$ via the scaling operator $\phi$. As we showed, this scaling acts homogeneously on the measures and preserves paraboloids.
\end{rem}

The following corollary bounds the probability of an added point becoming a new vertex.

\begin{corollary}
	\label{added-vertex}
	Let $v\in \R^d$. Let $\kappa h > 0$.
	Then
	$$
	e^{-
		2^d c
		(\kappa h)^{\kappa \beta + \frac{d}{2} + 1}
	}
	\le
	\P{v\in \cF_0(\cD(\eta_\beta^\maybeprime \cup \{(v,h)\}))}
	\le
	2^d e^{
		-c(\kappa h)^{\kappa \beta + \frac{d}{2} + 1}
	} ,
	$$
	with $c = \frac{1}{2^d}\mu_\beta^\maybeprime(\Pi^\downarrow_{(\origin,\kappa)})$, calculated in \cref{improp}.
\end{corollary}

\begin{proof}
	The key observation is that for the new point to become a vertex, some paraboloid $\Pi^\downarrow_{(\tilde{v},\tilde{h})} \ni (v,h)$ needs to be empty of points of $\eta_\beta^\maybeprime$.
	Suppose this is the case for some $(\tilde{v},\tilde{h})\ne (v,h)$. Find an orthant $O\ni \tilde{v} - v$. Then
	$$\Pi^\downarrow_{(v,h)} \cap ((O+v)\times\R) \subset \Pi^\downarrow_{(\tilde{v},\tilde{h})}.$$
	In other words, the following is true.
	$$
	\left\{
		\eta_\beta^\maybeprime \cap 	\Pi^\downarrow_{(v,h)} = \emptyset
	\right\}
	\subset
	\left\{
		v\in \cF_0(\cD(\eta_\beta^\maybeprime \cup \{(v,h)\}))
	\right\}
	\subset 
	\bigcup_{O\text{ is an orthant}} \left\{
	\eta_\beta^\maybeprime \cap 	\Pi^\downarrow_{(v,h)} \cap ((O+v)\times\R) = \emptyset
	\right\}.
	$$
	After taking probabilities of these events and bounding the probability of a union by a sum, we are done.
\end{proof}

%% file: inputs/complementary.tex
\section{Complementary theorem}
\label{sec:complementary}

In this section we describe the distribution of the typical $\beta^\maybeprime$-Voronoi cell, conditioned on its number of facets to be fixed to a given value $k$.
The constructions of such cells entails the following ingredients:
\begin{enumerate}
	\item The underlying point process is \textit{Poisson}. 
	\item The \textit{number of points involved is fixed}. Here, we have $1$ point for the nucleus of the cell and $k$ points for the nucleus of adjacent cells.
	\item These points are in an \textit{acceptable configuration}. Here this means that, apart from the nucleus, each of them contributes to a facet of the cell. 
	\item The points determine an \textit{excluding region}. Here this is the Voronoi flower which is a union of paraboloids, see \Cref{e:flower} and \Cref{fig:parpro}. This flower should not contain any other points of the process. 
	\item There is a \textit{scaling} action for which the intensity measure of the Poisson process is \textit{homogeneous} and the set of acceptable configurations is \textit{stable}.
\end{enumerate}
We will use these ingredients to show \Cref{complementary} at the end of this section.
Theorems of this form and based on similar ingredients have been formulated multiple times in various contexts. They are called \textit{complementary theorems}. The first one was discovered by Miles in the context of Poisson flats \cite{miles1970,miles1971complementary,miles1974},
and later extended in various directions by
Møller and Zuyev \cite{moller_zuyev1996GammaType}, 
Zuyev \cite{zuyev_stopping_1999}, 
Cowan \cite{cowan2006Complementary}, 
Baumstark and Last \cite{baumstark_gamma_2009} 
and Bonnet, Calka and Reitzner \cite{bonnet_cells_2018}.

\subsection{Typical cell}
\label{s:typ-cell}

Recall that $\mathbb{K}$ denotes the set of convex bodies and let us consider the spaces
$$\widehat{\K} = (\R^d \times\R) \times \mathbb{K}
\qquad \text{and} \qquad
\widehat{\K}_\origin = \{((v,h),\convexbody) \in \widehat{\K} \mid v = \origin\},$$
equipped with the topology derived from the Hausdorff distance for convex bodies and Euclidean distance for points.
The \textbf{typical cell} of $\cV_\beta^\maybeprime$ is a random element $\smash{\extypcellany} = ((\origin,\smash{h_\typcellany}),\smash{\typcellany})$ of $\widehat{\K}_\origin$ which has distribution
$$\P{\smash{\extypcellany} \in D} 
= \frac{\E{\left|\{x = (v,h)\in \eta_\beta^\maybeprime \cap B \times \R \mid (x,\cell_{\eta_\beta^\maybeprime}(x))\in D +v \}\right|}}
{\E{\left|\cF_0(\cD_\beta^\maybeprime)\cap B\right|}},
\qquad D \subset \widehat{\K}_\origin,$$
for an arbitrary bounded Borel set $B \subset \R^d$ of positive Lebesgue measure, and where $D$ is a Borel set.
Note that $(x,\cell_{\eta_\beta^\maybeprime}(x))\in D$ implies that $\cell_{\eta_\beta^\maybeprime}(x)\ne \emptyset$ and therefore that $v\in \cF_0(\cD_\beta^\maybeprime)$.

With $\smash{\extypcellany}$, we omit $\beta$ to avoid overcrowded notation, implying it is fixed in each particular consideration.

Due to linearity of the expectation and the stationarity of the Poisson point process with respect to the first entry, any measurable set $B$ of positive Lebesgue measure can be chosen for this definition. We will often use the unit cube $B=[0,1]^d$ for this purpose and use the notation
$$\Lambda_\beta^\maybeprime\coloneqq \E{\left|\cF_0(\cD_\beta^\maybeprime)\cap [0,1]^d\right|} $$
for the denominator, which can also be interpreted as the vertex densities of the respective triangulations.
Note that due to the Mecke equation (see e.g.\ \cite[Theorem 4.1]{last_lectures_2017}), we have
\begin{align*}
	\Lambda_\beta^\maybeprime
	&= \int_{\kappa \R_+} \P{\origin\in \cF_0(\cD(\eta_{\beta;(\origin,h)}^\maybeprime))} \gamma c_{d+1,\beta}^\maybeprime (\kappa h)^{\kappa \beta} \dint h
	= \Theta(1) \int_{\kappa \R_+}  e^{-\Theta((\kappa h)^{\kappa \beta + d/2 + 1})} (\kappa h)^{\kappa \beta} \dint h ,
\end{align*}
where the approximation is due to \cref{added-vertex}. From this, we can deduce that $\Lambda_\beta^\maybeprime$ is finite (and positive) for any $\beta$ and $\beta'$ within their admissible ranges. In particular, the typical cell is well defined in both models.

The following lemma gives a direction for investigating typical cells.

\begin{lemma}
	\label{typical-distribution}
	Let $D\subset \widehat{\K}_\origin$ be a Borel set. Then
	$$
	\P{\smash{\extypcellany}\in D} 
	= \frac{\gamma c^\maybeprime_{d+1,\beta}}{\Lambda_\beta^\maybeprime}
		\int_{\kappa \R_+} 
		\P{((\origin,h),\cell_{\eta_{\beta;(\origin,h)}^\maybeprime}(\origin,h))\in D}
		(\kappa h)^{\kappa \beta} \dint h,
	$$	
	where $\eta_{\beta;x}^\maybeprime\coloneqq \eta_\beta^\maybeprime \cup \{x\}$ is the Palm distribution of $\eta_\beta^\maybeprime$ conditioned on containing element $x$.
\end{lemma}

\begin{proof}
	By definition
	$$
	\P{\smash{\extypcellany}\in D}
	= \frac{1}{\Lambda_\beta^\maybeprime}
	\E{
		\sum_{x=(v,h)\in \eta_\beta^\maybeprime \cap [0,1]^d \times \R}
		\1{(x,\cell_{\eta_\beta^\maybeprime}(x))\in D + v}
	}.
	$$	
	As a consequence of the Mecke equation (see e.g.\ \cite[Theorem 4.1]{last_lectures_2017}), it is equal to
	$$
	\frac{1}{\Lambda_\beta^\maybeprime}
	\int_{[0,1]^d \times \R} 
	\P{((v,h),\cell_{\eta_{\beta;(v,h)}^\maybeprime}(v,h))\in D + v}
	\mu_\beta^\maybeprime (\dint v, \dint h).
	$$
	Since the integrand does not depend on $v$, it can be further simplified to achieve the statement of the lemma.
\end{proof}

From the last lemma, we derive two corollaries below, which will be essential in the proof of \cref{concentration} on the concentration of maximal degree of a $\beta$-Delaunay graph in a growing window.
First we get some bounds on the distribution of a typical height, which is the height $\smash{h_\typcellany}$ of the typical cell $\smash{\extypcellany} = ((\origin,\smash{h_\typcellany}),\smash{\typcellany})$.

\begin{corollary}
	\label{typical-height}
	For any $\kappa \height >0$, we have
	\begin{align}
		\label{e:bounds-typ-height}
		C_1 \Gamma(z, 2^d c(\kappa \height)^{\kappa\beta + d/2 + 1})
		\le
		\P{\smash{h_\typcellany} \ge \height}
		\le
		C_2 \Gamma(z, c(\kappa \height)^{\kappa\beta + d/2 + 1})
	\end{align}
	for $c>0$ as in \cref{added-vertex},
	$z = \frac{\kappa\beta + 1}{\kappa \beta + d/2 + 1}$,
	$C_1=\frac{\gamma c^\maybeprime_{d+1,\beta}}{\Lambda^\maybeprime_\beta}\frac{1}{(2^d c)^z |\kappa \beta + d/2 +1|}$ and
	$C_2=\frac{\gamma c^\maybeprime_{d+1,\beta}}{\Lambda^\maybeprime_\beta}\frac{2^d}{c^z |\kappa \beta + d/2 +1|}$.
	
	In particular, there exist constants $c' > 0$ and $\height_0 \in \kappa \R_{>0}$, such that for any $\height>\height_0$,
	\begin{align}
		\label{e:UB-typ-height}
		\P{\smash{h_\typcellany} \ge \height}
		&\leq e^{- c' (\kappa \height)^{\kappa \beta + \frac{d}{2} + 1}} .
	\end{align} 
\end{corollary}
Equation \eqref{e:bounds-typ-height} gives bounds on the tail distribution of the typical height, i.e.\ on $\P{\smash{h_\typcellany} \ge \height}$ which goes to zero for $\height \uparrow \infty$ in $\beta$-Delaunay model ($\kappa = 1$), and for $\height \uparrow 0$ in $\beta'$-Delaunay model ($\kappa = -1$). 
\begin{proof}
	By \cref{typical-distribution} applied to $D = \{((\origin,h),\convexbody) \in \widehat{\K}_\origin \mid h \ge \height\}$, we have
	$$
		\P{\smash{h_\typcellany} \ge \height} 
		= \frac{\gamma c^\maybeprime_{d+1,\beta}}{\Lambda^\maybeprime_\beta}
		\int_{\kappa \R_+}
		\1{h\ge \height}
		\P{\origin\in \cF_0(\cD(\eta_{\beta;(\origin,h)}^\maybeprime))}
		(\kappa h)^{\kappa \beta} \dint h.
	$$
	The probability in the integrand is bounded by \cref{added-vertex}.
	This gives the bounds
	\begin{align*}
		\P{\smash{h_\typcellany} \ge \height}
		& \geq \frac{\gamma c^\maybeprime_{d+1,\beta}}{\Lambda^\maybeprime_\beta}
		\int_{\kappa \R_+} \1{h\ge \height} e^{-2^d c(\kappa h)^{\kappa \beta + d/2 + 1}} (\kappa h)^{\kappa \beta} \dint h ,
		\\
		\P{\smash{h_\typcellany} \ge \height}
		& \leq \frac{\gamma c^\maybeprime_{d+1,\beta}}{\Lambda^\maybeprime_\beta}
		\int_{\kappa \R_+} \1{h\ge \height} 2^d e^{-c(\kappa h)^{\kappa \beta + d/2 + 1}} (\kappa h)^{\kappa \beta} \dint h,
	\end{align*}
	where $c = \frac{1}{2^d}\mu_\beta^\maybeprime(\Pi^\downarrow_{(\origin,\kappa)})$ as in \cref{added-vertex}.
	The bounds \eqref{e:bounds-typ-height} follow from applying the substitutions $t = c_i (\kappa h)^{\kappa \beta + d/2 + 1}$ with $c_i$ being either $c_1 = 2^d c$ or $c_2 = c$, for which we have
	\[ h = \kappa \left(\frac{t}{c_i}\right)^{\frac{1}{\kappa \beta + d/2 + 1}} ,
	\qquad 
	\frac{\gamma c^\maybeprime_{d+1,\beta}}{\Lambda^\maybeprime_\beta} (\kappa h)^{\kappa \beta} \dint h 
	= C_i t^{z-1} \dint t,
	\qquad
	\1{h\ge \height} 
	= \1{t \ge c_i (\kappa \height)^{\kappa \beta + d/2 + 1}}.
	\]
	
	The last part is due to \cref{incomplete-gamma} which, for $h>0$ with $z\leq \max(1,c h^{\kappa\beta + d/2 + 1})$, gives
	\begin{align*}
		C_2 \Gamma(z, c h^{\kappa\beta + d/2 + 1})
		&\leq C_2 \max(z,1) e^{-c h^{\kappa \beta + d/2 + 1}} (c h^{\kappa \beta + d/2 + 1})^{z-1}
		\leq 
		C e^{-c h^{\kappa \beta + d/2 + 1}} h^{-d/2} ,
	\end{align*}
	for some finite constant $C$.
	In the $\beta$-Delaunay model ($\kappa=1$), $z\leq 1$ and thus the condition $z\leq \max(1,c h^{\kappa\beta + d/2 + 1})$ is satisfied for any $h>0$, and we get the desired bound by setting $c' = c$ and $\height_0 = C^{2/d}$.
	In the $\beta'$-Delaunay model ($\kappa=-1$), we need to be more cautious because $z\geq 1$ and $(- H)^{-d/2}\uparrow \infty$ as $H\uparrow 0$. 
	Here, we pick $c' \in (0,c)$ arbitrarily and set $\height_0<0$ sufficiently close to $0$, such that $c (-\height_0)^{- \beta + d/2 + 1} \geq z$ and $e^{(c-c')(- \height)^{\kappa \beta + d/2 + 1}} \geq C (-\height)^{-d/2}$ for any $\height \in (\height_0,0)$.
\end{proof}

This in turn allows us to investigate the tail distribution of the number of vertices of the triangulation in a given set.

\begin{corollary}
	\label{number-of-vertices}
	There is a $t_0>0$ such that for $t>t_0$ and for any Borel set $B\subset \R^d$,
	$$\P{|\cF_0(\cD_\beta^\maybeprime) \cap B| \ge tV_d(B)}
	\le e^{-tV_d(B)} + C V_d(B) e^{-ct^{\frac{\kappa \beta + d/2 + 1}{\kappa \beta + 1}}},$$
	for some constants $c,C\in(0,\infty)$ that do not depend on $B$ or $t$.
\end{corollary}

\begin{proof}
	Let us introduce an event
	$$E_\height^B \coloneqq \{ h\le \height \text{ for any } (v,h) \in \eta_\beta^\maybeprime \text{ with } v\in \cF_0(\cD_\beta^\maybeprime) \cap B \}.$$
	Observe that
	\begin{align}
		\notag
		\P{|\cF_0(\cD_\beta^\maybeprime) \cap B| \ge tV_d(B)}
		&=
		\P{|\cF_0(\cD_\beta^\maybeprime) \cap B| \ge tV_d(B) \;\&\; E_\height^B} + \P{|\cF_0(\cD_\beta^\maybeprime) \cap B| \ge tV_d(B) \;\&\; (E_\height^B)^c}
		\\&\le
		\P{\eta_\beta^\maybeprime(B\times(-\infty,\height])\ge tV_d(B)} + \P{(E_\height^B)^c}.
		\label{e:UB-two-terms}
	\end{align}
	For the first probability, recall that $\eta_\beta^\maybeprime(B\times(-\infty,\height]) \sim \Poisson\left(
		\frac{\gamma c^\maybeprime_{d+1,\beta}}{|\kappa \beta + 1|} (\kappa \height)^{\kappa \beta + 1} V_d(B)
		\right)$.
	For Poisson variables $X\sim \Poisson(\lambda)$ with mean $\lambda > 0$, there are well-known bounds on their tails, see for instance \cite[Lemma 3.4]{bonnet_concentration_2024} which states that
	$$\P{X\ge n} \le \left(e\frac{\lambda}{n}\right)^n,$$
	for any $n>\lambda$.
	Applied with
	$n=tV_d(B)$ and
	$\lambda=\frac{\gamma c^\maybeprime_{d+1,\beta}}{|\kappa \beta + 1|} (\kappa \height)^{\kappa \beta + 1} V_d(B)$
	with
	${\height= \kappa \left(\frac{|\kappa \beta + 1| }{e^2 \gamma c^\maybeprime_{d+1,\beta}} t\right)^\frac{1}{\kappa \beta + 1}}$ (so that $e\frac{\lambda}{n}=e^{-1}$), we get
	\begin{align}
		\label{e:UB-first-proba}
		\P{\eta_\beta^\maybeprime(B\times(-\infty,\height])\ge tV_d(B)}
		& \leq e^{-tV_d(B)} .
	\end{align}
	Notice that this way $\height$ increases with $t$, because the exponent $\frac{1}{\kappa \beta + 1}$ has the same sign as $\kappa$ (in front of the bracket) in both the $\beta$ and $\beta'$ cases.
	More precisely $\height\uparrow\infty$ in the $\beta$-Delaunay model and $\height\uparrow 0$ in the $\beta'$-Delaunay model, as $t\uparrow\infty$.
	
	For the second probability, observe that	
	\begin{align}
		\notag
		\P{(E_\height^B)^c}
		\le\E{\left|\left\{(v,h)\in \eta_\beta^\maybeprime \mid v\in \cF_0(\cD_\beta^\maybeprime) \cap B ,\  h\ge \height \right\}\right|}
		&= \Lambda_\beta^\maybeprime V_d(B) \P{\smash{h_\typcellany} \ge \height}
		\\ & \label{e:UB-EHBc}
		\leq \Lambda_\beta^\maybeprime V_d(B) e^{- c' (\kappa \height)^{\kappa \beta + \frac{d}{2} + 1}} ,
	\end{align}
	where the last inequality (due to \cref{typical-height}) holds for some $c'>0$ and as long as $H>H_0$, where $H_0$ is a constant independent of $t$ and $B$.
	Hence the result follow by inserting \eqref{e:UB-first-proba} and \eqref{e:UB-EHBc} in \eqref{e:UB-two-terms}.
\end{proof}

\subsection{Arcs and measures on the space of half-spaces}

Let $x=(v,h)\in \R^d\times\R$. Consider the cell
\begin{align}
	\label{e:cell-as-intersection}
	\cell_{\eta_{\beta;x}^\maybeprime}(x) = \left\{w\in \R^d \mid \norm{w-v}^2 + h \le \norm{w-v'}^2 + h' \text{ for all } x'=(v',h')\in \eta_\beta^\maybeprime \right\}
	= \bigcap_{x'\in \eta_\beta^\maybeprime} \bop(x,x'),	
\end{align}
where $\bop(x,x')$ denote bounding half-spaces defined by
\begin{align}
	\label{e:def-bop}
	\bop((v,h),(v',h')) \coloneqq \halfspace(2 v' - 2 v, h' + \norm{v'}^2 - h - \norm{v}^2) \in \overline{\halfspaces}.
\end{align}
The function $\bop(x,\cdot)$ maps $\R^d\times\R$ to $\overline{\halfspaces}$. For $x = (v,h)$ and $\halfspace = \halfspace(u,t)\in\halfspaces$ where $u\in\Sph^{d-1}$ and $t\in\R$, define
$$\arc(x,\halfspace)\coloneqq \{(v + \alpha u, h + (t - \langle v, u \rangle)^2 - (t - \langle v, u \rangle - \alpha)^2) \mid \alpha > 0\}.$$
It is an arc of a parabola (see \cref{fig:pi} for more geometric details of this construction). Additionally set
$$\arc(x,\emptyset) \coloneqq \{(v,h') \mid h' < h\}
\qquad \text{and} \qquad
\arc(x,\R^d) \coloneqq \{(v,h') \mid h' \ge h\}.$$
\begin{figure}[ht]
	\centering
	\begin{subfigure}{0.85\textwidth}
		\includegraphics[width=\textwidth]{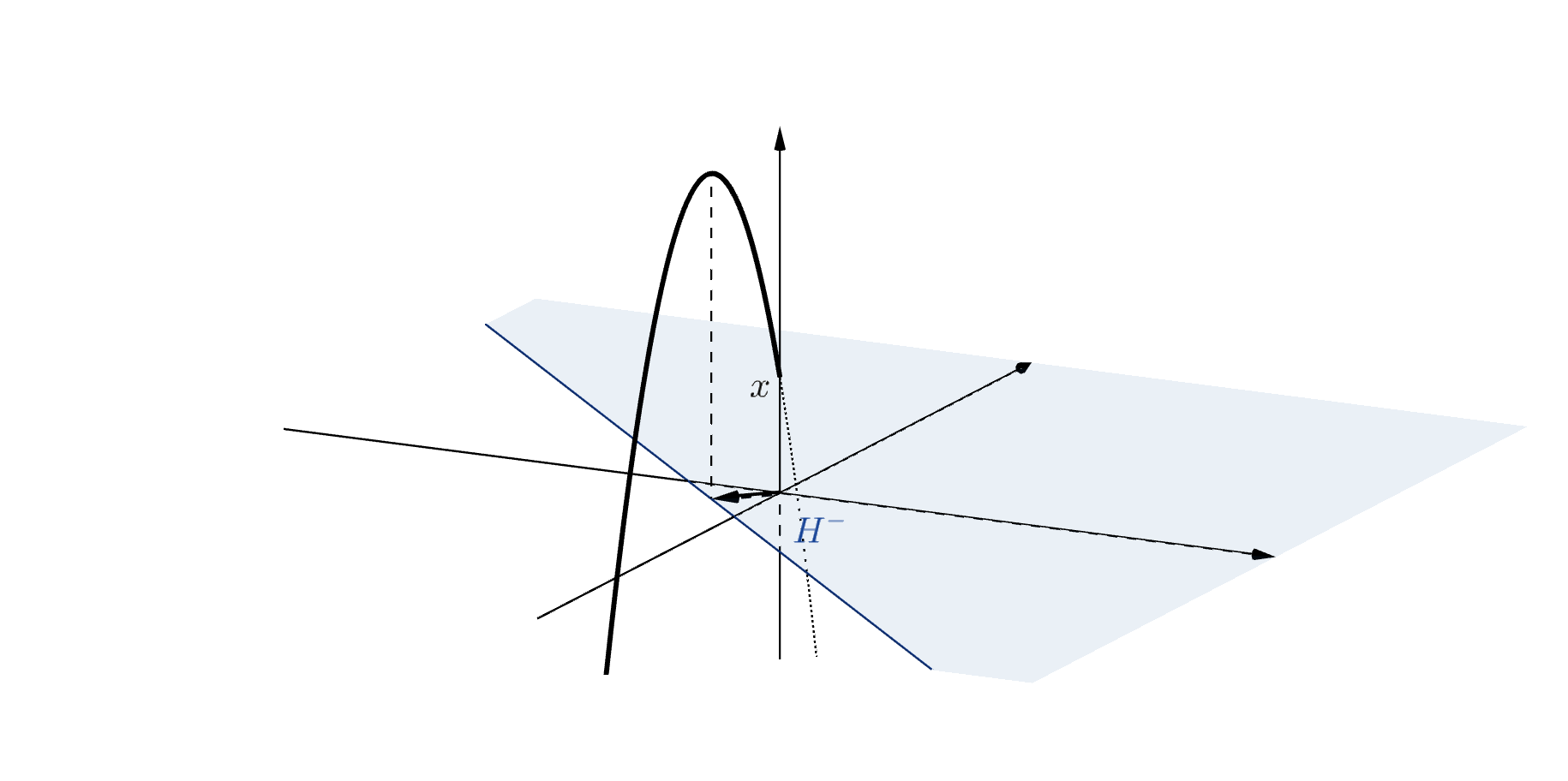}
		\caption{case where the apex of the parabola lies on $\arc(x,H^-)$}
	\end{subfigure}
	
	\begin{subfigure}{0.85\textwidth}
		\includegraphics[width=\textwidth]{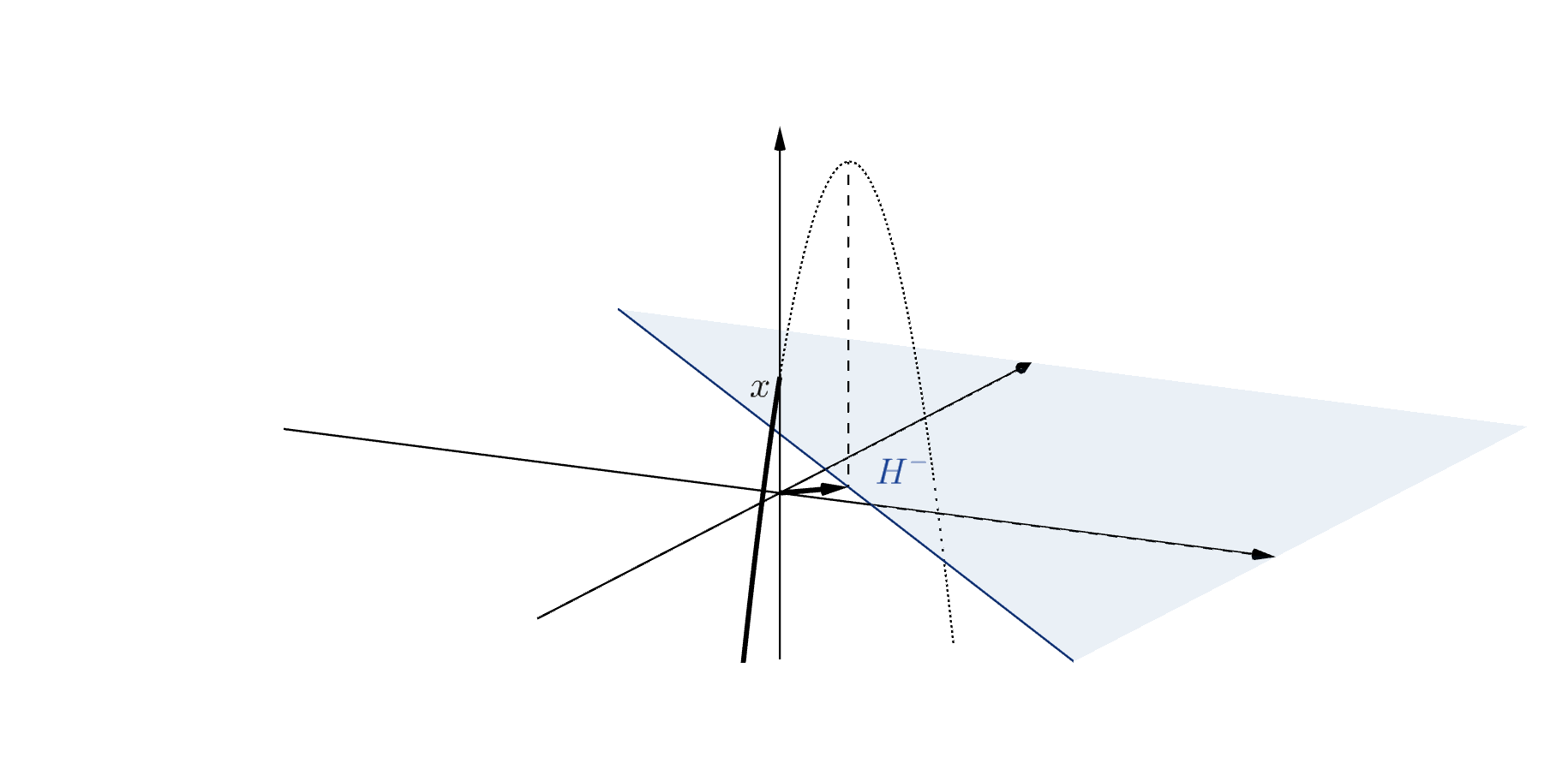}
		\caption{case where the apex of the parabola does not lie on $\arc(x,H^-)$}
	\end{subfigure}
	
	\caption{The curve $\arc(x,\halfspace)$ for some $x=(\origin,h)$, $h\in\R$, and $\halfspace=\halfspace(u,t)\in\halfspaces$. The bold vector is $tu$, with $t$ positive in the first figure and negative in the second. Geometrically the curve can be constructed as follows: (1) take the downwards parabola with apex above $tu$ containing $x$; (2) divide it into two arcs at $x$; (3) take the arc that has an intersection of finite (possibly zero) length with $H^-\times\R$.}
	\label{fig:pi}
\end{figure}%
The two functionals defined above are related as follows.
\begin{lemma}[Relation between $\bop$ and $\arc$, and equivariance of $\arc$]
	\label{pi-inverse-b}
	For any $x\in \R^d\times\R$ and $\halfspace \in\overline{\halfspaces}$, 
	\begin{align*}
		\arc(x,\halfspace) 
		& = \{x' \in \R^d \times \R \mid \bop(x,x') = \halfspace\} ,
	\end{align*}
	and for any $x\in \R^d\times\R$, $\halfspace \in\overline{\halfspaces}$ and $v\in\R^d$, 
	\begin{align*}
		\arc( v + \scale{c} x , v + \scale{c} \halfspace)
		& = v + \scale{c} (\arc(x,\halfspace)).
	\end{align*} 
\end{lemma}
\begin{proof}
	This follows by a straightforward calculation.
\end{proof}

By the last lemma, mapping $\eta_\beta^\maybeprime$ by $\bop(x,\cdot)$ gives rise to the following Poisson point process on $\overline{\halfspaces}$ (a.s.\ in $\halfspaces$):
\begin{equation}
	\label{e:tilde-mu}
	\widetilde{\eta}_{\beta;x}^\maybeprime \coloneqq \{\bop(x,x') \mid x'\in \eta_\beta^\maybeprime\}
	\quad \text{ of intensity measure} \quad
	\widetilde{\mu}_{\beta;x}^\maybeprime(\cdot) \coloneqq \mu_\beta^\maybeprime \left( \bigcup_{\halfspace \in \,\cdot} \arc(x,\halfspace ) \right).
\end{equation}
With this notation, we can reformulate \cref{typical-distribution} as follows.

\begin{lemma}
	\label{typical-distribution-reformulation}
	Let $D\subset \widehat{\K}_\origin$ be a Borel set. Then,
	$$
	\P{\smash{\extypcellany}\in D}
	= \frac{\gamma c^\maybeprime_{d+1,\beta}}{\Lambda_\beta^\maybeprime}
	\int_{\kappa \R_+} 
	\P{
		\left((\origin,h),\bigcap_{\halfspace \in \widetilde{\eta}_{\beta;(\origin,h)}^\maybeprime} \halfspace \right) \in D
	}
	(\kappa h)^{\kappa \beta} \dint h.
	$$
\end{lemma}
\begin{proof}
	Combining \eqref{e:cell-as-intersection} and \eqref{e:tilde-mu} gives that $\cell_{\eta_{\beta;x}^\maybeprime}(x) = \bigcap_{\halfspace \in \widetilde{\eta}_{\beta;x}^\maybeprime} \halfspace$.
	Then the statement follows from \cref{typical-distribution}.
\end{proof}

The following two lemmas give properties of the measure $\widetilde{\mu}_{\beta;x}^\maybeprime$, which we will use in the sequel.
\begin{lemma}[Homogeneity and translation invariance of $\widetilde{\mu}_{\beta;x}^\maybeprime$]
	\label{scaling-tilde}
	For any $x\in\R^d\times\R$, $c>0$, $v\in\R^d$ and measurable set $A\subset \halfspaces$,
	\begin{align}
		\label{e:scaling-tilde}
		\widetilde{\mu}_{\beta;v+\scale{c}x}^\maybeprime(v+cA) = c^{\kappa 2 \beta + d + 2}\widetilde{\mu}_{\beta;x}^\maybeprime(A) .
	\end{align}
\end{lemma}
\begin{proof}
	This is a direct consequence of the definition \eqref{e:tilde-mu} of $\widetilde{\mu}_{\beta;x}^\maybeprime$, the homogeneity and translation invariance of $\mu_\beta^\maybeprime$ and the equivariance of $\arc$, see \cref{improp-homo} of \cref{improp} and \cref{pi-inverse-b}, respectively.
\end{proof}

\begin{lemma}[Density of $\widetilde{\mu}_{\beta;(\origin,h)}^\maybeprime$]
	\label{density-on-half-spaces}
	Let $\kappa h > 0$.
	Let $U\subset \Sph^{d-1}$ and $T\subset \R$ be Borel sets. Let 
	\begin{equation}
		\label{e:hyperplaneUT}
		\halfspacesset 
		= \halfspacesset(U,T) 
		\coloneqq \{\halfspace(u,t) \mid u\in U,\ t\in T\}.
	\end{equation}
	Then 
	\begin{equation}
		\label{e:densitymutilde}
		\widetilde{\mu}_{\beta;(\origin,h)}^\maybeprime(\halfspacesset)
		= \sigma_{d-1}(U) \int_T m_{\beta;(\origin,h)}^\maybeprime(t) \dint t,
	\end{equation}	
	for some density function $m_{\beta;(\origin,h)}^\maybeprime$.
	Moreover, this function satisfies the following:
	\begin{enumerate}
		\item \label{density-on-half-spaces_item1} $m_{\beta;(\origin,c^2 h)}^\maybeprime(t) = c^{\kappa 2 \beta + d + 1} m_{\beta;(\origin,h)}^\maybeprime(\frac{t}{c})$
		\item \label{i:densityEstimates}
			\begin{enumerate}
				\item 
					\label{i:density-1} 
					$m_{\beta;(\origin,h)}(-t) 
					\leq m_{\beta;(\origin,h)}(t) 
					= \Theta\left( (h+t^2)^{\beta + \frac{d}{2} + \frac{1}{2}} \right)$ for $t\ge 0$, 
				\item 
					\label{i:density-2}
					$m_{\beta;(\origin,h)}^\prime (-t) 
					\leq m_{\beta;(\origin,h)}^\prime (t) 
					= \Theta\left((-h)^{\frac{d}{2}} (-h-t^2)^{-\beta + \frac12} \right)$ for $0\le t < \sqrt{-h}$, 
			\end{enumerate}
			with constants in $\Theta$'s independent of $h$.
	\end{enumerate}
\end{lemma}

\begin{proof}
	\Cref{density-on-half-spaces_item1} follows from the homogeneity \eqref{e:scaling-tilde} of $\widetilde{\mu}_{\beta;x}^\maybeprime$.
	For the rest of the proof, we assume without loss of generality that $h = \kappa$.
	Conveniently, the letter $h$ is now free for us to use as variable in integrals below.
	
	From \cref{pi-inverse-b} and the definition \eqref{e:tilde-mu} of the measures $\widetilde{\mu}_{\beta;x}^\maybeprime$, we have
	$$
	\widetilde{\mu}_{\beta;(\origin,\kappa)}^\maybeprime(\halfspacesset)
	= \int_{\R^d \times \R} \1{\bop((\origin,\kappa),x)\in \halfspacesset} \mu_\beta^\maybeprime(\dint x)
	= \int_{\R^d \times \kappa\R_+} \1{\frac{v}{\norm{v}} \in U} \1{\frac{h + \norm{v}^2 - \kappa}{2\norm{v}} \in T} \mu_\beta^\maybeprime(\dint v, \dint h) ,
	$$	
	with the second equality due to the definition \eqref{e:def-bop} of $\bop(\cdot)$.
	By definition of $\mu_\beta^\maybeprime$, see \eqref{def:mu_beta} and \eqref{def:mu_beta_prime}, and changing to polar coordinates $v = ru$, $r>0$, $u\in \Sph^{d-1}$, we get	
	\begin{align*}
		\widetilde{\mu}_{\beta;(\origin,\kappa)}^\maybeprime(\halfspacesset)
		&= \gamma c^\maybeprime_{d+1,\beta} \sigma_{d-1}(U)
		\int_{\kappa \R_+} \int_0^\infty
		\1{\frac{h + r^2 - \kappa}{2r} \in T} r^{d-1}
		(\kappa h)^{\kappa \beta}
		\dint r \dint h
		\\ &=
		2 \gamma c^\maybeprime_{d+1,\beta} \sigma_{d-1}(U)
		\int_T \int_0^\infty  \1{2tr - r^2 + \kappa \in \kappa \R_+}
		(\kappa(2tr - r^2 + \kappa))^{\kappa \beta} r^d
		\dint r \dint t,
	\end{align*}
	where the second equality is due to the change of variable $h = 2tr - r^2 + \kappa$, $\frac{\partial h}{\partial t} = 2r$.
	Therefore we have shown that \eqref{e:densitymutilde} holds with
	\begin{align*}
		m_{\beta;(\origin,\kappa)}^\maybeprime(t) 
		&= c \int_0^\infty  \1{1+\kappa t^2 -\kappa (r-t)^2  > 0}
		(1+\kappa t^2 -\kappa (r-t)^2 )^{\kappa \beta} r^d
		\dint r ,
	\end{align*}
	where $c= 2 \gamma c^\maybeprime_{d+1,\beta}$, and it only remains to establish \Cref{i:densityEstimates}.
	This expression can also be rearranged as follows, by the change of variable $q = \frac{r-t}{\sqrt{1+\kappa t^2}}$:
	\begin{align}
		\notag
		m_{\beta;(\origin,\kappa)}^\maybeprime(t) 
		&= c (1+\kappa t^2)^{\kappa \beta} \int_0^\infty  \1{\kappa \left( \frac{r-t}{\sqrt{1+\kappa t^2}} \right)^2 <1}
		\left( 1 - \kappa \left( \frac{r-t}{\sqrt{1+\kappa t^2}} \right)^2 \right)^{\kappa \beta} r^d
		\dint r 
		\\ 
		\label{e:m-prime-beta-eq}
		& = c (1+\kappa t^2)^{\kappa \beta + \frac{d}{2} + \frac{1}{2}} \int_{\frac{-t}{\sqrt{1+\kappa t^2}}}^\infty \1{\kappa q^2 < 1} \left( 1 - \kappa q^2 \right)^{\kappa \beta} \left(q + \frac{t}{\sqrt{1+\kappa t^2}}\right)^d \dint q .  
	\end{align}
	From this expression we see that $m_{\beta;(\origin,\kappa)}^\maybeprime(- t) \leq m_{\beta;(\origin,\kappa)}^\maybeprime(t)$ for any $t\ge 0$, which gives the first inequalities in \Cref{i:density-1} and \Cref{i:density-2}.
	It remains only to establish the approximations in \Cref{i:density-1} and \Cref{i:density-2}. We will consider the two cases $\kappa=1$ and $\kappa=-1$ separately.

	Let $\kappa=1$ and $t\ge 0$.
	From \eqref{e:m-prime-beta-eq}, we have
	\begin{align*}
		\int_0^1 (1-q^2)^\beta q^d \dint q
		\leq \frac{m_{\beta;(\origin,1)}^\maybeprime(t)}{c (1+t^2)^{\beta + \frac{d}{2} + \frac{1}{2}}}
		\leq \int_{-1}^1 (1-q^2)^\beta (q+1)^d \dint q .
	\end{align*}
	The left and right hand sides are positive and finite constants independent of $t$, which gives the desired approximation in \Cref{i:density-1}.
	
	Now let $\kappa = -1$ and $0 \le t < 1$.
	From \eqref{e:m-prime-beta-eq}, we have
	\begin{align*}
		m_{\beta;(\origin,-1)}^\prime(t) 
		&=  c (1-t^2)^{-\beta + \frac{d}{2} + \frac{1}{2}} \int_{\frac{-t}{\sqrt{1-t^2}}}^\infty (1 - q^2)^{-\beta} \left(q + \frac{t}{\sqrt{1-t^2}}\right)^d \dint q 
		\\
		&=  c (1-t^2)^{-\beta + \frac{1}{2}} \int_{\frac{-t}{\sqrt{1-t^2}}}^\infty (1 - q^2)^{-\beta} \left(\frac{\sqrt{1-t^2}}{t} q + 1 \right)^d \dint q .
	\end{align*}
	By the dominated convergence theorem, the last integral is a positive and and continuous function of $t$ on $[0,1)$ which converges to $\int_0^\infty (1 - q^2)^{-\beta} \dint q \in (0,\infty)$ as $t\uparrow 1$.
	The result follows.
\end{proof}

\subsection[Flower and Φ-content]{Flower and $\Phi$-content}
\label{s:flower}

Recall that in \eqref{e:flower} we defined the Voronoi flower $\flower_X(x)$ of \textit{a point $x \in \R^d \times \R$ in a set $X \subset \R^d\times \R$}. Below we will show that it depends only on $x$ and its Laguerre cell $\cell_X(x) \subset \R^d$. For that, we describe a more general notion of Voronoi flower, defined for any \textit{pair of a point in $\R^d\times\R$ and a convex body in $\R^d$}.
The direct relation between these two notions is given by \cref{lem:flowers-relation} below.

Let $\widehat{\convexbody}=(x,\convexbody)\in\widehat{\K}$. Define its \textbf{flower} as
\begin{equation}
	\label{e:flower2}
	\flower(\widehat{\convexbody}) = \flower(x,\convexbody)\coloneqq \bigcup_{\halfspace  \in\overline{\halfspaces} \,:\, \convexbody \not\subset \Int\halfspace} \arc(x,\halfspace ).
\end{equation}
To see that this notion indeed generalizes Voronoi flower, the key observation is the (perhaps geometrically intuitive) lemma below.

\begin{lemma}
	\label{par-in-par}
	Let $x\in\R^d\times\R$, $\halfspace \in\overline{\halfspaces}, w\in\R^d$. Then
	
	\begin{enumerate}
		\item \label{i:par-in-par1} If $w \notin \Int\halfspace $, then $\arc(x,\halfspace ) \subset \Pi^\downarrow_w(x)$.
		\item \label{i:par-in-par2} If $w \in \Int\halfspace $, then $\arc(x,\halfspace ) \cap \Pi^\downarrow_w(x) = \emptyset$.
	\end{enumerate}
\end{lemma}

\begin{proof}
	When $\halfspace $ is $\R^d$ or $\emptyset$, the claim is clearly true.
	Thus let $x = (v,h)$, $\halfspace  = \halfspace(u,t)$ with $u\in\Sph^{d-1}$, $t\in\R$.
	Any element of $\arc(x,\halfspace )$ can be expressed as
	$$(v',h') = (v + \alpha u, h + (t - \langle v, u \rangle)^2 - (t - \langle v, u \rangle - \alpha)^2)$$
	for some $\alpha>0$.
	Recall that $\Pi^\downarrow_w(x) = \Pi^\downarrow_{(w,h+\norm{w-v}^2)}$. By shifting the picture, it is enough to show that
	$(v'',h'') \in \Pi^\downarrow$ if and only if $w \notin \halfspace $, where $(v'',h'') \coloneqq (v',h') - (w,h+\norm{w-v}^2)$.
	This is in fact almost tautological. Observe that	
	$$\norm{v''}^2 = \norm{v-w}^2 + \alpha^2 + 2\alpha \langle u, v - w \rangle.$$
	At the same time	
	$$h'' = - \norm{v-w}^2 - \alpha^2 + 2\alpha(t - \langle v, u \rangle).$$
	Since $\alpha$ is positive, that means that indeed $h'' \le - \norm{v''}^2$ precisely when $\langle w, u \rangle \ge t$.
\end{proof}

\begin{corollary}
	\label{flower-paraboloids}
	For $\widehat{\convexbody} = (x,\convexbody)\in \widehat{\K}$, 
	$$\flower(\widehat{\convexbody}) = \bigcup_{w\in \convexbody} \Pi^\downarrow_w(x) = \bigcup_{w\in \Ext(\convexbody)} \Pi^\downarrow_w(x),$$
	where $\Ext(\convexbody)$ denotes the extreme points of the convex body $\convexbody$.
\end{corollary}

\begin{proof}
	Let $\halfspace \in \overline{\halfspaces}$. If $\convexbody \not\subset \Int\halfspace $, there exists $w\in \Ext(\convexbody)$ such that $w\notin \Int\halfspace $.
	By \cref{i:par-in-par1} of \cref{par-in-par}, $\arc(x,\halfspace ) \subset \Pi^\downarrow_w(x)$, and thus
	$$
		\flower(\widehat{\convexbody}) \subset \bigcup_{w\in \Ext(\convexbody)} \Pi^\downarrow_w(x).
	$$
	On the other hand, let $w\in \convexbody$ and $x'\in \Pi^\downarrow_w(x)$.
	Let $\halfspace = \bop(x,x')$.
	By \cref{pi-inverse-b} we have that $x'\in \arc(x,\halfspace )$.
	Thus, by \cref{i:par-in-par2} of \cref{par-in-par}, $w\notin \Int\halfspace $, so $\convexbody\not\subset \Int\halfspace $. Thus
	$$\bigcup_{w\in \convexbody} \Pi^\downarrow_w(x) \subset \flower(\widehat{\convexbody}).$$
	It remains to observe that clearly
	$\bigcup_{w\in \Ext(\convexbody)} \Pi^\downarrow_w(x) \subset \bigcup_{w\in \convexbody} \Pi^\downarrow_w(x)$.
\end{proof}

This, along with \cref{flower-above-vertices}, brings us to the following conclusion.
\begin{corollary}
	\label{lem:flowers-relation}
	For any locally finite set $X\subset \R^d \times \R$ satisfying the general position assumption \eqref{e:general-position} and any $x\in X$, we have
	\begin{align}
		\label{e:flower-relation} 
		\flower( x, \cell_X(x) ) 
		= \bigcup_{w\in \cF_0(\cell_X(x))} \Pi^\downarrow_w(x) 
		= \flower_X(x) .
	\end{align}
	Moreover, if $X = \eta_\beta^\prime$, then almost surely for any $x=(v,h)\in\eta_\beta^\prime$,
	\begin{align}
		\label{e:flower-below-hyperplane} 
		\flower ( x , \cell_{\eta_\beta^\prime}(x) ) 
		\subset \R^d \times (-\infty,0)
		\quad \text{and} \quad 
		\cell_{\eta_\beta^\prime}(x) 
		\subset \B_v^d(\sqrt{-h}) .
	\end{align}
\end{corollary}
\begin{proof}
	Equation \eqref{e:flower-relation} is a direct consequence of \cref{flower-paraboloids} and \cref{flower-above-vertices}. 

	Now, consider $x = (v,h)\in\eta_\beta^\prime$ with $\cell_{\eta_\beta^\prime}(x) \neq \emptyset$ (otherwise \eqref{e:flower-below-hyperplane} is trivial).
	Let $w\in \cF_0(\cell_{\eta_\beta^\prime}(x))$.
	By \eqref{e:flower-relation} and the definition \eqref{e:flower} of $\flower_X(x)$, we have $\Int (\Pi^\downarrow_w(x)) \cap \eta_\beta^\prime = \emptyset$, from which we get $\Pi^\downarrow_w(x) \subset \R^d \times (-\infty,0)$ and $\norm{w-v}^2 \le -h$. 
	This implies \eqref{e:flower-below-hyperplane}.
\end{proof}

For $\widehat{\convexbody}\in \widehat{\K}$, the measure of its flower in the chosen model will be called \textbf{$\Phi$-content} of $\widehat{\convexbody}$ and denoted by
\begin{equation}
	\label{e:content}
	\Phi_\beta^\maybeprime(\widehat{\convexbody}) \coloneqq \mu_\beta^\maybeprime(\flower(\widehat{\convexbody})).
\end{equation}
\begin{lemma}
	\label{phi-point}
	Let $v\in \R^d$ and $\kappa h > 0$. 
	If $\kappa = -1$ assume additionally that $\norm{v}^2 < -h$.
	Then
	$$
	\Phi_\beta^\maybeprime\left((\origin,h),\{v\}\right) = c (\kappa h + \kappa \norm{v}^2)^{\kappa \beta + \frac{d}{2} + 1}
	$$
	for some positive constant $c$ depending on $d$ and $\beta$.
\end{lemma}
\begin{proof}
	By \cref{flower-paraboloids},
	$$\flower\left((\origin,h),\{v\}\right) = \Pi^\downarrow_v((\origin,h)) = \Pi^\downarrow_{(v,h+\norm{v}^2)}.$$
	By \cref{improp} \cref{improp-para},
	$$\Phi_\beta^\maybeprime\left((\origin,h),\{v\}\right) = \mu_\beta^\maybeprime\left(\Pi^\downarrow_{v,h+\norm{v}^2}\right) = c (\kappa h + \kappa \norm{v}^2)^{\kappa \beta + \frac{d}{2} + 1}.$$
\end{proof}

\begin{lemma}
	\label{phi-in-polar}
	The $\Phi$-content of $\widehat{\convexbody}\in\widehat{\K}_\origin$ can be expressed as
	$$
		\Phi_\beta^\maybeprime((\origin,h),\convexbody)
		= \int_{\Sph^{d-1}} \int_{-\infty}^{\support(\convexbody,u)} m_{\beta;(\origin,h)}^\maybeprime(t) \dint t \sigma_{d-1}(\dint u) ,
	$$
	where $m_{\beta;(\origin,h)}^\maybeprime$ is defined in \cref{density-on-half-spaces} and $\support(\convexbody,u)$ is the support function of $\convexbody$ in direction $u$.
\end{lemma}
\begin{proof}
	Let $\widehat{\convexbody} = ((\origin,h),\convexbody)\in \widehat{\K}_\origin$.
	Applying equations \eqref{e:content}, \eqref{e:flower2}, and \eqref{e:tilde-mu} in this order gives
	\begin{align*}
		\Phi_\beta^\maybeprime(\widehat{\convexbody})
		&= \mu_\beta^\maybeprime(\flower(\widehat{\convexbody}))
		=\mu_\beta^\maybeprime\left(\bigcup_{\halfspace  \in\overline{\halfspaces} \,:\, \convexbody \not\subset \Int\halfspace} \arc((\origin,h),\halfspace )\right)
		= \widetilde{\mu}_{\beta;\origin}^\maybeprime\left(\{\halfspace \in \halfspaces \mid \convexbody \not\subset \Int\halfspace\}\right) .
	\end{align*}
	Finally, \eqref{e:densitymutilde} of \cref{density-on-half-spaces} gives the desired expression.
\end{proof}

\begin{rem}[Homogeneity of $\Phi_\beta^\maybeprime$]
	\label{scaling-Phi}
	From \cref{improp-homo} of \cref{improp} it directly follows that, for $c>0$ and $\widehat{\convexbody}\in \widehat{\K}$,
	$$
		\Phi_\beta^\maybeprime(\scale{c} \widehat{\convexbody}) = c^{\kappa 2 \beta + d + 2} \Phi_\beta^\maybeprime(\widehat{\convexbody}).
	$$	
\end{rem}

\subsection[Typical cell with n facets]{Typical cell with $n$ facets}

We want to describe the distribution of the $\Phi$-content and shape of the typical cell of $\cV_\beta^\maybeprime$. While difficult endeavor by itself, something can be said in the situation of a fixed number of facets.
For this, we consider the spaces
\begin{align*}
	\widehat{\cP}_{\origin,n} &\coloneqq \{((v,h),P) \mid v = \origin,\ h > 0,\ P\in \cP_n \},
	\\
	\widehat{\cP}_{\origin,n}^\prime &\coloneqq \{((v,h),P) \mid v = \origin,\ h < 0,\ P\in \cP_n,\ P \subset \sqrt{-h}\B^d\}.
\end{align*}
By \eqref{e:flower-below-hyperplane} the typical cell $\smash{\extypcellprime} = ((\origin,\smash{h_\typcellprime}),\smash{\typcellprime})$ of $\cV_\beta^\prime$, satisfies $\smash{\typcellprime} \subset \sqrt{-h_\typcellprime}\B^d$ a.s., and thus the typical cell of $\cV_\beta^\maybeprime$ a.s.\ lies in $\widehat{\cP}_{\origin,n}^\maybeprime$ for some $n\in\N$.
We will now equip the spaces $\cP_n$ with measures that will allow us to describe the distribution of the typical cell with $n$ facets.
We identify polytopes $P = \cap_{i=1}^n \halfspace_i$ with $n$ facets with the sets $\{\halfspace_1,\ldots, \halfspace_n\} $ of half-spaces defining them.
That means that we see the space $\cP_n$ as a subset of $(\halfspaces)^n / \mathfrak{S}_n$, the quotient of $(\halfspaces)^n$ by the symmetric group.
For any fixed point $x = (v,h)\in\R^d\times\R$, we define the measure
\begin{equation}
	\label{eq:polytope_measure}
	\widetilde{\mu}_{\beta;x;n}^\maybeprime \coloneqq \frac{1}{n!} (\widetilde{\mu}_{\beta;x}^\maybeprime)^{\otimes n},
\end{equation}
on $\cP_n$, where we recall that $\widetilde{\mu}_{\beta;x}^\maybeprime$ is the measure on $\halfspaces$ defined by \eqref{e:tilde-mu}.
More explicitly, for a Borel set $A\subset \cP_n$, we have
\begin{align*}
	\widetilde{\mu}_{\beta;x;n}^\maybeprime(A) 
	&= \frac{1}{n!}
	\int_{(\halfspaces)^n} \1{ \cap_{i=1}^n \halfspace_i \in A} (\widetilde{\mu}_{\beta;x}^\maybeprime)^{\otimes n} (\dint H^-_{1:n}).
\end{align*}

\begin{rem}[Homogeneity of $\widetilde{\mu}_{\beta;(\origin,h);n}^\maybeprime$]
	\label{scaling-tilde-n}
	From the homogeneity of $\widetilde{\mu}_{\beta;x}^\maybeprime$ (see \cref{scaling-tilde}), it follows that for measurable $A\subset \widehat{\cP}_{\origin,n}^\maybeprime$ and $c>0$,
	$$
	\widetilde{\mu}_{\beta;(\origin,c^2 h);n}^\maybeprime(cA) = c^{n(\kappa 2 \beta + d + 2)}\widetilde{\mu}_{\beta;(\origin,h);n}^\maybeprime(A).
	$$
\end{rem}

Equip $\widehat{\cP}_{\origin,n}^\maybeprime$ with measures
\begin{align}
	\label{e:def-nu-beta-o-n}
	\nu_{\beta;\origin;n}^\maybeprime (\cdot)
	\coloneq \frac{\gamma c^\maybeprime_{d+1,\beta}}{\Lambda_\beta^\maybeprime}
	\int_{\kappa \R_+}
	\int_{\cP_n} \1{((\origin,h),P) \in \cdot}
	\widetilde{\mu}_{\beta;(\origin,h);n}^\maybeprime (\dint P)
	(\kappa h)^{\kappa \beta} \dint h.
\end{align}
\begin{rem}[Homogeneity of $\nu_{\beta;\origin;n}^\maybeprime$]
	\label{scaling-nu}
	From the homogeneity of $\widetilde{\mu}_{\beta;(\origin,h);n}^\maybeprime$ (see \cref{scaling-tilde-n}), it follows that for measurable $A\subset \widehat{\cP}_{\origin,n}^{(\prime)}$ and $c>0$,
	$$
	\nu_{\beta;\origin;n}^\maybeprime(\scale{c} A) = c^{n(\kappa 2 \beta + d + 2) + \kappa 2 \beta + 2} \nu_{\beta;\origin;n}^\maybeprime(A) .
	$$
\end{rem}

Now we can write

\begin{lemma}
	\label{typ-n-fac}
	Fix $\beta$ and $\kappa$. Let $D\subset \widehat{\cP}_{\origin,n}^\maybeprime$ be measurable.
	$$
	\P{\smash{\extypcellany}\in D}
	= \int_{D}
	e^{-\Phi_\beta^\maybeprime(\widehat{P})}
	\nu_{\beta;\origin;n}^\maybeprime (\dint \widehat{P}).
	$$
\end{lemma}
\begin{proof}
	Recall from \cref{typical-distribution-reformulation} that 
	$$
	\P{\smash{\extypcellany}\in D}
	= \frac{\gamma c^\maybeprime_{d+1,\beta}}{\Lambda_\beta^\maybeprime}
	\int_{\kappa \R_+}
	\P{\left((\origin,h),\bigcap_{\halfspace \in \widetilde{\eta}_{\beta;(\origin,h)}^\maybeprime} \halfspace \right) \in D}
	(\kappa h)^{\kappa \beta} \dint h.
	$$
	Since the smallest tuple of half-spaces defining a polytope is unique up to a permutation,
	\[
		\1{\left((\origin,h),\bigcap_{\mathclap{\halfspace \in \widetilde{\eta}_{\beta;(\origin,h)}^\maybeprime}} \halfspace \right) \in D}
		=
		\sum_{\halfspace_{1:n}\in (\widetilde{\eta}_{\beta;(\origin,h)}^\maybeprime)^n}
		\1{\left((\origin,h),\bigcap_{i=1}^n \halfspace_i\right) \in D}
		\1{\forall \halfspace \in\widetilde{\eta}_{\beta;(\origin,h)}^\maybeprime,\ \bigcap_{i=1}^n \halfspace_i \subset \halfspace }.
	\]
	Taking the expectation, and applying the Multivariate Mecke equation (see e.g.\ \cite{last_lectures_2017}, Theorem 4.4) we get
	\begin{align*}
		&\P{\left((\origin,h),\bigcap_{\halfspace \in \widetilde{\eta}_{\beta;(\origin,h)}^\maybeprime} \halfspace \right) \in D}
		\\&=
		\int_{(\halfspaces)^n} \1{\left((\origin,h),\bigcap_{i=1}^n \halfspace_i\right) \in D}
		\P{\forall \halfspace \in\widetilde{\eta}_{\beta;(\origin,h)}^\maybeprime,\ \bigcap_{i=1}^n \halfspace_i \subset \halfspace }
		(\widetilde{\mu}_{\beta;(\origin,h)}^\maybeprime)^{\otimes n} (\dint \halfspace_{1:n}),
	\end{align*}
	which can be further simplified by employing the definition \eqref{eq:polytope_measure} of $\widetilde{\mu}_{\beta;(\origin,h);n}^\maybeprime$.
	This brings us finally to this expression:
	\begin{align*}
		\P{\smash{\extypcellany}\in D}
		&= \frac{\gamma c^\maybeprime_{d+1,\beta}}{\Lambda_\beta^\maybeprime}
		\int_{\kappa \R_+}
		\int_{\cP_n} \1{((\origin,h),P) \in D}
		\P{\forall \halfspace \in\widetilde{\eta}_{\beta;(\origin,h)}^\maybeprime,\ P \subset \halfspace }
		\widetilde{\mu}_{\beta;(\origin,h);n}^\maybeprime (\dint P)
		(\kappa h)^{\kappa \beta} \dint h 
		\\&= \int_{D} \P{\widetilde{\eta}_{\beta;(\origin,h)}^\maybeprime \cap \{\halfspace \in\overline{\halfspaces} \mid P \not\subset \halfspace \}  =\emptyset} \nu_{\beta;\origin;n}^\maybeprime (\dint \widehat{P}) .
	\end{align*}
	Finally, notice that the probability inside the integral can be written as
	$$
	\P{\widetilde{\eta}_{\beta;(\origin,h)}^\maybeprime \cap \{\halfspace \in\overline{\halfspaces} \mid P \not\subset \halfspace \} = \emptyset}
	= \P{\eta_\beta^\maybeprime\cap \flower((\origin,h),P) = \emptyset}
	= e^{-\Phi_\beta^\maybeprime((\origin,h),P)} ,
	$$
	where the first equality follows from the definition \eqref{e:tilde-mu} of $\widetilde{\mu}_{\beta;(\origin,h)}^\maybeprime$, \cref{pi-inverse-b} and the definition \eqref{e:flower2} of the $\flower((\origin,h),P)$, and the second equality follows from the definition \eqref{e:content} of $\Phi_\beta^\maybeprime$.
\end{proof}

\subsection{Homeomorphisms and complementary theorem}
\label{s:homeo_and_CT}

Let $\widehat{\convexbody}\in\widehat{\K}$ with $0 < \Phi_\beta^\maybeprime(\widehat{\convexbody}) < \infty$. As explained below for $\widehat{\convexbody} = \widehat{P} \in \widehat{\cP}_{\origin,n}^\maybeprime$, the element $\widehat{\convexbody}$ can be uniquely determined by its $\Phi$-content and its \textbf{shape}, i.e.\ its class in $\widehat{\K} / \R_+$ where scalars $c\in\R_+$ act on $\widehat{\K}$ by scaling $\scale{c}$.
For any element of $\widehat{\cP}_{\origin,n}^\maybeprime$, its shape is an element of
$\widehat{\cP}_{\origin,n}^\maybeprime / \R_+ \cong \cP_n^\maybeprime,$
where $\cP_n^\prime \coloneqq \{P\in\cP_n \mid P \subset \B^d\}$.
We can identify it with functions
\begin{equation}
	\label{e:shape}
	\begin{split}
		\shape^\maybeprime \colon \phantom{((\origin,h),P)} \llap{$\widehat{\cP}_{\origin,n}^\maybeprime$} & \rightarrow \cP_n^\maybeprime
		\\
		((\origin,h),P) & \mapsto \frac{1}{\sqrt{\kappa h}} P.
	\end{split}
\end{equation}
We consider the homeomorphisms
$$
\begin{aligned}
	\mathfrak{h}^\maybeprime_\beta \colon \widehat{\cP}_{\origin,n}^\maybeprime & \rightarrow (0,\infty)\times \mathcal{P}^\maybeprime_n
	\\
	\widehat{P} & \mapsto \left( \Phi_\beta^\maybeprime(\widehat{P}), \shape^\maybeprime(\widehat{P}) \right),
\end{aligned}
$$
and observe that the pushforward measures of $\nu_{\beta;\origin;n}^\maybeprime$ (defined in \eqref{e:def-nu-beta-o-n}) split nicely due to the scaling properties (see \cref{scaling-nu}):
$$\mathfrak{h}^\maybeprime_\beta(\nu_{\beta;\origin;n}^\maybeprime)((0,b) \times C) =
b^{n + \frac{\kappa 2 \beta+2}{\kappa 2 \beta+d+2}} \mathfrak{h}^\maybeprime_\beta(\nu_{\beta;\origin;n}^\maybeprime)((0,1) \times C).$$
This means that
\begin{equation}
	\label{e:splitting}
	\mathfrak{h}^\maybeprime_\beta(\nu_{\beta;\origin;n}^\maybeprime) 
	= \lambda_1^{(n + \frac{\kappa 2 \beta+2}{\kappa 2 \beta+d+2})} \otimes \tilde{\nu}_{\beta;\origin;n}^\maybeprime,
\end{equation}
where $\tilde{\nu}_{\beta;\origin;n}^\maybeprime (\cdot) \coloneq \mathfrak{h}^\maybeprime_\beta(\nu_{\beta;\origin;n}^\maybeprime)((0,1) \times \cdot)$ is a simplified notation, and where $\lambda_1^{(r)}$ is the $r$-homogeneous measure on $(0,\infty)$ defined by \eqref{e:lambdakr}.
With this observation we are finally ready to prove the main result of this section.
Recall that the typical cell is denoted $\smash{\extypcellany} = (\smash{x_\typcellany}, \smash{\typcellany})$.
\begin{theorem}[Complementary theorem]
	\label{complementary}
	Fix $\kappa$ and $\beta$. 
	Let $n\geq d+1$ be an integer.
	\begin{enumerate}
		\item For any Borel set of shapes $S \subset \cP_n^\maybeprime$,
		$$
		\P{
			\shape^\maybeprime(\smash{\extypcellany})\in S}
		=
		\Gamma\left(n + \frac{\kappa 2 \beta+2}{\kappa 2 \beta+d+2}+1\right)
		\int_{\widehat{\cP}_{\origin,n}^\maybeprime}
		\1{\shape^\maybeprime(\widehat{P})\in S}
		\1{\Phi_\beta^\maybeprime(\widehat{P}) < 1}
		\nu_{\beta;\origin;n}^\maybeprime (\dint \widehat{P}).
		$$		
		\item If we condition $\smash{\typcellany}$ to have $n$ facets, then
		\begin{enumerate}
			\item $\shape^\maybeprime(\smash{\extypcellany})$ and $\Phi_\beta^\maybeprime(\smash{\extypcellany})$ are independent random variables,
			\item $\Phi_\beta^\maybeprime(\smash{\extypcellany})$ is $\operatorname{Gamma}(n + \frac{\kappa 2 \beta+2}{\kappa 2 \beta+d+2},1)$-distributed, and
			\item $\shape^\maybeprime(\smash{\extypcellany})$ is distributed with probability measure
			$\tilde{\nu}_{\beta;\origin;n}^\maybeprime(\cdot)/\tilde{\nu}_{\beta;\origin;n}^\maybeprime(\cP_n^\maybeprime)$.
		\end{enumerate}
	\end{enumerate}
\end{theorem}

\begin{proof}
	Observe that by applying the measure's splitting \eqref{e:splitting} in the statement of \cref{typ-n-fac} we get
	\begin{align*}
		\P{\smash{\typcellany}\in\cP_n,\ \shape^\maybeprime(\smash{\extypcellany})\in S,\ \Phi_\beta^\maybeprime(\smash{\extypcellany})\in A}
		&= \int_S \int_A
		e^{-t}
		\lambda_1^{(n + \frac{\kappa 2 \beta+2}{\kappa 2 \beta+d+2})} (\dint t)
		\tilde{\nu}_{\beta;\origin;n}^\maybeprime (\dint P).
		\\&= 
		\tilde{\nu}_{\beta;\origin;n}^\maybeprime (S)
		\left(n + \frac{\kappa 2 \beta+2}{\kappa 2 \beta+d+2}\right)
		\int_A t^{n + \frac{\kappa 2 \beta+2}{\kappa 2 \beta+d+2} - 1} e^{-t} \dint t.
	\end{align*}	
	The first part of the theorem is then recovered by applying $A = (0,\infty)$ and using that by definition
	$$ \tilde{\nu}_{\beta;\origin;n}^\maybeprime (S)
	= \nu_{\beta;\origin;n}^\maybeprime\left(
		\left\{
			\widehat{P} \in \widehat{\cP}_{\origin,n}^\maybeprime \mid \shape^\maybeprime(\widehat{P})\in S,\ \Phi_\beta^\maybeprime(\widehat{P}) < 1
		\right\}
	\right).
	$$	
	For the second part, it follows from the same calculation that
	$$
	\P{\shape^\maybeprime(\smash{\extypcellany})\in S \mid \smash{\typcellany}\in\cP_n}
	=
	\frac{\tilde{\nu}_{\beta;\origin;n}^\maybeprime (S)}{\tilde{\nu}_{\beta;\origin;n}^\maybeprime(\cP_n^\maybeprime)}
	$$
	and
	\begin{equation*}
		\P{\Phi_\beta^\maybeprime(\smash{\extypcellany})\in A \mid \smash{\typcellany}\in\cP_n}
		=
		\frac{1}{\Gamma\left(n + \frac{\kappa 2 \beta+2}{\kappa 2 \beta+d+2}\right)}
		\int_A t^{n + \frac{\kappa 2 \beta+2}{\kappa 2 \beta+d+2} - 1} e^{-t} \dint t.
		\qedhere
	\end{equation*}
\end{proof}

\begin{rem}
	\label{complementary-variable-change}
	Fix $\kappa$ and $\beta$. Let $b>0$. Notice that	
	$$
	\int_{\widehat{\cP}_{\origin,n}^\maybeprime}
	\1{\Phi_\beta^\maybeprime(\widehat{P}) < b}
	\nu_{\beta;\origin;n}^\maybeprime (\dint \widehat{P})
	=
	b^{n + \frac{\kappa 2 \beta + 2}{\kappa 2 \beta + d + 2}}
	\int_{\widehat{\cP}_{\origin,n}^\maybeprime}
	\1{\Phi_\beta^\maybeprime(\widehat{P}) < 1}
	\nu_{\beta;\origin;n}^\maybeprime (\dint \widehat{P}).
	$$
\end{rem}

In the next corollary we give an alternative expression for the probability $\P{\smash{\typcellany}\in\cP_n,\ \shape^\maybeprime(\smash{\extypcellany})\in S}$.
\begin{corollary}
	\label{shape}
	Fix $\kappa$ and $\beta$. Let $S \subset \cP_n^\maybeprime$ be measurable. Then 
	$$
		\P{
			\shape^\maybeprime(\smash{\extypcellany})\in S}
		=
		\frac{\gamma c^\maybeprime_{d+1,\beta}}{\Lambda_\beta^\maybeprime}
		\frac{\Gamma\left(n + \frac{\kappa 2 \beta+2}{\kappa 2 \beta+d+2}\right)}{\kappa \beta + \frac{d}{2} + 1}
		\int_{S}
		\Phi_\beta^\maybeprime((\origin,\kappa),P)^{-n-\frac{\kappa 2 \beta+2}{\kappa 2 \beta+d+2}}
		\widetilde{\mu}_{\beta;(\origin,\kappa);n}^\maybeprime (\dint P).
	$$
\end{corollary}

\begin{proof}
	Recall from the definitions \eqref{e:scale} of the scaling operator $\scale{c}$ and \eqref{e:shape} of $\shape^\maybeprime$ that 
	$$\widehat{P} = \scale{\sqrt{\kappa h}} ((\origin,\kappa),\shape^\maybeprime(\widehat{P})),$$
	for any $\widehat{P} = ((\origin,h),P)\in \widehat{\cP}_{\origin,n}^\maybeprime$.
	By \cref{scaling-Phi}, it follows that
	$$\Phi_\beta^\maybeprime(\widehat{P}) = (\kappa h)^{\kappa \beta + \frac{d}{2} + 1} \Phi_\beta^\maybeprime((\origin,\kappa),\shape^\maybeprime(\widehat{P})).$$
	By the definition of $\nu_{\beta;\origin;n}^\maybeprime$,
	\begin{align*}
		&\int_{\widehat{\cP}_{\origin,n}^\maybeprime}
		\1{\shape^\maybeprime(\widehat{P})\in S}
		\1{\Phi_\beta^\maybeprime(\widehat{P}) < 1}
		\nu_{\beta;\origin;n}^\maybeprime (\dint \widehat{P})
		\\&=
		\frac{\gamma c^\maybeprime_{d+1,\beta}}{\Lambda_\beta^\maybeprime}
		\int_{\kappa \R_+}
		\int_{\cP_n}
		\1{\frac{1}{\sqrt{\kappa h}} P\in S}
		\1{\Phi_\beta^\maybeprime((\origin,h),P) < 1}
		\widetilde{\mu}_{\beta;(\origin,h);n}^\maybeprime (\dint P)
		(\kappa h)^{\kappa \beta} \dint h
		\\&=
		\frac{\gamma c^\maybeprime_{d+1,\beta}}{\Lambda_\beta^\maybeprime}
		\int_{\R_+}
		\int_{S}
		\1{\tilde{h}^{\kappa \beta + \frac{d}{2} + 1} \Phi_\beta^\maybeprime((\origin,\kappa),\tilde{P}) < 1}
		\tilde{h}^{n(\kappa \beta + \frac{d}{2} + 1)} \widetilde{\mu}_{\beta;(\origin,\kappa);n}^\maybeprime (\dint \tilde{P})
		\tilde{h}^{\kappa \beta} \dint \tilde{h}
		\\&=
		\frac{\gamma c^\maybeprime_{d+1,\beta}}{\Lambda_\beta^\maybeprime}
		\int_{S}
		\int_{\R_+}
		\1{\tilde{h} < \Phi_\beta^\maybeprime((\origin,\kappa),\tilde{P})^{-\frac{1}{\kappa \beta + \frac{d}{2} + 1}}}
		\tilde{h}^{n(\kappa \beta + \frac{d}{2} + 1) + \kappa \beta} 
		\dint \tilde{h} \;
		\widetilde{\mu}_{\beta;(\origin,\kappa);n}^\maybeprime (\dint \tilde{P}),
	\end{align*}
	where the second equality is obtained by the change of variables $h = \kappa \tilde{h}$ and $P = \sqrt{\tilde{h}} \tilde{P}$.	
	Calculating the inner integral in the right-hand side yields the result.
	
\end{proof}

%% file: inputs/lower-bounds.tex
\section[Lower bounds on the distribution of typical number of facets]{Lower bound on $\P{\smash{\typcellany}\in\cP_n}$}
\label{sec:lower}

In this section we establish lower bounds on $\P{\smash{\typcellany}\in\cP_n}$ for the $\beta$-Voronoi model (\Cref{lower-bound-beta}) and the $\beta'$-Voronoi model (\Cref{lower-bound-beta-prime}).
In the previous sections we established a set-up which allows us to represent these probabilities in the form of integrals, which unfortunately are not tractable due to the complexity of the geometric constraints involved.
However, in order to establish lower bounds it is enough to consider $\P{\smash{\typcellany}\in\cP_n,\ \shape^\maybeprime(\smash{\extypcellany})\in S^\maybeprime}$ for conveniently chosen sets of shapes $S^\maybeprime$.
These sets, constructed below (\crefrange{lem:delta-rho}{in-a-ball}) and illustrated by \Cref{fig:regular}, consist of polytopes which are, in a sense, almost regular and parameterized carefully so that geometric constraints do not interfere anymore with the calculations.
Then we will obtain the desired lower bounds by mean of \cref{shape} and elementary analysis.

In order to construct almost regular polytopes, our first order of business is to find an almost regular set of directions on a sphere. \Cref{lem:delta-rho} below formalizes that intuition. 

\begin{lemma}[{\cite[Lemma 6.1.7]{bonnet_thesis}}]
	\label{lem:delta-rho}
	For each $n\in\N$ one can choose $u_1,\ldots,u_n\in\Sph^{d-1}$, as well as $\delta_n > \rho_n>0$ such that
	\begin{itemize}
		\item \textbf{($\delta_n$-covering)} For any $u\in \Sph^{d-1}$, there exists $i\in[n]$ such that 
		\begin{equation}
			\label{eq:delta-cover}
			\norm{u - u_i} < \delta_n;
		\end{equation}
		\item \textbf{($2\rho_n$-packing)} For any $i,j\in[n]$ such that $i\ne j$,
		\begin{equation}
			\label{eq:2rho-packing}
			\norm{u_i - u_j} > 4\rho_n;
		\end{equation}
		\item \textbf{(order of growth)} As $n$ grows to infinity,
		\begin{equation}
			\label{eq:delta-rho-order}
			\delta_n = \Theta(n^{-\frac{1}{d-1}}) = \rho_n.
		\end{equation}
	\end{itemize}
\end{lemma}

For $i\in[n]$, consider the cap $S_{u_i}(\rho_n) = \Sph^{d-1} \cap \B^d_{u_i}(\rho_n)$ and set
\begin{align}
	\label{e:defBigH}
	\halfspacesset_i 
	\coloneqq \halfspacesset(S_{u_i}(\rho_n), (1-\rho_n^2, 1))
	= \left\{\halfspace(u,t) \mid u\in S_{u_i}(\rho_n),\ t\in (1-\rho_n^2, 1)\right\} .
\end{align}
We can now define the almost regular polytopes mentioned above.
These are the ones of the form $P = \cap_{i=1}^n \halfspace_i$, with $\halfspace_i\in r \halfspacesset_i$, for some fixed $r>0$, see \cref{fig:regular}.
\begin{figure}[ht]
	\centering
	\includegraphics[width=0.5\textwidth]{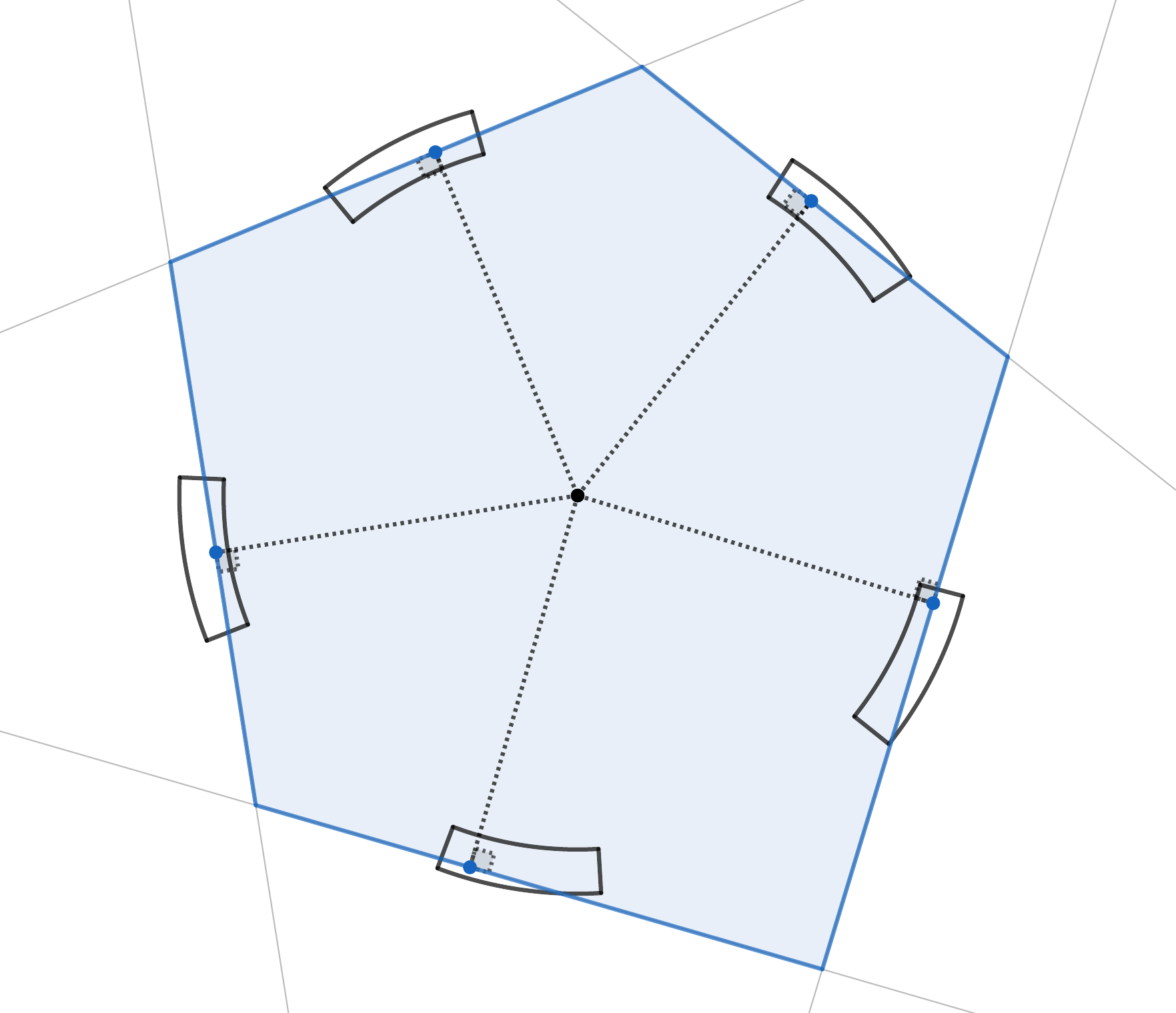}
	\caption{An almost regular polygon $P$ from \cref{nregular}}
	\label{fig:regular}
\end{figure}
The next lemma provides a bound on such $P$, and therefore shows that it is indeed a polytope, i.e.\ a \textit{bounded} polyhedron.
For this we set $n_0$ such that 
\[ \delta_n < 1/\sqrt{2} \text{ for all } n\ge n_0 ,\]
which is possible due to \Cref{lem:delta-rho}.

\begin{lemma}
	\label{in-a-ball}
	Let $n\geq n_0$ and $R>r>0$ be such that $\frac{r}{R} < 1 - 2 \delta_n^2$.
	Let $\halfspace_i\in r\halfspacesset_i$, $i\in[n]$.
	Then $P = \bigcap_{i=1}^n \halfspace_i \subset \B^d_\origin(R)$.
\end{lemma}

\begin{proof}
	It is enough to show this for $r=1$.
	Let $a>0$, $u\in\Sph^{d-1}$ be such that $au\in P$. From \eqref{eq:delta-cover} we know that $\norm{u-u_i}<\delta_n$ for some $i\in[n]$.
	Let $\halfspace_i = \halfspace(w,t)$, $w\in S_{u_i}(\rho_n)$, $t\in (1-\rho_n^2, 1)$.
	By the triangle inequality
	$$\norm{u-w} 
	\le \norm{u-u_i} + \norm{u_i - w}
	\le \delta_n + \rho_n 
	< 2 \delta_n.$$	
	From this, since $au \in \halfspace(w,t)$, $t<1$, we have
	\begin{equation*}
		\frac{1}{R} 
		< 1 - 2\delta_n^2 
		< 1 - \frac{1}{2} \norm{u-w}^2 
		= \langle u, w \rangle 
		< \frac{1}{a}
	\end{equation*}
	and therefore $a < R$ which yields the lemma.
\end{proof}

With the next lemma we show that the almost regular polytopes $\cap_{i=1}^n \halfspace_i$ have $n$ facets, i.e.\ that each of the half-spaces $\halfspace_i$ contributes to the intersection. 
\begin{lemma}
	\label{nregular}
	Let $n \ge n_0$ and $\halfspace_i\in \halfspacesset_i$, $i\in[n]$,
	$$P = \bigcap_{i=1}^n \halfspace_i \in \cP_n.$$
\end{lemma}

\begin{proof}
	By \cref{in-a-ball}, $P$ is bounded and therefore it is a polytope. It remains to check that it has exactly $n$ facets, i.e.\ that every half-space contributes to the intersection.	
	Let $w_1\in S_{u_1}(\rho_n)$, $w_2\in S_{u_2}(\rho_n)$. It is now enough to show that $w_1\in \Int\halfspace(w_2, 1-\rho_n^2)$.
	From \eqref{eq:2rho-packing} and triangle inequality we have
	$$\norm{w_1 - w_2} \ge 4\rho_n - 2\rho_n = 2\rho_n.$$
	Thus
	$$
		\langle w_1, w_2 \rangle
		= 1 - \frac{\norm{w_1 - w_2}^2}{2}
		\le 1 - 2 \rho_n^2,
	$$
	proving the claim.
\end{proof}

\begin{lemma}[Order of the $\Phi$-content of a ball]
	\label{content-of-a-ball}
	\begin{itemize}
		\item[]
		\item $\Phi_\beta\left((\origin,1),r\overline{\B^d}\right)
		= \Theta\left( 1 + r^{2\beta + d + 2} \right)
		= \Theta\left( (1 + r^2)^{\beta + \frac{d}{2} + 1} \right)$ for $r\ge 0$;
		\item $\Phi_\beta^\prime\left((\origin,-1),r\overline{\B^d}\right)
		= \Theta\left( (1-r)^{-\beta+\frac{3}{2}} \right)
		= \Theta\left( (1-r^2)^{-\beta+\frac{3}{2}} \right)$ for $0 \leq r <1$.
	\end{itemize}
\end{lemma}

\begin{proof}
	From the definitions of $\Phi$-content in \eqref{e:content}, $\flower(\cdot)$ in \eqref{e:flower2} and $\widetilde{\mu}_{\beta;(\origin,\kappa)}^\maybeprime$ in \eqref{e:tilde-mu} it follows that 
	\begin{equation}
		\begin{aligned}
			\label{e:march13}
		\Phi_\beta^\maybeprime \left((\origin,\kappa),r\overline{\B^d}\right) 
		&= \mu_\beta^\maybeprime\left(\flower \left((\origin,\kappa),r\overline{\B^d}\right) \right)
		\\&= \mu_\beta^\maybeprime\left(\bigcup_{r\overline{\B^d} \not\subset \halfspace \in\overline{\halfspaces}} \arc((\origin,\kappa),\halfspace) \right)	
		\\&= \widetilde{\mu}_{\beta;(\origin,\kappa)}^\maybeprime(\halfspacesset(\Sph^{d-1},(-\infty,r))) ,
		\end{aligned}
	\end{equation}
	where we recall that $\halfspacesset(\cdot,\cdot)$ is defined by \eqref{e:hyperplaneUT}. 
	Recall from \eqref{e:densitymutilde} that the measure $\R \supset T \mapsto \widetilde{\mu}_{\beta;(\origin,\kappa)}^\maybeprime(\halfspacesset(\Sph^{d-1},T))$ has a density $\omega_d \,m_{\beta;(\origin,h)}^\maybeprime(t)$, for which \Cref{i:densityEstimates} of \cref{density-on-half-spaces} gives estimates for $t\geq 0$.
	Using them we get
	\begin{align*}
		&\widetilde{\mu}_{\beta;(\origin,\kappa)}^\maybeprime(\halfspacesset(\Sph^{d-1},(0,r)))
		\\&=
		\begin{cases}
			\Theta \left( \int_0^r (1+t^2)^{\beta+\frac{d}{2}+\frac{1}{2}} \dint t \right)
			=\Theta \left( \int_0^r \max(1 , t)^{2\beta+d+1} \dint t \right)
			= \Theta \left( \min(1,r) + r^{2\beta+d+2} \right)
			& \kappa=1 ,
			\\
			\Theta \left( \int_0^r (1-t^2)^{-\beta+\frac{1}{2}} \dint t \right)
			= \Theta \left( \int_0^r (1-t)^{-\beta+\frac{1}{2}} \dint t \right)
			= \Theta \left( (1-r)^{-\beta+\frac{3}{2}} -1 \right)
			& \kappa=-1 .		
		\end{cases}
	\end{align*}
	To get the estimates of $\Phi_\beta^\maybeprime \left((\origin,\kappa),r\overline{\B^d}\right) $ we only need to add to the above the constant
	\[ 
	C^\maybeprime 
	\coloneqq \widetilde{\mu}_{\beta;(\origin,\kappa)}^\maybeprime(\halfspacesset(\Sph^{d-1},(-\infty,0)))
	= \Phi_\beta^\maybeprime \left((\origin,\kappa),\{\origin\}\right)>0,\] 
	where the strict inequality is due to \cref{phi-point}.
	Therefore
	\begin{align*}
		\Phi_\beta^\maybeprime \left((\origin,\kappa),r\overline{\B^d}\right) 
		&=
		\begin{cases}
			C + \Theta \left( \min(1,r) + r^{2\beta+d+1} \right)
			= \Theta(1+r^{2\beta+d+2})
			& \kappa=1 ,
			\\
			C' + \Theta \left( (1-r)^{-\beta+\frac{3}{2}} -1 \right)
			= \Theta \left((1-r)^{-\beta+\frac{3}{2}}\right)
			& \kappa=-1 .		
		\end{cases}
	\end{align*}
\end{proof}

\begin{rem}
	\label{phi-comparable}
	In fact, for the $\beta$-Voronoi model a stronger statement can be made. For $\widehat{\convexbody} = ((\origin,h),\convexbody)\in \widehat{\K}_\origin$,
	$$\Phi_\beta(\widehat{\convexbody}) = \Theta\left( (h + \max_{v\in \convexbody} \norm{v}^2)^{\beta + \frac{d}{2} + 1} \right).$$	
	This can be seen from the fact that $\{\argmax_{v\in \convexbody} \norm{v}\} \subset \convexbody \subset (\max_{v\in \convexbody} \norm{v}) \overline{\B^d}$ using \cref{phi-point} and \cref{content-of-a-ball} with \cref{scaling-Phi}.
\end{rem}
	
Now we can move to the bounds.

\begin{theorem}[Lower bound for the $\beta$-Voronoi model]
	\label{lower-bound-beta}
	Let $\kappa = 1$ and $\beta > -1$. Then there exists a constant $c>0$ such that for all integer $n\ge d+1$,
	$$\P{\smash{\typcell}\in\cP_n} > (c n^{-\frac{2}{d-1}})^n.$$
\end{theorem}

\begin{proof}
	Consider the set of shapes
	$S = \{P = \bigcap_{i=1}^n \halfspace_i \mid \halfspace_i\in \halfspacesset_i,\ i\in[n]\}.$
	By \Cref{nregular}, $S\subset \cP_n$.
	By the definition of $\widetilde{\mu}_{\beta;(\origin,1);n}$ (see \eqref{eq:polytope_measure}) the expression for $\P{\smash{\typcell}\in\cP_n,\ \shape(\smash{\extypcell})\in S}$ in \cref{shape} simplifies to
	$$
	c_1
	\Gamma\left(n + \frac{2 \beta+2}{2 \beta+d+2}\right)
	\int_{\halfspacesset_1 \times \cdots \times \halfspacesset_n}
	\Phi_\beta^\maybeprime((\origin,1),\bigcap_{i=1}^n \halfspace_i)^{-n-\frac{2 \beta+2}{2 \beta+d+2}}
	\widetilde{\mu}_{\beta;(\origin,1)}^{\otimes n} (\dint \halfspace_1, \ldots, \dint \halfspace_n).
	$$
	for some $c_1>0$.
	Recall that $\Gamma\left(n + \frac{2 \beta+2}{2 \beta+d+2}\right)=\left(\Omega(n)\right)^n$.
	Moreover, when $n$ is sufficiently large, for any $P\in S$, by \cref{in-a-ball}, $P\subset \B_\origin^d(2)$ making $\Phi_\beta\left((\origin,1),P\right)$ bounded by the constant $\Phi_\beta\left((\origin,1),\B_\origin^d(2)\right)$.
	Thus,
	$$\P{\smash{\typcell}\in\cP_n}
	\ge
	\P{\shape(\smash{\extypcell})\in S}
	>
	(c_2\,n\,\widetilde{\mu}_{\beta;(\origin,1)}(\halfspacesset_1))^n
	$$
	for some $c_2>0$. 
	By \eqref{e:densitymutilde} of \cref{density-on-half-spaces} and the definition \eqref{e:defBigH} of $\halfspacesset_1$, we have
	\[
		\widetilde{\mu}_{\beta;(\origin,1)}(\halfspacesset_1)
		= \sigma_{d-1}(S_{u_1}(\rho_n)) \int_{1-\rho_n^2}^1 m_{\beta;(\origin,1)}(t) \dint t
		= \Theta\left(\rho_n^{d-1} \int_{1-\rho_n^2}^1 (1+t^2)^{\beta+\frac{d}{2}+\frac{1}{2}} \dint t\right)
		= \Theta( \rho_n^{d+1} )
		= \Theta( n^{-\frac{d+1}{d-1}} ) ,
	\]
	where the first approximation is due to \cref{i:density-1} of \cref{density-on-half-spaces} and the last approximation is due to the order relation $\rho_n = \Theta(n^{-\frac{1}{d-1}})$ in \eqref{eq:delta-rho-order}.
	The bound follows.
	Note that it is enough to show this only for sufficiently large $n$, as then we can adjust the constant $c$ so that the same inequality holds for all values of $n$ where this probability is non-zero.
\end{proof}

\begin{theorem}[Lower bound for the $\beta'$-Voronoi model]
	\label{lower-bound-beta-prime}
	Let $\kappa = -1$, $\beta > \frac{d}{2} + 1$. Then there exists a constant $c>0$ such that for all integer $n \ge d+1$,
	$$\P{\smash{\typcellprime}\in\cP_n} > c^{n}.$$
\end{theorem}

\begin{proof}
	Set $r_n = 1 - 4 \delta_n^2$ and $R_n = 1 - \delta_n^2$. We assume that $n$ is large enough so that both are positive.
	Define a set of shapes
	$S^\prime = \{P = \bigcap_{i=1}^n \halfspace_i \mid \halfspace_i\in r_n\halfspacesset_i,\ i\in[n]\}.$
	By \cref{in-a-ball}, $P\subset \B_\origin^d(R_n)$ for any $P\in S^\prime$.
	Combined with \Cref{nregular} we therefore have $S^\prime\subset \cP_n^\prime$, where we recall that $\cP_n^\prime = \{P\in\cP_n \mid P \subset \B^d\}$ as defined at the beginning of \Cref{s:homeo_and_CT}.
	By applying \cref{shape} and using the same logic as in the proof of \cref{lower-bound-beta}, we get
	$$
	\P{\smash{\typcellprime}\in\cP_n}
	\ge
	\P{\shape(\smash{\extypcellprime})\in S^\prime}
	>
	\left(c_1\,n\,\frac{\widetilde{\mu}_{\beta;(\origin,-1)}^\prime(r_n\halfspacesset_1)}{\Phi_\beta^\prime\left((\origin,-1),R_n \overline{\B^d}\right)^{1+O(\frac{1}{n})}}\right)^n ,
	$$
	for some $c_1>0$.
	By \cref{density-on-half-spaces}, \cref{content-of-a-ball} and \cref{lem:delta-rho}, there exist $c_2, c_3, c_4>0$, such that
	$$
	\frac{\widetilde{\mu}_{\beta;(\origin,-1)}^\prime(r_n\halfspacesset_1)}{\Phi_\beta^\prime\left((\origin,-1),R_n \overline{\B^d}\right)^{1+O(\frac{1}{n})}}
	> c_2 \frac{\delta_n^{d+1 + 2(-\beta + \frac12) }}{\delta_n^{2(-\beta + \frac32)+O(\frac{1}{n})}}
	= c_2 \delta_n^{d-1+O(\frac{1}{n})}
	> c_3 n^{-1+O(\frac{1}{n})}
	> c_4 n^{-1}.
	$$
	The desired bound follows immediately with $c=c_1 c_4$.
	Again, we can adjust the constant $c$ so that the inequality holds for all values of $n$ where this probability is non-zero.
\end{proof}

%% file: inputs/upper-bound.tex
\section[Upper bound on the distribution of typical number of facets: β-Voronoi]{Upper bound on $\P{\smash{\typcell}\in\cP_n}$: $\beta$-Voronoi}
\label{sec:upper}

In this section we restrict ourselves to the $\beta$-Voronoi setting, i.e.\ $\kappa=1$, $\beta>-1$. 
We prove an upper bound on the tail distribution of the facet number of a typical cell.
The overall strategy follows closely the approach developed in \cite{bonnet_cells_2018} for Poisson hyperplane tessellations.
More precisely, the proof proceeds in three steps:
(i) show that a generic $n$-facet polytope admits many facets whose removal only slightly perturbs its size functional,
(ii) use a permutation argument to reduce the analysis to configurations where the last facet is one of these “good” facets,
and (iii) integrate out the last half-space to obtain a recursive bound.

While this scheme is the same, its implementation in the $\beta$-setting requires several non-trivial modifications.
In particular, the role of the classical $\Phi$-content is played here by $\Phi_\beta((\origin,h),P)$,
whose homogeneity and dependence on the height variable $h$ necessitate new estimates.
This affects both the geometric step (control of the variation of $\Phi_\beta$ under facet removal)
and the analytic step (integration with respect to the induced half-space measure $\mu_{\beta;(\origin,h)}$).
The lemmas below follow the structure of \cite{bonnet_cells_2018}, but incorporate these adaptations.

For ease of writing in this section we will denote by $$\alpha \coloneqq 2 \beta + d + 2$$ the order of homogeneity of $\mu_\beta$ (see \cref{improp}).
The main result of this section is the following theorem.

\begin{theorem}[Upper bound for the $\beta$-Voronoi model]
	\label{upper-bound-beta}
	There exists a positive constant $C$ such that
	\[
	\P{\smash{\typcell}\in\cP_n} 
	\leq C n^{-\frac{2}{d-1}} \P{\smash{\typcell}\in\cP_{n-1}} 
	\]	
	for all $n\geq d+1$.
	In particular there exists a positive constant $C$ such that
	\[
	\P{\smash{\typcell}\in\cP_n} 
	\leq (C n^{-\frac{2}{d-1}})^{n} 
	\]	
	for all $n\geq d+1$.
\end{theorem}

Note that the second part of the theorem follows directly from iterating the bound of the first part and the fact that $n! \geq (n/e)^n $. The rest of the section is dedicated to prove the first part of the theorem.

We start by preparing some tools. Let $h\ge 0$. For $\widehat{\convexbody} = ((\origin,h),\convexbody) \in \widehat{\K}_\origin$ denote by 
\begin{equation}
	\label{e:R(K)}
	R(\widehat{\convexbody}) = R((\origin,h),\convexbody)\coloneqq \sqrt{h + \max_{v\in \convexbody}\norm{v}^2},
	\quad \text{ and } \quad
	R(\convexbody) \coloneqq R((\origin,0),\convexbody) = \max_{v\in \convexbody}\norm{v}.
\end{equation}
The last quantity is simply the radius of the smallest ball centered at the origin and containing $\convexbody$.
Clearly $R((\origin,h),\convexbody)\ge R(\convexbody)$ for all $h\in\R$, $\convexbody\in\K$.
By \cref{phi-comparable}, there is a constant
$c_\Phi 
> 0$ such that
\begin{equation}
	\label{cphi}
	R(\widehat{\convexbody}) \le c_\Phi \Phi_\beta(\widehat{\convexbody})^\frac{1}{\alpha}
\end{equation}
for all $\widehat{\convexbody} \in \widehat{\K}_\origin$.
At the same time, by \cref{density-on-half-spaces}, there exists a constant $c_m>0$ such that, for any $h>0$ and any $t\in\R$,
\begin{equation}
	\label{cm}
	m_{\beta;(\origin,h)} (t) \le c_m (h+t^2)^\frac{\alpha - 1}{2}.
\end{equation}
The following lemma bounds the possible variation in $\Phi$-content for a localized variation of the body. This will be especially useful when we compare polytopes sharing a large proportion of their facets, i.e.\ of the form $\cap_{i\in I} \halfspace_i$ and $\cap_{j\in J}^n \halfspace_j$ with $I\cap J$ being a significant subset of both $I$ and $J$.
To measure the local variation between two convex bodies, we will use the difference between their support functions $\support(\cdot,\cdot)$, defined in \eqref{e:support}.
\begin{lemma}
	\label{lem:contPhi}
	Let
	$ U \subset \Sph^{d-1} $,
	$ \convexbody \subset \biggerconvexbody \in \mathbb{K} $,
	$h>0$
	be such that,
	for any
	$u\in\Sph^{d-1}$,
	\begin{equation}
		\label{eq:localvariation}
		0
		\leq \support(\biggerconvexbody,u) - \support(\convexbody,u) 
		\leq \1{ u \in U } \delta ,
	\end{equation} 
	where
	$\delta =
	\delta' c_\Phi \Phi_\beta((\origin,h),\convexbody)^{\frac{1}{\alpha}}
	> 0
	$, with $c_\Phi$ as in \eqref{cphi}.
	Then
	\[
	\Phi_\beta((\origin,h),\biggerconvexbody) - \Phi_\beta((\origin,h),\convexbody)
	\leq
	c_m c_\Phi^\alpha
	\delta'
	(1 + \delta')^{\alpha - 1}
	\Phi_\beta((\origin,h),\convexbody)
	\sigma_{d-1} (U),
	\]
	with $c_m$ as in \eqref{cm}.
\end{lemma}

\begin{proof}
	By \cref{phi-in-polar},
	\begin{align*}
		\Phi_\beta((\origin,h),\biggerconvexbody) - \Phi_\beta((\origin,h),\convexbody) 
		&=
		 \int_{\Sph^{d-1}}
			\int_{\support(\convexbody,u)}^{\support(\biggerconvexbody,u)}
			m_{\beta;(\origin,h)}(t)
			\dint t \sigma_{d-1}(\dint u)
		\\
		&\leq  
		\int_U
		(\support(\biggerconvexbody,u) - \support(\convexbody,u))
		c_m
		\left(h + \support(\biggerconvexbody,u)^2\right)^\frac{\alpha - 1}{2}
		\sigma_{d-1}(\dint u) ,
	\end{align*}  
	where the inequality follows from \eqref{cm} and \eqref{eq:localvariation}.	
	Let us bound uniformly $h + \support(\biggerconvexbody,u)^2$ for $u\in U$.
	By the hypothesis,
	$$h + \support(\biggerconvexbody,u)^2 \le h + (\support(\convexbody,u)+\delta)^2 =  h + \support(\convexbody,u)^2 + 2\delta \support(\convexbody,u) + \delta^2,$$
	Recall that, by \eqref{e:R(K)} and \eqref{cphi}, we have $\support(\convexbody,u) \le \sqrt{h + \support(\convexbody,u)^2} \le R((\origin,h),\convexbody) \le c_\Phi \Phi_\beta((\origin,h),\convexbody)^\frac{1}{\alpha}$. This translates to
	\begin{align*}
		h + \support(\biggerconvexbody,u)^2
		&\le h + \support(\convexbody,u)^2 + 2\delta \sqrt{h + \support(\convexbody,u)^2} + \delta^2
		\\&= \sqrt{h + \support(\convexbody,u)^2}^2 + 2\delta' c_\Phi \Phi_\beta((\origin,h),\convexbody)^{\frac{1}{\alpha}} \sqrt{h + \support(\convexbody,u)^2} + (\delta' c_\Phi \Phi_\beta((\origin,h),\convexbody)^{\frac{1}{\alpha}})^2
		\\&\le (1 + 2\delta' + \delta'^2) c_\Phi^2 \Phi_\beta((\origin,h),\convexbody)^{\frac{2}{\alpha}}
		\\&= (1 + \delta')^2 c_\Phi^2 \Phi_\beta((\origin,h),\convexbody)^{\frac{2}{\alpha}}.
	\end{align*}	
	And now we can plug it into the initial calculation to get
	\begin{align*}
		\Phi_\beta((\origin,h),\biggerconvexbody) - \Phi_\beta((\origin,h),\convexbody)
		&\le 
		\delta
		c_m 
		\left((1 + \delta')^2 c_\Phi^2 \Phi_\beta((\origin,h),\convexbody)^{\frac{2}{\alpha}}\right)^\frac{\alpha - 1}{2}
		\sigma_{d-1} (U)
		\\
		&=
		c_m c_\Phi^\alpha
		\delta'
		(1 + \delta')^{\alpha - 1}
		\Phi_\beta((\origin,h),\convexbody)
		\sigma_{d-1} (U).
	\end{align*}	
	That is precisely the bound we claimed.
\end{proof}

We will also use the following statement as is.

\begin{lemma}[{\cite[Corollary 2.8]{reisner_dropping_2001}}]
	\label{lem:RSWddependent}
	There exist constants $ N_1 $ and $ C_1 $, depending on $d$, such that the following holds.  
	For any integer $ n > N_1 $ and any $ \convexbody \in \mathbb{K} $, there exists a polytope $Q \supset \convexbody$ with $n$ facets such that
	\[
	\distH (\convexbody , Q) < C_1 R(\convexbody) n^{-\frac 2{d-1}} ,
	\]
	where $R(\convexbody)$ is defined in \eqref{e:R(K)}.
\end{lemma}

When $\convexbody$ is a polytope, the enveloping polytope $Q$ can be chosen to be $\convexbody$ with some facets deleted. Let us formalize this.
For $P = \bigcap_{i=1}^n \halfspace_i\in\cP_n$ and $I\subset [n]$, we set
\[ 
	P_I 
	:= \bigcap_{i\in I} \halfspace_i .
\]
Recall that a polytope is \textbf{simple} if each vertex is the intersection of exactly $d$ facets.
The following lemma, identical in spirit to \cite[Lemma~3.2]{bonnet_cells_2018},
allows us to remove a positive fraction of facets while keeping a good control of the polytope.
\begin{lemma}
	\label{lem:RSW2}
	There exist constants $ k_0 $ and $ C_2 >0$, depending on $d$ and $\beta$, such that the following holds.
	For any integer $ k > k_0 $ and any simple polytope $ P $ with $ n $ facets, there exists a subset $ I \subset [ n ] $ with $|I| \leq k $ such that for any $h\ge0$
	\[
	\distH ( P , P_I ) 
	< C_2 c_\Phi \Phi_\beta((\origin,h),P)^{\frac{1}{\alpha}} k^{ - \frac 2{d - 1} } ,
	\]
	where $c_\Phi$ is the constant from \eqref{cphi}.
\end{lemma}
\begin{proof}
	We set 
	$k_0$ such that $\left\lfloor k_0/d \right\rfloor \ge N_1 $,
	where $ N_1 $ is from Lemma \ref{lem:RSWddependent}.
	We apply Lemma \ref{lem:RSWddependent} to $P$ and $n'=\lfloor k/d \rfloor$.
	We obtain a polytope $Q\supset P$ with $\lfloor k/d \rfloor$ facets and 
	\[
	\distH(P,Q)
	< C_1 R(P) \left\lfloor \frac kd \right\rfloor ^{- \frac 2{d-1}}   
	< C_2 c_\Phi \Phi_\beta((\origin,h),P)^{\frac{1}{\alpha}} k ^{- \frac 2{d-1}} , 
	\]
	where the second inequality is a direct consequence of the definition \eqref{cphi} of $c_\Phi$.
	
	By means of shifting and rotating the facets of $Q$ slightly, we can assume that each of the facets of $Q$ meets exactly one vertex of $P$ in its interior.
	Let $I$ be the set of indices of facets of $P$ with one vertex in a facet of $Q$.
	Since $P$ is simple, we have 
	$$|I| \leq d\, f_{d-1}(Q) = d \left\lfloor \frac kd \right\rfloor \leq k,$$
	where we recall that $f_{d-1}(Q)$ is the number of facets of $Q$.
	Finally, we observe that
	$ P \subset P_I \subset Q $, which implies $ \distH ( P , P_I ) \leq \distH ( P , Q )$.
\end{proof}

It turns out that we can actually delete a facet of a polytope while controlling its $\Phi$-content.
This is stated formally in the following lemma, which is analogue to \cite[Lemma~3.3]{bonnet_cells_2018}, with $\Phi$ replaced by $\Phi_\beta$.
\begin{lemma}
	\label{lem:RSW3}
	There exist constants $N_2$ and $C_3$, depending on $d$ and $\beta$ such that the following holds.
	For any  $ n > N_2 $ and any simple polytope $ P = \cap_{i=1}^n \halfspace_i \in \cP_n$, there exists a subset $ J \subset [ n ] $ of cardinality at least $ n / 4 $ such that, for any $ j \in J $ and $h\ge 0$,
	\begin{equation}
		\label{eq:dHdiff1}
		\distH \left( P , P_{[n]\setminus\{j\}} \right) 
		< C_3 c_\Phi \Phi_\beta((\origin,h),P)^{\frac{1}{\alpha}} n^{-\frac 2{d-1}} ,
	\end{equation}
	where $c_\Phi$ is the constant from \eqref{cphi}, and 
	\begin{equation}
		\label{eq:Phidiff1}
		\Phi_\beta \left( (\origin,h),P_{[n]\setminus\{j\}} \right) 
		< \exp \left( C_3^\alpha c_\Phi^\alpha n^{-\frac{d+1}{d-1}} \right) \Phi_\beta((\origin,h),P).
	\end{equation}
\end{lemma}

Note that, both \eqref{eq:dHdiff1} and \eqref{eq:Phidiff1} imply that $P_{[n]\setminus\{j\}}$ is bounded and thus is a polytope.

\begin{proof}
	We will start by showing that \cref{lem:RSW2} implies \eqref{eq:dHdiff1} for $j\in J'$ for some $J'\subset [n]$ of cardinality at least $n/2$, and then we will prove that \eqref{eq:Phidiff1} holds for at least half of the $j\in J'$.
	
	Put $ k = n - 2 \lceil n/4 \rceil $ which implies $k \geq k_0$ when $n > 2k_0 + 4$, where $k_0$ is taken from \cref{lem:RSW2}.
	Note that we will set the value of $N_2 \geq 2k_0 + 4$ later in the proof.
	By \cref{lem:RSW2} there is a set $ I \subset [n] $ of cardinality $ k $ such that 
	\begin{equation}
		\label{e:bound-dH-P-PI}
		\distH ( P , P_I )
		< 
		C_2 c_\Phi  \Phi_\beta((\origin,h),P)^{\frac{1}{\alpha}} k^{ - \frac 2{d - 1} }     
		\leq
		C_3 c_\Phi  \Phi_\beta((\origin,h),P)^{\frac{1}{\alpha}} n^{-\frac 2{d-1}}.
	\end{equation}
	Hence \eqref{eq:dHdiff1} holds for any $j \in J' = [n]\setminus I$, because $P \subset P_{[n]\setminus\{j\}} \subset P_I$.
	Note that $J'$ has cardinality at least $n/2$.
	Therefore, it only remains to show that, for at least half of the $j$ not in $I$, equation $(\ref{eq:Phidiff1})$ holds as well.
	Set 
	\begin{align}
		\label{e:delta-prime}
		\delta'
		\coloneqq C_3 n^{-\frac2{d-1}} 
		\qquad \text{ and } \qquad
		U_j = {\rm cl} \left\{  u  \in \Sph^{d-1} :
		h \left( P_{ [n]\setminus\{j\} } ,  u \right) \neq h ( P ,  u ) \right\} .
	\end{align}
	Also, set
	$N_2$
	such that we have 
	$ N_2 \ge 2k_0 + 4 $
	and
	$ \delta' < 1 $
	for 
	$ n > N_2 $.
	Applying \cref{lem:contPhi} with 
	$ K = P $,
	$ L = P_{[n]\setminus\{j\}} $,
	$ \delta 
	= \delta' c_\Phi \Phi_\beta((\origin,h),P)^{\frac{1}{\alpha}} $
	and 
	$ U = U_j $,
	gives
	\begin{equation}
		\label{eq:PhidiffS}
		\Phi_\beta \left((\origin,h), P_{ [n]\setminus\{j\} } \right) - \Phi_\beta((\origin,h),P)
		< \delta' c_m c_\Phi^\alpha 2^{\alpha - 1} c_\Phi^\alpha \Phi_\beta((\origin,h),P) \sigma_{d-1} \left( U_j \right)
	\end{equation}	
	because the condition \eqref{eq:localvariation} is satisfied thanks to the definition \eqref{e:delta-prime} of $U_j$ and the fact that $\support(\biggerconvexbody,u) - \support(\convexbody,u) \leq \distH(P,P_I)$ which is bounded by \eqref{e:bound-dH-P-PI}.
	We need to estimate the measure of the set $ U_j $.
	Denote by ${v}_1, \ldots, {v}_m$ the vertices of the polytope $P$.
	Since the polytope is simple, each vertex is the intersection of precisely $d$ hyperplanes. 
	Denote by $N ( {v}_l )$ the unit vectors in the normal cone of $P$ at $ {v}_l $, i.e.
	\[
	N({v}_l )
	= \left\{  {u}  \in \Sph^{d-1} \mid
	\support(P,  {u} ) = \langle {v}_l , {u} \rangle  \right\} .
	\]
	The essential observation is that
	\[
	U_j = \bigcup_{ {v}_l \in \hyperplane_j} N( {v}_l ) .
	\]
	Observe that the sets $N( {v}_l )$ have pairwise disjoint interiors and cover $\Sph^{d-1}$. Thus for almost all $ {u} \in \Sph^{d-1} $ we have
	\begin{equation*}
		\sum_{j=1}^n \1{ u \in U_j }
		=
		\sum_{j=1}^n \sum_{l=1}^m \1{ {v}_l \in \hyperplane_j} \1{ u \in N({v}_l) } 
		=
		\underbrace{ \sum_{l=1}^m  \1{ u  \in N({v}_l) } }_{=1}
		\underbrace{ \sum_{j=1}^n  \1{ {v}_l \in \hyperplane_j } }_{=d}
		= d 
	\end{equation*}
	This yields
	$ \sum_{j=1}^n \sigma_{d-1} (U_j)  = d \omega_d$
	and in particular
	\[
	\sum_{j\notin I} \sigma_{d-1} ( U_j)  \leq d \omega_d  .
	\]	
	This implies that, for at least half of the $j\notin I$, we have
	\[
	\sigma_{d-1}(U_j)
	\leq \omega_d d \left(\frac{n-k}{2}\right)^{-1} 
	= \omega_d d \left\lceil \frac n4 \right\rceil ^{-1}
	\leq 4 \omega_d d n^{-1} .
	\]
	Otherwise we would have at least half of the $j\notin I$ with the reverse inequality and, because $|I|=k = n - 2 \lceil n/4 \rceil$, that would imply 
	\[
	d \omega_d
	\geq \sum_{j\notin I}^n \sigma_{d-1} (U_j)
	> \frac12 (n-k) \frac{2d \omega_d}{n-k}
	= d \omega_d.
	\]	
	Combined with \eqref{eq:PhidiffS}, it shows that there exists a set $J\subset [n]\setminus I$ of cardinality $(n-k)/2 = \lceil n/4 \rceil$ such that, for any $j\in J$, we have
	\begin{align*}
		\Phi_\beta\left((\origin,h), P_{[n]\setminus\{j\}} \right)-\Phi_\beta((\origin,h),P) 
		&<
		\delta' c_m c_\Phi^\alpha 2^{\alpha - 1} c_\Phi^\alpha \Phi_\beta((\origin,h),P)
		4 \omega_d d n^{-1}
		\\&=
		4 \omega_d d C_3 2^{\alpha - 1} c_\Phi^\alpha \Phi_\beta((\origin,h),P) n^{-\frac{d+1}{d-1}}
		\\&\leq
		C_3^\alpha c_\Phi^\alpha \Phi_\beta((\origin,h),P) n^{-\frac{d+1}{d-1}} ,
	\end{align*}
	where the equality is a direct consequence of the definition \eqref{e:delta-prime} of $\delta'$ and the last inequality can be achieved by choosing $C_3$ big enough.
	This, by the fact that $1+x\le e^x$, implies \eqref{eq:Phidiff1}.
\end{proof}

We need to state the following elementary but useful lemma.
Let $ \mathfrak{S}_n $ denote the set of permutations of $[n]$.
For $ x = ( x_1 , \ldots , x_n ) $ and $ \sigma \in  \mathfrak{S}_n $, we write
$  x_\sigma \coloneqq ( x_{\sigma(1)} , \ldots , x_{\sigma(n)} ) $.
It is clear that the following holds, see \cite[Lemma~4.4]{bonnet_cells_2018},
which allows us to exploit the existence of many “good” facets.
\begin{lemma}
	\label{lem:TechnicalPermutation}
	Let $ ( X , \Sigma , \psi ) $ be a measured space,
	$m$, $n > 0 $ be integers,
	$ f \colon X^n \to [0,\infty) $ be a measurable function
	and $S$, $T \subset X^n $ measurable sets.
	Assume that
	\begin{itemize}
		\item $ f $ is symmetric:
		for any $ \sigma \in \mathfrak{S}_n $
		and any $  {x} \in X^n $,
		we have $ f (  {x}_\sigma ) = f (  {x} ) $;
		\item $ S $ is symmetric:
		for any $ \sigma \in \mathfrak{S}_n $,
		and any $  {x} \in X^n $ we have
		$ \1 {x_\sigma \in S } = \1 {x \in S } $;
		\item for any $  {x} \in S $,
		there exist at least $ p $ permutations $ \sigma \in \mathfrak{S}_n $
		such that $ {x}_\sigma \in T $.
	\end{itemize}
	Then
	\[
	\frac{p}{n!}
	\int_{ X^n }
	\1{x \in S }
	f (  {x} ) 
	{\psi}^{n}(\dint x)
	\leq 
	\int_{ X^n } 
	\1{x \in T}
	f ( {x} )
	{\psi}^{n}(\dint x) .
	\]
\end{lemma}

In the next lemma we investigate the measure of those polytopes $P_{[n]}$ which are close to $P_{[n-1]}$ in the Hausdorff distance.
For this, we estimate the contribution of the last half-space.
This corresponds to \cite[Lemma~4.5]{bonnet_cells_2018}, but the dependence of $\Phi_\beta$ on $(\origin,h)$
leads to a different scaling in the bound.
\begin{lemma}
	\label{lem:0IntTakeOutFacets}
	Let $h\geq 0$.
	For any $a>0$ and any measurable function $f \colon (\halfspaces)^{n-1}\to[0,\infty)$, it holds that
	\begin{align*}
		& \int_{(\halfspaces)^n}
		\1{ P_{[n]} \in \cP_n }
		\1 { \distH \left( P_{[n]} , P_{[n-1]} \right) < a }
		f ( \halfspace_1, \ldots , \halfspace_{n-1} )    
		\widetilde{\mu}_{\beta;(\origin,h)}^{\otimes n}(\dint \halfspace_{1:n})
		\\
		&<
		a c_m \omega_d c_\Phi^{\alpha-1}
		\int_{(\halfspaces)^{n-1}}
		\Phi_\beta\left((\origin,h),  P_{[n-1]} \right)^{\frac{\alpha-1}{\alpha}}
		\1 { P_{[n-1]} \in \cP_{n-1} }
		f ( \halfspace_{1:n-1} )    
		\widetilde{\mu}_{\beta;(\origin,h)}^{\otimes n-1}(\dint \halfspace_{1:n-1}) ,
	\end{align*}
	where $c_\Phi$ and $c_m$ are the constants from \eqref{cphi} and \eqref{cm}, respectively.
\end{lemma}
\begin{proof}
	The essential part of the proof is to bound
	\[
	I_Q =
	\int_{\halfspaces}
	\1 {Q \cap \hyperplane_n \ne \emptyset}
	\1 {\distH (Q , Q\cap \halfspace_n ) < a}
	\widetilde{\mu}_{\beta;(\origin,h)}(\dint \halfspace_n),
	\]
	for any $ Q \in \mathbb{K}  $, where $\hyperplane_n$ is the hyperplane bounding the half-space $\halfspace_n$.
	The lemma will follow by applying that bound to
	$ Q 
	= P_{[n-1]} 
	= \cap_{i=1}^{n-1} \halfspace_i
	\in \cP_{n-1} $, then multiplying by $f$ and integrating over $\widetilde{\mu}^{\otimes(n-1)}_{\beta,(\origin,h)}$.
	
	By \cref{density-on-half-spaces},
	\begin{align*}
		I_Q
		&=
		\int_{\Sph^{d-1}} \int_{-\infty}^\infty 
		\1 {Q \cap \hyperplane(u,t) \ne \emptyset}
		\1 {\distH (Q , Q\cap \halfspace(u,t) ) < a}
		m_{\beta;(\origin,h)}(t) \dint t \, \sigma_{d-1}( \dint u ) .
	\end{align*}
	Note that $Q \cap \hyperplane(u,t) \ne \emptyset$ if and only if $t\in [\support(Q,-u),\support(Q,u)]$, and that, when this condition is fulfilled, $\distH (Q , Q\cap \halfspace(u,t) ) < a$ if and only if $t > \support(Q,u)-a$.
	Hence, we can rewrite the above integral as
	\begin{align*}
		I_Q
		&=
		\int_{\Sph^{d-1}} \int_{\max \left( \support(Q,-u) , \support(Q,u)-a \right) }^{\support(Q,u)}
		m_{\beta;(\origin,h)}(t) \dint t \, \sigma_{d-1}( \dint u )
		\\&\le
		a c_m \int_{\Sph^{d-1}} (h + \support(Q,u)^2)^{\frac{\alpha-1}{2}} \sigma_{d-1}( \dint u )
		\le a c_m \omega_d c_\Phi^{\alpha-1} \Phi_\beta((\origin,h),Q)^{\frac{\alpha-1}{\alpha}} ,
	\end{align*}	
	where the first inequality is a direct consequence of \Cref{i:density-1} in \cref{density-on-half-spaces}, and the second inequality follows from \eqref{e:R(K)} and the definition \eqref{cphi} of $c_\Phi$.
	The rest follows.
\end{proof}

Combining the previous ingredients yields the desired recursive bound,
following the same scheme as in \cite[Theorem~4.1]{bonnet_cells_2018}.

\begin{proof}[Proof of \cref{upper-bound-beta}]
	Set
	$ a 
	= C_3 c_\Phi n^{-2/(d-1)} $
	and
	$ b 
	= C_3^\alpha c_\Phi^\alpha n^{-(d+1)/(d-1)} $,
	where $C_3$ is the constant of \cref{lem:RSW3}.  
	By \cref{complementary}
	$$
	\P{\smash{\typcell}\in\cP_n}
	=
	\Gamma\left(n + \frac{2 \beta+2}{\alpha}+1\right)
	\int_{\widehat{\cP}_{\origin,n}}
	\1{\Phi_\beta(K) < 1}
	\nu_{\beta;\origin,n} (\dint K).
	$$
    This can also be written 
	$$
	\P{\smash{\typcell}\in\cP_n}
	=
	\frac{\gamma c_{d+1,\beta}}{\Lambda_\beta}
	\frac{\Gamma\left(n + \frac{2 \beta+2}{\alpha}+1\right)}{n!}
	\int_{\R_+}
	\int_{(\halfspaces)^n}
	\1 { P_{[n]} \in \cP_n}
	\1{\Phi_\beta((\origin,h), P_{[n]}) < 1}
	\widetilde{\mu}_{\beta;(\origin,h)}^{\otimes n}(\dint \halfspace_{1:n})
	h^\beta \dint h,
	$$
	where $P_{[n]} = \cap_{i=1}^n \halfspace_i$.
	Now, we want to use Lemma $\ref{lem:RSW3}$ which, roughly speaking, tells us that the variable $\halfspace_n$ has a small influence.
	Set 
	\[
	S = \left\{ \halfspace_{1:n} \in (\halfspaces)^n \mid \cap_{i=1}^n \halfspace_i \in \cP_n  \text{ and $\cap_{i=1}^n \halfspace_i$ is a simple polytope}   \right\} ,
	\]
	and
	\[ 
	T = \left\{ \halfspace_{1:n} \in S \mid 
	\distH \left( P_{[n]} , P_{[n-1]} \right) < a \Phi_\beta \left((\origin,h),P_{[n]}\right)^\frac{1}{\alpha} \text{, } 
	\Phi_\beta\left((\origin,h),P_{[n-1]}\right) < e^b \Phi_\beta\left((\origin,h),P_{[n]}\right) \right\} .
	\]
	Lemma $\ref{lem:RSW3}$ tells us that, 
	for $n$ large enough (specifically, when $n>N_2$ with $N_2$ from \cref{lem:RSW3}), 
	for any $\halfspace_{1:n} \in S $, there exists at least $n!/4$ permutations $\sigma\in\mathfrak{S}_n$ such that $(\halfspace_{1:n})_\sigma \in T$.
	Hence, Lemma $\ref{lem:TechnicalPermutation}$ implies
	\begin{align*}
		\frac14 \P{\smash{\typcell}\in\cP_n}
		&\leq
		\frac{\gamma c_{d+1,\beta}}{\Lambda_\beta}
		\frac{\Gamma\left(n + \frac{2 \beta+2}{\alpha}+1\right)}{n!}
		\int_{\R_+}
		\int_{(\halfspaces)^n}
		\1 { P_{[n]} \in \cP_n}
		\1 {\distH \left( P_{[n]} , P_{[n-1]} \right) < a}
		\\&\quad \times
		\1{\Phi_\beta\left((\origin,h),P_{[n-1]}\right) < e^b }
		\widetilde{\mu}_{\beta;(\origin,h)}^{\otimes n}(\dint \halfspace_{1:n})
		h^\beta \dint h.
	\end{align*}
	Using Lemma \ref{lem:0IntTakeOutFacets} with 
	$ f(\halfspace_1,\ldots,\halfspace_{n-1})
	= \1{\Phi_\beta\left((\origin,h),P_{[n-1]}\right) < e^b }$,
	we have 
	\begin{align*}
		&\int_{(\halfspaces)^n}
		\1 { P_{[n]} \in \cP_n}
		\1 {\distH \left( P_{[n]} , P_{[n-1]} \right) < a}
		\1{\Phi_\beta\left((\origin,h),P_{[n-1]}\right) < e^b }
		\widetilde{\mu}_{\beta;(\origin,h)}^{\otimes n}(\dint \halfspace_{1:n})
		\\& \le
		a c_m \omega_d c_\Phi^{\alpha-1}
		\int_{(\halfspaces)^{n-1}}
		\Phi_\beta\left((\origin,h),  P_{[n-1]} \right)^{\frac{\alpha-1}{\alpha}}
		\1 { P_{[n-1]} \in \cP_{n-1} }
		\1{\Phi_\beta\left((\origin,h),P_{[n-1]}\right) < e^b }
		\widetilde{\mu}_{\beta;(\origin,h)}^{\otimes n-1}(\dint \halfspace_{1:n-1})
		\\& \le
		a c_m \omega_d c_\Phi^{\alpha-1}
		\exp \left(\frac{\alpha-1}{\alpha} b\right)
		\int_{(\halfspaces)^{n-1}}
		\1 { P_{[n-1]} \in \cP_{n-1} }
		\1{\Phi_\beta\left((\origin,h),P_{[n-1]}\right) < e^b }
		\widetilde{\mu}_{\beta;(\origin,h)}^{\otimes n-1}(\dint \halfspace_{1:n-1}).
	\end{align*}
	
	By recalling the definition \eqref{e:def-nu-beta-o-n} of $\nu_{\beta;\origin,n}$, this brings us to
	\begin{equation*}
		\frac14 \P{\smash{\typcell}\in\cP_n}
		\le
		a c_m \omega_d c_\Phi^{\alpha-1}
		\exp \left(\frac{\alpha-1}{\alpha} b\right)
		\frac{\Gamma\left(n + \frac{2 \beta+2}{\alpha}+1\right)}{n}
		\int_{\widehat{\cP}_{\origin,n-1}}
		\1{\Phi_\beta(\widehat{P}) < e^b }
		\nu_{\beta;\origin,n-1} (\dint \widehat{P}).
	\end{equation*}
	
	Applying now \cref{complementary-variable-change} with $n'=n-1$ and $b' = e^b $, we get
	\begin{align*}
		&\frac14 \P{\smash{\typcell}\in\cP_n} 
		\\&\leq
		a c_m \omega_d c_\Phi^{\alpha-1}
		\exp \left(\left(\frac{\alpha-1}{\alpha} + n - 1 + \frac{2\beta}{\alpha}\right) b\right)
		\frac{n + \frac{2 \beta+2}{\alpha}}{n}
		\Gamma\left(n + \frac{2 \beta+2}{\alpha}\right)
		\int_{\widehat{\cP}_{\origin,n-1}}
		\1{\Phi_\beta(\widehat{P}) < 1}
		\nu_{\beta;\origin,n-1} (\dint \widehat{P})
		\\&=
		a c_m \omega_d c_\Phi^{\alpha-1}
		\frac{n + \frac{2 \beta+2}{\alpha}}{n}
		\exp \left(\left(\frac{\alpha-d-1}{\alpha} + n\right)b\right)
		\P{\smash{\typcell}\in\cP_{n-1}},
	\end{align*}
	where the equality is due to \cref{complementary}.
	Recall that $a = C_3 c_\Phi n^{-2/(d-1)}$.
	Since 
	$
	\left(\frac{\alpha-d-1}{\alpha} + n \right)b
	< 2 n b
	= 2 C_3^\alpha c_\Phi^\alpha n^{- \frac2{d-1}} 
	< 1
	$
	for 
	$ n > (2 C_3^\alpha c_\Phi^\alpha )^{\frac{d-1}{2}}$,
	and $\frac{n + \frac{2 \beta+2}{\alpha}}{n}<2$,
	we thus have
	\[ 
	\P{\smash{\typcell}\in\cP_n} 
	\leq C n^{-\frac{2}{d-1}} \P{\smash{\typcell}\in\cP_{n-1}},
	\]
	for any $C\geq 2 e \omega_d C_3 c_m c_\Phi^\alpha $ and any
	$ n > \max \left( N_2 , \left( 2 C_3^\alpha c_\Phi^\alpha \right)^{(d-1)/2} \right) $.
	Now, similar to previous bounds, by taking $C$ sufficiently large, the inequality holds also for any $n\geq d+1$. Thus the theorem holds.
\end{proof}

%% file: inputs/concentration.tex
\section[Concentration for β-Delaunay]{Concentration for $\beta$-Delaunay}
\label{sec:concentration}

In this section we establish the concentration of the distribution of the maximal degree in a growing window for the $\beta$-Delaunay model for some fixed $\beta$.
We consider the maximal degree of $\cD_\beta$ observed in a window $B \subset \R^d$, defined by
$$M^B \coloneqq \max_{v\in \cF_0(\cD_\beta) \cap B} \deg_{G(\cD_\beta)} v ,$$
and establish the following concentration phenomenon when $B$ is a large cubical window of the form 
$$ W_\rho \coloneqq [0,\rho^{\frac{1}{d}}]^d , \quad  \rho>0.$$
\begin{theorem}[Concentration for $\beta$-Delaunay]
	\label{concentration}
	Let $\beta>-1$, $d\ge 2$.	
	\begin{enumerate}
		\item There exists a sequence $(\rho(k))_{k\in\N}$ such that $\P{M^{W_{\rho(k)}} \in \{k, \ldots, k + \ceil{\frac{d}{2}} - 1\}} \rightarrow 1$, as $k\to\infty$.
		\item There exists a function $\rho \mapsto k(\rho)$ such that $\P{M^{W_\rho} \in \{k(\rho), \ldots, k(\rho) + \ceil{\frac{d}{2}}\}} \rightarrow 1$, as $\rho\to\infty$.
	\end{enumerate}
\end{theorem}

\begin{rem}[Comparison to the results  and proof techniques of \cite{bonnet_maximal_2020} for the classical Delaunay model]
	\label{relation-poisson}
	All the arguments in this section are applicable to the classical Poisson-Delaunay model (up to a few straightforward adaptations) and thus \cref{concentration} applies to that setting as well.
	We do not write the details of the proof in the classical setting to avoid excessive repetition.

	In particular we recover \cite[Theorem 1]{bonnet_maximal_2020} about the case $d=2$, and  improve \cite[Theorem 3]{bonnet_maximal_2020} covering the case $d\geq 3$.
	For the latter, the result therein established a concentration on $1+\floor{\frac{d+3}{2}}$ values, which is always strictly worse than our concentration on $1+\ceil{\frac{d}{2}}$ values.

	Our proof is essentially an adaption of the proof of \cite[Theorems 3]{bonnet_maximal_2020}, with a key difference of a more apt choice of $k(\rho)$ resulting in a stronger result. In particular this strengthened proof shows directly concentration on two values for $d=2$, in contrast with \cite{bonnet_maximal_2020}.
	This means that our proof does not rely on graph planarity and other lengthy technical geometric arguments that the proof of \cite[Theorems 1]{bonnet_maximal_2020} relies on.

	In summary the current proof is simpler and yields a stronger result.
\end{rem}

Let us prepare some tools for the proof.

Denote the amount of vertices of degree at least $k$ in a window $B$ by 
$$N^B[k] \coloneqq \left|\left\{ v \in \cF_0(\cD_\beta) \cap B  \mid \deg_{G(\cD_\beta)} v \ge k  \right\}\right|.$$
Due to \cref{duality},
$$G(k) \coloneqq \E{N^{W_1}[k]} = \Lambda_\beta \P{\typcell\in \bigcup_{n=k}^\infty \cP_n}.$$
From \cref{lower-bound-beta} and \cref{upper-bound-beta} we can deduce that for some constants $c,C>0$ and for sufficiently large $k$
\begin{align}
	\label{bounds-on-G}
	G(k) \ge (ck)^{-\frac{2}{d-1}k}
	\quad \text{ and } \quad
	G(k+1)\le C k^{-\frac{2}{d-1}} G(k).
\end{align}
Let us define a continuous extension of $G$ to $(0,\infty)$:
$$G_c(x) \coloneqq \exp(\log G(\floor{x}) + (x-\floor{x}) (\log G(\floor{x} + 1) - \log G(\floor{x}))).$$
Note that $G$ and $G_c$ are decreasing functions.
The second inequality of \eqref{bounds-on-G} carries over in the following way.
\begin{corollary}
	\label{continuous-upper-bound}
	Let $0\ll x<y$. Then $G_c(y) \le C \floor{x}^{-\frac{2(y-x)}{d-1}} G_c(x)$. 
\end{corollary}

\begin{proof}
	Since $\log G_c(\cdot)$ is a piecewise linear function
	$$
	\log G_c(y)
	\le \log G_c(x) + (y-x) \max_{\floor{x} \le k \le \floor{y}}(\log G(k + 1) - \log G(k)).
	$$	
	But for any $k\ge \floor{x}$
	$$
	\log G(k + 1) - \log G(k) \le \log (C k^{-\frac{2}{d-1}}) \le \log (C \floor{x}^{-\frac{2}{d-1}} ).
	$$	
	After taking an exponent again, we are done.
\end{proof}

By the first moment method,
\begin{align}
	\label{e:first_moment_method}
	\P{M^B \ge k} = \P{N^B[k] \ge 1} \le V_d(B) G(k).
\end{align}
The following key lemma provides an inequality in the other direction.

\begin{lemma}
	\label{maximum-lower-bound}
	There exist constants $c>0$ and $v_0>0$ such that
	$$
	\P{M^B \ge k} \ge V_d(B)^{-\frac{2\beta+2}{d}} \left(G(k) - e^{-c V_d(B)^{\frac{2\beta+d+2}{d}}} \right),
	$$
	for any Borel set $B$ with $V_d(B) > v_0$
\end{lemma}

\begin{proof}
	We start by setting some positive real valued parameters whose relevance will become clear along the proof.
	Set
	\begin{align}
		\label{e:param-pf-lem-7-4}
		t &= V_d(B)^{\frac{2\beta+2}{d}}, 
		&n &= t V_d(B) = V_d(B)^{\frac{2\beta+d+2}{d}} ,
		&H &= V_d(B)^{\frac{2}{d}}.
	\end{align}
	We can assume without loss of generality that $t>t_0$ and $H>H_0$, where $t_0$ and $H_0$ are the constants from \cref{number-of-vertices} and \cref{typical-height}, respectively.
	Notice that
	$$
	V_d(B) G(k)
	= \E{N^B[k]}
	\le n \P{1\le N^B[k] \le n} +  \E{N^B[k] \1{N^B[k] > n}}.
	$$	
	Consequently,
	\begin{align}
		\label{e:lem7-4-pf-ub1}
		\P{M^B \ge k} 
		&\ge \P{1\le N^B[k] \le n} 
		\ge V_d(B)^{-\frac{2\beta+2}{d}} G(k) - \frac{1}{n}\E{N^B[k] \1{N^B[k] > n}}.	
	\end{align}
	Let us estimate the last expectation. Since $N^B[k] \le |\cF_0(\cD_\beta) \cap B|$,	
	\begin{align*}
			\E{N^B[k] \1{N^B[k] > n}}
			&\le \E{ \, |\cF_0(\cD_\beta) \cap B| \, \1{|\cF_0(\cD_\beta) \cap B| > n}}
			\\
			&= \E{\sum_{(v,h)\in \eta_\beta} \1{v\in B} \1{v\in\cF_0(\cD_\beta)} \1{|\cF_0(\cD_\beta) \cap B| > n}}.
	\end{align*}
	Using the Mecke equation, we can write this expectation as an integral. Recall that $\eta_{\beta;(v,h)} = \eta_\beta \cup \{(v,h)\}$. 	
	\begin{align*}
		\E{N^B[k] \1{N^B[k] > n}}
		&\leq
		\int_{B\times \R_+} \P{v\in \cF_0(\cD(\eta_{\beta;(v,h)})),\, |\cF_0(\cD(\eta_{\beta;(v,h)})) \cap B| > n } \mu_\beta(\dint v,\dint h) .
	\end{align*}
	Since $\cF_0(\cD(\eta_{\beta;(v,h)}))
	\subset \cF_0(\cD(\eta_\beta )) \cup \{v\}$,
	the integrand is bounded by $\P{ |\cF_0(\cD_\beta) \cap B| \ge n}$.
	It is also trivially bounded by $\P{v\in \cF_0(\cD(\eta_{\beta;(v,h)}))}$.
	Using these two bounds strategically, we split the last integral as follows.
	\begin{align*}
		&\E{N^B[k] \1{N^B[k] > n}}
		\\&\le
		\int_{B\times [0,\height ]} \P{ |\cF_0(\cD_\beta) \cap B| \ge n} \mu_\beta(\dint v,\dint h)
		+ \int_{B\times [\height ,+\infty)} \P{v\in \cF_0(\cD(\eta_{\beta;(v,h)}))}  \mu_\beta(\dint v,\dint h)
		\\
		&=
		\P{ |\cF_0( \cD_\beta) \cap B| \ge n}  \frac{\gamma c_{d+1,\beta}}{\beta + 1} V_d(B) \height^{\beta+1}
		+ V_d(B) \Lambda_\beta \P{\smash{h_\typcell} \ge \height} ,
	\end{align*}
	where the last equality holds due to the definition \eqref{def:mu_beta} of $\mu_\beta$ and the definition of the typical cell $(h_\typcell,\typcell)$, see \cref{s:typ-cell}.
	We bound the last two probabilities by applying \cref{number-of-vertices} (for the first) and \cref{typical-height} (for the second).
	This gives
	\begin{align*}
		\frac{1}{n} \E{N^B[k] \1{N^B[k] > n}}
		&\leq \frac{V_d(B)}{n}\left[\left(e^{-tV_d(B)} + C e^{-ct^{\frac{\beta + d/2 + 1}{\beta + 1}}}\right) \frac{\gamma c_{d+1,\beta}}{\beta + 1} \height^{\beta+1}
		+ \Lambda_\beta e^{- c' \height^{\beta + \frac{d}{2} + 1}} \right] ,
	\end{align*}
	where $c, c', C \in (0,\infty)$ are constants independent of $B$ and $k$. 
	Recalling the definition of the parameters in \eqref{e:param-pf-lem-7-4}, we have that the three exponential terms in the last display are bounded by $e^{- \min(c,c',1) V_d(B)^{\frac{2\beta+d+2}{d}}}$, while the factor $H^{\beta+1}$ is only polynomial in $V_d(B)$.
	Therefore, since $V_d(B)$ is assumed to be large enough, we have
	\begin{align*}
		\frac{1}{n} \E{N^B[k] \1{N^B[k] > n}}
		&\leq \frac{V_d(B)}{n} e^{- c'' V_d(B)^{\frac{2\beta+d+2}{d}}} ,
	\end{align*}
	where $c'' = \min(c,c',1)/2$.
	Since $\frac{V_d(B)}{n} = V_d(B)^{- \frac{2\beta + 2}{d}}$, combining the last bound with \eqref{e:lem7-4-pf-ub1} leads to the desired bound.	
\end{proof}

\begin{proof}[Proof of \cref{concentration}]
	
We will start with the first statement.
Let $\rho = \rho(k) \coloneqq G_c(k + \ceil{\frac{d}{2}} - \frac{1}{4})^{-1}$.
Then, by \eqref{e:first_moment_method} and \cref{continuous-upper-bound},
$$
\P{M^{W_\rho}\ge k + \ceil[\Big]{\frac{d}{2}}}
\le \rho G\left(k + \ceil[\Big]{\frac{d}{2}}\right)
= \frac{G_c(k + \ceil{\frac{d}{2}})}{G_c(k + \ceil{\frac{d}{2}} - \frac{1}{4})}
\rightarrow 0.
$$
Now subdivide $W_\rho$ into cubes of volume $V$ to be determined later. Take every $2^d$-th into a family of $N = \frac{\rho}{2^d V}$ cubes: $\{Q_i\}_{i\in[N]}$ (yellow squares on \cref{fig:grid}).

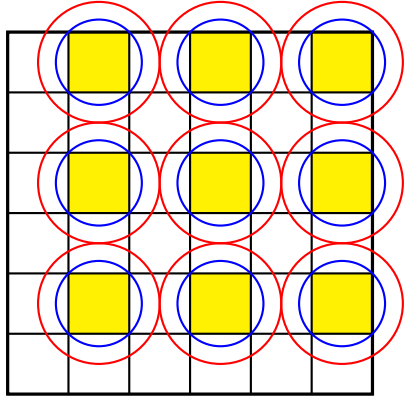
\begin{figure}[ht]
	\centering
	\begin{tikzpicture}[scale=0.8] 

		\def\gridSize{6} 
		\def\middleRadius{0.707} 
		\def\outerRadius{1}
		
		\foreach \x in {1, 3, 5, ..., \numexpr\gridSize-1\relax}
		\foreach \y in {1, 3, 5, ..., \numexpr\gridSize-1\relax} {
			\fill[yellow] (\x, \y) rectangle (\x+1, \y+1);
		}
		
		\draw[thick] (0,0) grid (\gridSize,\gridSize);
		
		\draw[very thick] (0,0) rectangle (\gridSize,\gridSize);
		
		\foreach \x in {1, 3, 5, ..., \numexpr\gridSize-1\relax}
		\foreach \y in {1, 3, 5, ..., \numexpr\gridSize-1\relax} {
			\draw[blue, thick] (\x+0.5, \y+0.5) circle (\middleRadius);
			
			\draw[red, thick] (\x+0.5, \y+0.5) circle (\outerRadius);
		}
		
		\end{tikzpicture}
		
	\caption{Subcubes and dependencies}
	\label{fig:grid}
\end{figure}

We consider the event $E_i$ that the skeleton of $\cD_\beta$ in the blue ball corresponding to $Q_i$ is determined by the points above the red ball.
It was shown in \cite[Item 1 of Lemma 4]{gusakova_beta-delaunay_2023}\footnote{In \cite{gusakova_beta-delaunay_2023} the model is defined in $\R^{d'-1}\times \R$ whereas in the current paper the model is defined in $\R^d \times \R$. Thus one needs to replace $d'$ by $d+1$ in the bound of \cite[Lemma 4]{gusakova_beta-delaunay_2023} in order to derive \eqref{dependence}.} that there exist constants $c,C>0$ such that for any $i\in[N]$,
\begin{equation}
	\label{dependence}
	\P{E_i^c} < c V^{\frac{1}{d}} e^{-C V^{\frac{2\beta + d + 2}{d}}}.
\end{equation}

Clearly, conditional on $\bigcap_{i=1}^N E_i$, the happenings of $\cD_\beta$ inside cubes $Q_i$ are independent.
\begin{align}
	\notag
	\P{N^{W_\rho}[k] = 0}
	&\le \P{\bigcap\{N^{Q_i}[k] = 0\}}
	\\
	\notag
	&\le \P{\bigcap\{N^{Q_i}[k] = 0\} \mid \bigcap E_i} + N \P{E_i^c}
	\\
	\notag
	&= \prod \P{N^{Q_i}[k] = 0 \mid \bigcap E_i} + N \P{E_i^c}
	\\
	\notag
	&\le \left(\frac{\P{N^{Q_1}[k] = 0}}{\P{\bigcap E_i}}\right)^N + N \P{E_i^c}
	\\
	\label{eq:180326a}
	&\le \left(1 - \P{N^{Q_1}[k] \ge 1}\right)^N
		\left(1-N \P{E_i^c}\right)^{-N}
		+ N \P{E_i^c}.
\end{align}
Recall from \eqref{bounds-on-G} that $|\log G(k)| > (\frac{2}{d-1} - \varepsilon) k\log k$. 
Now, we set $V$ to be of the form $V = c (k\log k)^{\frac{d}{2\beta+d+2}}$ with $c$ sufficiently large so that the bound of \cref{maximum-lower-bound} applied with $B = Q_1$ (and thus $V_d(B) = V$) gives
\begin{align}
	\label{e:LB-P-N-geq-1}
	\P{N^{Q_1}[k] \ge 1}
	=\P{M^B \ge k}
	&\ge \frac{1}{2} V^{-\frac{2\beta+2}{d}} G(k) ,
\end{align}
and also sufficiently large so that 
$$
V \ge (3 \log \rho)^\frac{d}{2\beta + d + 2}.
$$
The last condition can be satisfied since, by definition of $\rho$ and monotonicity of $G_c$, we have $ \log \rho \leq - \log G (k+\ceil{\frac{d}{2}} ) = O(k\log k)$ with the last approximation due to the lower bound \eqref{bounds-on-G}.
From \eqref{dependence},
\begin{align}
	\label{eq:180326b}
	N \P{E_i^c} < N^2 \P{E_i^c} < \rho^2 e^{-3 \log \rho} \rightarrow 0,
\end{align}
so at the same time 
\begin{align}
	\label{eq:180326c}
	\left(1-N \P{E_i^c}\right)^{-N} \approx \exp\left(N^2 \P{E_i^c}\right) \rightarrow 1.
\end{align}

Finally, applying \eqref{e:LB-P-N-geq-1} and \cref{continuous-upper-bound}, and recalling that $N = \frac{\rho}{2^d V}$ and $\rho = G_c(k + \ceil{\frac{d}{2}} - \frac{1}{4})^{-1}$, we get
\begin{align}
	\label{eq:180326d}
	N \P{N^{Q_1}[k] \ge 1}
	\ge \frac{c\rho}{V V^{\frac{2\beta+2}{d}}}  G(k)
	= \frac{cG_c(k)}{(k\log k) G_c(k + \ceil{\frac{d}{2}} -\frac{1}{4})}
	\ge \frac{ck^{(\ceil{\frac{d}{2}} -\frac{1}{4})\frac{2}{d-1}}}{k\log k}
	\ge \frac{ck^\frac{d-0.5}{d-1}}{k\log k}
	\rightarrow \infty.
\end{align}
Therefore, plugging \eqref{eq:180326b}, \eqref{eq:180326c} and \eqref{eq:180326d} in \eqref{eq:180326a}, we get $\P{M^{W_\rho} < k} = \P{N^{W_\rho}[k] = 0} \rightarrow 0$, and that concludes the proof of the first statement.

The second statement immediately follows from the first. Indeed, let $k(\rho) \coloneqq \max \{k \mid \rho(k) \le \rho\}$. Clearly
$$\P{M^{W_\rho} < k(\rho)} \le \P{M^{W_{\rho(k(\rho))}} < k(\rho)} \rightarrow 0$$
and
$$\P{M^{W_\rho} > k(\rho) + \ceil*{\frac{d}{2}}} \le \P{M^{W_{\rho(k(\rho)+1)}} \ge k(\rho) + 1 + \ceil*{\frac{d}{2}}} \rightarrow 0,$$
which is exactly what we claim.
\end{proof}

%% file: inputs/epilogue.tex
\section{Open questions}
\label{sec:open}

We conclude with several conjectures whose resolution would complement our results.
We provide supporting heuristics and discuss the challenges involved in proving them.

\paragraph{Concentration on two values for $\beta$-Delaunay model in higher dimensions}
Concentration of the maximal degree in random graphs often occurs on just two values. We have shown that this is the case for the $\beta$-Delaunay graph in dimension two. It would be reasonable to expect this applies also in higher dimensions.
\begin{conj}
Let $d\geq 3$. There exists a function $\rho \mapsto k(\rho)$ such that, as $\rho\to\infty$,
\[ \P{M^{W_\rho} \in \{k(\rho), k(\rho) + 1\}} \rightarrow 1.\]
\end{conj}
Our concentration result (\cref{concentration}) shows concentration on $1+\ceil{\frac{d}{2}}$ values.
As mentioned in \cref{relation-poisson}, we have adapted and optimized the proof of \cite[Theorem 3]{bonnet_maximal_2020}, reaching a stronger concentration than the one therein.
It seems that the arguments in these proofs cannot be further optimized and one would need a different approach to show concentration on only two values.

We believe that the result holds.
One possible approach would be to show that the vertices of high degree do not form clusters (or form only clusters of bounded size).
With the additional mixing properties of the $\beta$-Delaunay tessellation \cite{gusakova_beta-delaunay_2023}, this would imply that the maximal degree behaves essentially as the maximum of i.i.d.\ random variables. This, in turn, combined with the super-exponential decay of \cref{upper-bound-beta}, would imply desired concentration using the approach of the classical result of Anderson \cite[Theorem~1]{anderson_extreme_1970}.

However, although clusters do not seem to appear when looking at two dimensional simulations, showing that this is the case seems to be a difficult task.
In the proof of \cite[Theorem 1]{bonnet_maximal_2020} for the two dimensional Poisson-Delaunay model, it was shown that  there are no clusters of 5 vertices of high degree \cite[Proposition 11]{bonnet_maximal_2020}. That proof relies on the planarity of the graph, and there is no clear way of adapting it to higher dimensions.

\paragraph{Upper bound for $\beta'$-Voronoi model}
A challenging problem is to establish an upper bound matching the lower bound in \cref{lower-bound-beta-prime}.
\begin{conj}
	\label{conj:upper-bound}
	There exists a constant $C=C(d,\beta)<1$ such that for all $n\geq d+1$,
	\[
	\P{\typcellprime\in\cP_n} 
	< C^n.
	\]
\end{conj}
The approach of the upper bound for the $\beta$-model does not seem to work here.
Recall that a key observation in the proof of \cref{upper-bound-beta} is that for any $((\origin,h),P)\in\widehat{\cP}_{\origin,n}$, there exist facets $\hyperplane_j$ of $P = \cap_{i\in[n]} \halfspace_i$ such that, when erased, we get a polytope $P_{[n]\setminus\{j\}} = \cap_{i\in[n]\setminus\{j\}} \halfspace_i$ which is close to $P$, both in Hausdorff distance and it terms of $\Phi$-content, see \cref{lem:RSW3}. 
Unfortunately (or interestingly!), in the $\beta'$-model the situation is more complex. 
It is even easy to find elements $((\origin,h),P) \in \widehat{\cP}_{\origin,n}^\prime$ for which erasing any facet of $P$ results in a polytope $P_{[n]\setminus\{j\}}$ which is not contained in the ball $\sqrt{-h}\B^d$, and for which therefore $((\origin,h),P_{[n]\setminus\{j\}}) \notin \widehat{\cP}_{\origin,n-1}^\prime$, for any $j\in[n]$. 
Another challenge is to track an explicit constant in all of the bounds involved in the result to make sure the resulting constant $C$ is less than one.

However, we believe that $\frac{1}{n} \log \P{\typcellprime\in\cP_n}$ converges, which is a slightly stronger than \cref{conj:upper-bound}.
We reformulate it in the following equivalent form, for presentation purposes of \cref{conj:anti-concentration}.
\begin{conj}
	\label{conj:upper-bound-2}
	There is a constant $\alpha=\alpha(d,\beta)>0$ such that
	\[
	\frac{\P{ f_{d-1}(\typcellprime)\geq n}}{\P{ f_{d-1}(\typcellprime)\geq n+1}} 
	\to e^\alpha ,
	\]
	as $n\to\infty$, where $f_{d-1}(\typcellprime)$ is the number of facets of $\typcellprime$.
\end{conj}

\paragraph{Anti-concentration results for $\beta'$-Delaunay model}
By Anderson in \cite[Theorem 2]{anderson_extreme_1970}, the maximum of i.i.d.\ random variables with exponential tails has a distribution which tends to be bounded by two Gumbel distributions shifted by $1$.
If one were to confirm \cref{conj:upper-bound-2} and to show that the vertices of large degrees are sufficiently independent so that the maximal degree behaves as the maximum of i.i.d.\ random variables, then \cite[Theorem 2]{anderson_extreme_1970} would lead to the following anti-concentration conjecture, with the constant $\alpha$ as in \Cref{conj:upper-bound-2}.
\begin{conj}
	\label{conj:anti-concentration}
	There exists a constant $\alpha>0$ and a function $\rho \mapsto k(\rho)$, such that for all $z\in\Z$, 
\[
    \exp\left(- e^{-\alpha(z-1)}\right)
    \leq \liminf_{\rho\to\infty} \P{M^{W_\rho} \geq k(\rho) + z}
    \leq \limsup_{\rho\to\infty} \P{M^{W_\rho} \geq k(\rho) + z}
    \leq \exp\left(- e^{-\alpha z}\right).
\]
\end{conj}
To our knowledge, this behavior has not been yet observed for a graph arising from a tessellation.

%% file: inputs/bibliography.tex
\printbibliography